\pdfoutput=1
\documentclass[10pt,reqno]{amsart}
\newcommand{\mode}{0}
\calclayout
\usepackage[margin=15pt,labelfont=bf,textformat=period]{caption}

\usepackage{enumitem}

\usepackage{xcolor}
\definecolor{weblmcolor}{cmyk}{0.86,0.23,0.44,0.02}
\usepackage[pdfusetitle,colorlinks,allcolors=weblmcolor]{hyperref}

\usepackage[british]{babel}
\usepackage[utf8]{inputenc}

\usepackage{amssymb}
\usepackage{amsthm}
\usepackage{mathtools}

\usepackage[capitalise]{cleveref}

\numberwithin{equation}{section}

\newtheorem{theorem}{Theorem}[section]
\newtheorem{proposition}[theorem]{Proposition}
\newtheorem{lemma}[theorem]{Lemma}
\newtheorem{corollary}[theorem]{Corollary}
\theoremstyle{definition}
\newtheorem{remark}[theorem]{Remark}
\newtheorem{example}[theorem]{Example}
\newtheorem{definition}[theorem]{Definition}
\newtheorem{hypothesis}{Hypothesis}

\makeatletter
\newcommand{\sethypothesistag}[1]{
  \let\oldthehypothesis\thehypothesis
  \renewcommand{\thehypothesis}{#1}
  \g@addto@macro\endhypothesis{
    \addtocounter{hypothesis}{-1}
    \global\let\thehypothesis\oldthehypothesis}
}
\makeatother

\usepackage{tikz}
\usetikzlibrary{patterns,patterns.meta}
\tikzdeclarepattern{
  name=mydlines,
  parameters={
    \pgfkeysvalueof{/pgf/pattern keys/size},
    \pgfkeysvalueof{/pgf/pattern keys/angle},
    \pgfkeysvalueof{/pgf/pattern keys/line width},
  },
  bounding box={
    (0,-0.5*\pgfkeysvalueof{/pgf/pattern keys/line width}) and
    (\pgfkeysvalueof{/pgf/pattern keys/size},
    0.5*\pgfkeysvalueof{/pgf/pattern keys/line width})},
  tile size={(\pgfkeysvalueof{/pgf/pattern keys/size},
    \pgfkeysvalueof{/pgf/pattern keys/size})},
  tile transformation={rotate=\pgfkeysvalueof{/pgf/pattern keys/angle}},
  defaults={
    size/.initial=5pt,
    angle/.initial=45,
    line width/.initial=.4pt,
  },
  code={
    \draw [dashed,
           line width=\pgfkeysvalueof{/pgf/pattern keys/line width}]
    (0,0) -- (\pgfkeysvalueof{/pgf/pattern keys/size},0);
  },
}
\tikzdeclarepattern{
  name=myslines,
  parameters={
    \pgfkeysvalueof{/pgf/pattern keys/size},
    \pgfkeysvalueof{/pgf/pattern keys/angle},
    \pgfkeysvalueof{/pgf/pattern keys/line width},
  },
  bounding box={
    (0,-0.5*\pgfkeysvalueof{/pgf/pattern keys/line width}) and
    (\pgfkeysvalueof{/pgf/pattern keys/size},
    0.5*\pgfkeysvalueof{/pgf/pattern keys/line width})},
  tile size={(\pgfkeysvalueof{/pgf/pattern keys/size},
    \pgfkeysvalueof{/pgf/pattern keys/size})},
  tile transformation={rotate=\pgfkeysvalueof{/pgf/pattern keys/angle}},
  defaults={
    size/.initial=5pt,
    angle/.initial=45,
    line width/.initial=.4pt,
  },
  code={
    \draw [solid,
           line width=\pgfkeysvalueof{/pgf/pattern keys/line width}]
    (0,0) -- (\pgfkeysvalueof{/pgf/pattern keys/size},0);
  },
}
\usepackage{pgfplots}
 \pgfplotsset{
     tick align=outside,
     x grid style={white},
     xmajorgrids,
     y grid style={white},
     ymajorgrids,
     axis line style={white},
     axis background/.style={fill=white!92!black},
     legend style={draw=white, fill=white},
     legend cell align={left}
 }

\newcommand{\ee}{\mathrm{e}} 

\newcommand{\dd}{\, \mathrm{d}} 
\newcommand{\var}{\varepsilon}

\newcommand{\N}{{\mathbb N}}

\newcommand{\R}{{\mathbb R}}
\newcommand{\T}{{\mathbb T}}
\newcommand{\E}{{\mathbb E}}
\renewcommand{\S}{{\mathbb S}}

\newcommand{\sD}{\mathsf D}
\newcommand{\sV}{\mathsf V}
\newcommand{\sU}{\mathsf U}
\newcommand{\cR}{\mathcal R}
\newcommand{\cL}{\mathcal L}
\newcommand{\cT}{\mathcal T}
\newcommand{\cS}{\mathcal S}
\newcommand{\cA}{\mathcal A}
\newcommand{\cB}{\mathcal B}
\newcommand{\cC}{\mathcal C}
\newcommand{\cD}{\mathcal D}
\newcommand{\cG}{\mathcal G}
\newcommand{\cZ}{\mathcal Z}

\newcommand{\ip}[2]{\left\langle{#1},{#2}\right\rangle}
\newcommand{\vAvg}[1]{\langle{#1}\rangle} 
\newcommand{\gAvg}[1]{{\langle\!\langle}{#1}{\rangle\!\rangle}} 

\DeclareMathOperator{\range}{Range}
\DeclareMathOperator{\domain}{Domain}
\newcommand{\supp}{\operatorname{supp}} 
\newcommand{\indicator}{\mathbf{1}} 
\newcommand{\init}{\mathrm{in}} 

\newcommand{\sym}{\mathrm{\tiny sym}}
\newcommand{\anti}{\mathrm{\tiny anti}}

\newcommand{\strans}{\cS^{\textnormal{trans}}}
\newcommand{\sfull}{\cS^{\textnormal{full}}}

\newcommand{\hypref}[1]{\hyperref[#1]{\normalfont\textbf{(H\ref*{#1})}}}

\renewcommand{\vec}[1]{\mathbf{#1}}

\begin{document}

\ifnum\mode>0
  \begin{Frontmatter}
    \title[Quantitative Geometric Control in Linear Kinetic Theory]
    {Quantitative Geometric Control in Linear Kinetic Theory}
    \author[1]{Helge Dietert}\orcid{0000-0002-2071-9202}
    \author[2]{Frédéric Hérau}\orcid{0000-0001-9832-7105}
    \author[3]{Harsha Hutridurga}
    \author[4]{Clément Mouhot}\orcid{ORCID 0000-0002-4755-5404}
    \authormark{Helge Dietert \textit{et al}.}

    \address[1]{\orgname{Université Paris Cité and Sorbonne
        Université, CNRS}, \orgaddress{IMJ-PRG, F-75006 Paris,
        France}, \email{helge.dietert@imj-prg.fr}}

    \address[2]{\orgname{Laboratoire de Mathématiques Jean Leray,
        Nantes Université}, \orgaddress{2 rue de la Houssinière, BP
        92208 F-44322 Nantes Cedex 3, France},
      \email{frederic.herau@univ-nantes.fr}}

    \address[3]{\orgname{Department of Mathematics, Indian Institute
        of Technology Bombay}, \orgaddress{Powai, Mumbai 400076,
        India}, \email{hutri@math.iitb.ac.in}}

    \address[4]{\orgname{Department of Pure
        Mathematics and Mathematical Statistics, University of
        Cambridge}, \orgaddress{CMS, Wilberforce Road, CB3 0WA
        Cambridge, UK}, \email{cmouhot@dpmms.cam.ac.uk}}

    \received{** *** 2023}
    \abstract{  We consider general linear kinetic equations combining transport and a
      linear collision on the kinetic variable with a spatial weight that
      can vanish on part of the domain. The considered transport operators
      include external potential forces and boundary conditions, e.g.\
      specular, diffusive and Maxwell conditions. The considered collision
      operators include the linear relaxation (scattering) and the
      Fokker-Planck operators and the boundary conditions include
      specular, diffusive and Maxwell conditions. We prove quantitative
      estimates of exponential stabilisation (spectral gap) under a
      \emph{geometric control condition}. The argument is new and relies
      entirely on trajectories and weighted functional inequalities on the
      divergence operators. The latter functional inequalities are of
      independent interest and imply quantitatively weighted Stokes and
      Korn inequalities.  We finally show that uniform control conditions
      are not always necessary for the existence of a spectral gap when
      the equation is hypoelliptic, and prove weaker control conditions in
      this case.}
  \end{Frontmatter}
  \localtableofcontents
\else
  \title{Quantitative Geometric Control in Linear Kinetic Theory}

  \author[Dietert]{Helge Dietert} \address{Université Paris Cité and
    Sorbonne Université, CNRS\\ IMJ-PRG, F-75006 Paris, France.}
  \email{helge.dietert@imj-prg.fr}

  \author[Hérau]{Frédéric Hérau} \address{Laboratoire de Mathématiques
    Jean Leray, Nantes Université\\ 2 rue de la Houssinière, BP 92208
    F-44322 Nantes Cedex 3, France}
  \email{frederic.herau@univ-nantes.fr}

  \author[Hutridurga]{Harsha Hutridurga} \address{Department of
    Mathematics, Indian Institute of Technology Bombay, Powai,
    Mumbai 400076, India.}  \email{hutri@math.iitb.ac.in}

  \author[Mouhot]{Clément Mouhot} \address{Department of Pure
    Mathematics and Mathematical Statistics, University of
    Cambridge, Wilberforce Road, CB3 0WA Cambridge, UK}
  \email{c.mouhot@dpmms.cam.ac.uk}

  \begin{abstract}
    We consider general linear kinetic equations combining transport and a
    linear collision on the kinetic variable with a spatial weight that
    can vanish on part of the domain. The considered transport operators
    include external potential forces and boundary conditions, e.g.\
    specular, diffusive and Maxwell conditions. The considered collision
    operators include the linear relaxation (scattering) and the
    Fokker-Planck operators and the boundary conditions include
    specular, diffusive and Maxwell conditions. We prove quantitative
    estimates of exponential stabilisation (spectral gap) under a
    \emph{geometric control condition}. The argument is new and relies
    entirely on trajectories and weighted functional inequalities on the
    divergence operators. The latter functional inequalities are of
    independent interest and imply quantitatively weighted Stokes and
    Korn inequalities.  We finally show that uniform control conditions
    are not always necessary for the existence of a spectral gap when
    the equation is hypoelliptic, and prove weaker control conditions in
    this case.
  \end{abstract}
  \date{\today}
  \maketitle
  \tableofcontents
\fi

\section{Introduction}

\subsection{Summary}

This manuscript is part of a novel recent development in kinetic
theory, namely the control theory of hypocoercive structures (with or
without hypoellipticity). We consider a general class of linear
kinetic equations as prototypes of such structures. They combine a
degenerate relaxation (in $v$) with a first-order transport dynamics,
and with the addition of a further degeneracy of the relaxation along
the spatial variable: this corresponds to a \emph{thermalisation
  degeneracy}. A natural question in control theory, that of
\emph{exponential stabilisation}, is to determine conditions upon the
thermalising region under which a spectral gap can be obtained.

Our main contribution is to provide a quantitative exponential
stability estimate under a general control condition, and for a large
class of operators, boundary conditions and potential
confinements. Moreover, we obtain quantitative exponential stability
for a hypoelliptic equation in a critical case where standard control
conditions fail, by exploiting the hypoellipticity. Our results are
new for most concrete equations considered, apart from the linear
relaxation without boundary conditions. The general approach is also
new.

We first present our ideas in an abstract result, see the assumptions
in \cref{sec:abstract-assumptions} and the abstract
Theorem~\ref{theo:main}. Before the abstract discussion we present
informally four key examples in \cref{sec:intro:examples}. The
application to some concrete models is proved in \cref{cor:LB}. The
key new feature of our method is to follow trajectories and to use the
divergence inequality (i.e.\ constructing a non-unique inverse to the
divergence in $H^1_0$ with quantitative estimate). Even without
thermalisation degeneracy it leads to novel proofs of hypocoercivity
and covers cases with boundaries and confining potentials that were
not known before. The required divergence inequalities (see
\cref{theo:ineq}) are extensions (with weight and boundary conditions)
of key results of Bogovskii, Bourgain and Brézis. They are proved in
\cref{sec:divergence} and have interest per se.

\subsection{The problem at hand}

We consider the abstract kinetic equation
\begin{align}
  \label{eq:gen}
  \begin{dcases}
    \partial_t f + \cT f = \sigma \cL f
    &\text{in } \Omega \times \sV, \\
    \tilde \gamma_- f = \cR \gamma_+ f &\text{in }
    \partial \Omega \times \sV \text{ with } \vec{n} \cdot v \ge 0,
  \end{dcases}
\end{align}
for a time-dependent probability density $f=f(t,x,v) \ge 0$ on the
phase space $(x,v) \in \Omega \times \sV$, where $\Omega$ is a smooth
open set of $\R^d$ or $\Omega=\T^d$ and $\sV \subset \R^d$. The
transport operator
$\cT := v \cdot \nabla_x - \nabla_x \phi \cdot \nabla_v$ includes an
external potential $\phi : \R^d \to \R$; the linear dissipation
operator $\cL$ acts on $v$ only and models the thermalisation effect
due to collisions with a background medium; and the thermalisation
degeneracy function $\sigma=\sigma(x) \ge 0$ can vanish on part of
$\Omega$.

If $\Omega$ has a boundary, $\partial \Omega$ is assumed to be smooth
and we denote $\vec{n} : \partial \Omega \to \S^{d-1}$ the unit
outgoing normal vector and
$\Gamma_\pm := \{(x,v) \in \partial \Omega \times \sV \ : \pm
(\vec{n} \cdot v) \ge 0\}$. We then denote $\gamma_{\pm} \varphi$ the
trace of $\varphi$ at $\pm (\vec{n} \cdot v) \ge 0$, and
$\tilde \gamma_- \varphi(v) := \varphi (v-2(\vec{n}\cdot v)\vec{n})$
on $\Gamma_+$. Then the boundary conditions in~\eqref{eq:gen} are
defined by the operator (following the spirit of the framework in~\cite{MR432101})
\begin{align}
  \label{eq:boundary}
  (\cR \gamma_+ f)(x,v) := \int_{(\vec{n} \cdot v_*) \ge 0}
  \left( \gamma_+ f \right)(x,v_*)\, R(x,v,v_*)\,
  (\vec{n} \cdot v_*) \dd v_*,
\end{align}
for $(x,v) \in \Gamma_+$, with a measurable kernel
$R=R(x,v,v_*) \ge 0$ preserving mass:
\begin{equation}
  \label{eq:mass-conservation-bdd}
  \forall \, (x,v_*) \in \partial \Omega \times \sV \text{ with }
  (\vec{n} \cdot v_*) \ge 0, \quad
  \int_{(\vec{n} \cdot v) \ge 0} R(x,v,v_*)\,
  (\vec{n} \cdot v) \dd v \equiv 1.
\end{equation}

The question addressed in this paper is the relaxation to
equilibrium $f(t,x,v) \to f_\infty(x,v)$ for large time, and more
specifically whether the linear
evolution~\eqref{eq:gen}-\eqref{eq:boundary} has a spectral
gap.

\subsection{Previous results and contribution}

When $\sigma$ vanish on part of $\Omega$,
\cite{bernard-salvarani-2013-degenerate-linear-boltzmann} proved
exponential relaxation by non-constructive methods when $\cL$ is the
linear Boltzmann operator, $\Omega = \T^d$ is the flat torus,
$\phi=0$, $\sV$ is bounded, and under a \emph{geometric control
  condition} on $\sigma$ inspired from wave
equations~\cite{MR1178650}. The same authors obtained lower bound on
the rate of convergence for similar models in~\cite{MR3048598} when
the geometric control condition
fails. Later~\cite{han-kwan-leautaud-2015-geometric} extended this
non-constructive result to $\phi \not =0$ and unbounded velocities
$\sV = \R^d$. Finally the recent
preprint~\cite{evans-moyano-2019-preprint-quantitative-rates-convergence-equilibrium-degenerate}
gave the first constructive proof, in the same setting, by a
probabilistic approach based on Döblin's theorem. When there is no
thermalisation degeneracy ($\sigma=1$) but $\Omega$ has a boundary,
the paper~\cite{MR2679358} proved exponential relaxation in $L^2$ for
the related linearised Boltzmann equation with specular boundary
conditions by non-constructive compactness arguments, and the
paper~\cite{MR3562318} proved it for the linearised Boltzmann equation
with diffusive boundary conditions by constructive methods. The more
recent
preprint~\cite{bernou-carrapatoso-mischler-tristani-2021-preprint-hypocoercivity-bounded-domain}
finally proved it for specular, diffusive and Maxwell boundary
conditions by extending the constructive method
of~\cite{dolbeault-mouhot-schmeiser-2015-hypocoercivity} to
initial-boundary-value problems, which includes solving specific
difficulties.

Our method here recovers all these previous results (and others) in a
quantitative manner. We propose a novel quantitative approach based on
\emph{trajectories} which, apart from unifying and simplifying
previous works, gives the first quantitative estimates when a
thermalisation degeneracy is combined with a diffusive operator in
velocity, or when a thermalisation degeneracy is combined with
boundary conditions. In the terminology of control theory, we prove
quantitative and unconditional \emph{exponential stabilisation} for
linear kinetic equations and our estimates readily imply the
\emph{unique continuation property} for the set
$\omega:=\supp \sigma$. When $\cL$ is bounded, they also
straightforwardly imply the quantitative \emph{observability} of this
set (see in particular~\eqref{eq:following}); the estimates
established here are however stronger than the latter property. We
also initiate the study of how geometric control conditions and
hypoellipticity interplay (see \cref{ex:hypoelliptic-decay}).

\subsection{Setting and concrete examples}
\label{sec:intro:examples}

The main theorem handles the abstract
equations~\eqref{eq:gen}-\eqref{eq:boundary} in a general
form. Before stating the precise abstract conditions, we discuss the
following key examples in kinetic theory to motivate the general
theory and help the reader gain intuition.

Take for the velocity either \(v \in \sV := \R^d\) with the
equilibrium measure $M(v) = (2\pi)^{-d/2} \ee^{-|v|^2/2}$ or
\(v \in \sV := \S^{d-1}\) with the equilibrium measure
$M(v)=|\S^{d-1}|^{-1}$. Consider then the collision operator to be the
linear Boltzmann operator (also called scattering operator)
\begin{equation}
  \label{eq:LB}
  \cL f(v) := \int_{\sV} \Big[ k(v,v_*) f(v_*) -
  f(v) k(v_*,v) \Big] \dd v_*
\end{equation}
with $0 \le k \in C^0(\sV^2)$ so that $M$ is the only invariant
measure and $\cL$ has a spectral gap in \(L^2(M^{-1})\)
(see~\cref{ss:local-coerc-concrete} for a discussion of
sufficient conditions) or the linear Fokker-Planck operator
\begin{align}
  \label{eq:FP}
  \cL f(v) :=
  \begin{dcases}
    \Delta_{\mathrm{LB}} f
    &\text{if } \sV = \S^{d-1}\\
    \nabla_v \cdot \left( \nabla_v f + v f \right)
    &\text{if } \sV =\R^d
  \end{dcases}
\end{align}
where $\Delta_{\mathrm{LB}}$ denotes the \emph{Laplace-Beltrami
  operator} on $\S^{d-1}$.

We now discuss four paradigmatic examples of increasing
complexity:
\begin{example}[Periodic confinement]\label{ex:decay-torus}
  Assume that $\Omega = \T^d$ and let
  \[
    (x_0,v_0) \mapsto (X_t(x_0,v_0),V_t (x_0,v_0))
  \]
  be the characteristic map starting from $(x_0,v_0)$
  associated to the transport $\cT$. Suppose there exists a good set
  $\Sigma$ so that $\sigma \gtrsim \indicator_{\Sigma}$ and
  $\chi \in C^\infty(\Omega)$, $c,T>0$ with
  $\supp \chi \subset \Sigma$ so that
  \begin{equation}\label{eq:ex:decay-torus:gcc}
    \forall(x_0,v_0) \in \Omega \times \sV,\quad
    \int_0^T \chi(X_t(x_0,v_0)) \dd t \ge c.
  \end{equation}
  Then our method yields exponential convergence with quantitative
  estimate.
\end{example}

In other words, we prove exponential convergence if there exists a set
$\Sigma$ in which thermalisation occurs and which is exposed to all
configuration by the transport flow, i.e.\ for any initial point the
trajectory is thermalised a strictly positive amount in the set
\(\Sigma\) over some time interval \([0,T]\) (heuristically we can
think of \(\chi\) as \(\indicator_{\Sigma}\), omitting technical
regularity requirements on $\chi$). For (nearly) massless particles
like neutrons, the velocities can be modelled by $\sV = \S^{d-1}$ and
without external potential ($\phi \equiv 0$) the
condition~\eqref{eq:ex:decay-torus:gcc} reduces to whether straight
lines hit the good set $\Sigma$, see \cref{fig:torus-gcc}.

\begin{figure}[htb]
  \begin{center}
    \begin{tikzpicture}[scale=1.4]
      \draw[fill, color=black!10!white] (0,0) -- (3,0) -- (3,3) -- (0,3) -- cycle;
      \draw[green, pattern=north west lines, pattern color=green]
      (0,1) -- (0,2) -- (1,2) -- (1,3) -- (2,3) -- (2,2) -- (3,2) --
      (3,1) -- (2,1) -- (2,0) -- (1,0) -- (1,1) -- cycle;
      \draw[very thick] (0,0) -- (3,0) -- (3,3) -- (0,3) -- cycle;
      \draw (3,3) node[anchor=north west] {$\Omega$};
      \draw (1.5,1.5) node {$\sigma \equiv 1$};
      \draw (0.5,0.5) node {$\sigma \equiv 0$};
      \draw[blue,very thick,->] (0.5,2.3) -> (2.5,2.3);
    \end{tikzpicture}
    \hspace{3cm}
    \begin{tikzpicture}[scale=1.4]
      \draw[fill, color=black!10!white] (0,0) -- (3,0) -- (3,3) -- (0,3) -- cycle;
      \draw[green, pattern=north west lines, pattern color=green]
      (0,1) -- (0,2) -- (3,2) -- (3,1) -- cycle;
      \draw[very thick] (0,0) -- (3,0) -- (3,3) -- (0,3) -- cycle;
      \draw (3,3) node[anchor=north west] {$\Omega$};
      \draw (1.5,1.5) node {$\sigma \equiv 1$};
      \draw (0.5,0.5) node {$\sigma \equiv 0$};
      \draw[blue,very thick,->] (0.5,2.3) -> (2.5,2.3);
    \end{tikzpicture}
  \end{center}
  \caption{Illustration of \eqref{eq:ex:decay-torus:gcc} for a
    periodic domain with $\sV = \S^{d-1}$ (massless particles) and no
    external potential $\phi\equiv 0$. On the left, all lines hit and
    spend a controlled fraction of time in the thermalisation set
    $\Sigma = \supp \sigma$ on a time interval $[0,T]$, yielding
    exponential convergence. On the right, there is a set of
    configurations with zero measure whose trajectories never hit
    $\Sigma$, and around this set the time to hit $\Sigma$ can be
    arbitrarily large: the control condition is not satisfied, and one
    typically expects polynomial rate of convergence}
  \label{fig:torus-gcc}
\end{figure}
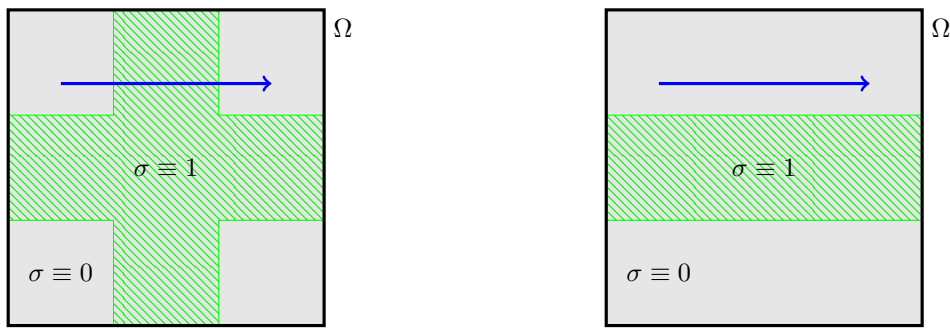

\begin{example}[Boundaries and potential]
  \label{ex:boundaries-potential-gamma2}
  In order to cover boundaries, we face two problems: (i) in the case
  of non-convex geometries the evolution may create discontinuities
  and (ii) with diffusive boundary conditions the transport flow
  $\strans_t$ associated with the transport operator $\cT$ and the
  boundary conditions $\cR$ ceases to admit deterministic
  characteristic trajectories. To overcome these issues, we
  rewrite~\eqref{eq:ex:decay-torus:gcc} as
  \begin{equation}
    \label{eq:gcc-trans}
    \forall(x_0,v_0) \in \Omega \times \sV,\quad
    \int_0^T \int_{\Omega \times \sV}
    (\strans_t\delta_{(x_0,v_0)})(x,v)\,
    \chi(x,v) \dd x \dd v \dd t \ge c >0
  \end{equation}
  and our method applies, once characteristic flow are adequately
  treated more abstractly, as long as the operator $\cL$ is bounded or
  the transport flow propagates enough regularity.

  To overcome the regularity issue when the latter is not satisfied,
  we introduce the full semigroup $\sfull_t$ associated to
  \eqref{eq:gen}. In the stochastic interpretation of the collision
  operator, we have a positive probability that the particle has
  undergone no collision (in the linear Boltzmann
  setting~\eqref{eq:LB}) or the change of velocity has been arbitrary
  small (in the Fokker-Planck setting~\eqref{eq:FP}). Hence, one
  expects in general that \eqref{eq:gcc-trans} implies the same
  condition with $\strans_t$ replaced by $\sfull_t$. Using $\sfull_t$,
  we can replace the regularity condition by a $\Gamma$ condition on
  the collision operator which we can justify for the Fokker-Planck
  operator and all linear Boltzmann operators satisfying
  reversibility.

  As for the boundary conditions, we impose a compatibility condition
  which includes all standard cases. In particular, it holds for the
  specular, diffusive and Maxwell boundary conditions, whose kernel
  in~\eqref{eq:boundary} is
  \begin{align}
    \label{eq:boundary-concrete}
    R(x,v,v_*) :=
    (1-\alpha(x)) \frac{\delta_{v=v_*}}{(\vec{n} \cdot v_*)}
    + \alpha(x) \sqrt{2\pi} M(v)
  \end{align}
  with an \emph{accommodation coefficient}
  $\alpha : \partial \Omega \to [0,1]$. Specular boundary
  conditions correspond to $\alpha=0$, while diffusive boundary
  conditions correspond to $\alpha=1$.

  Then \eqref{eq:gen} converges exponentially if there exists a set
  $\Sigma \subset \Omega$ with $\sigma \gtrsim \indicator_{\Sigma}$
  such that a weighted Poincaré inequality with potential $\phi$ holds
  on $\Sigma$ and the set \(\Sigma\) is exposed to all initial data in
  the following sense: there exists $\chi \in L^\infty(\Omega)$ with
  $\supp \chi \subset \Sigma$ such that
  \begin{equation}\label{eq:gcc-full}
    \forall(x_0,v_0) \in \Omega \times \sV,\quad
    \int_0^T \int_{\Omega \times \sV}
    (\sfull_t\delta_{(x_0,v_0)})(x,v)\,
    \chi(x,v) \dd x \dd v \dd t \ge c.
  \end{equation}

  As an illustration, consider the spherical vessel $\Omega = B(0,1)$
  in dimension~$3$ with diffusive boundary conditions and
  $\sigma = \indicator_{\bar B(0,1/2)}$
  in~\cref{fig:spherical-vessel}. In such case, it is proved
  in~\cite{MR2765738} (with optimal polynomial rate) that the
  transport flow with diffusive boundary conditions starting from
  initial data in $L^\infty(\Omega \times \sV ; \dd x \dd \mu)$
  relaxes polynomially towards equilibrium in
  $L^2(\Omega \times \sV ; \dd x \dd \mu)$. Our method then yields
  exponential convergence in spite of the fact that the uniform
  geometric control condition would fail for the specular transport
  flow; in other words our theorem genuinely uses the dissipativity at
  the boundary.

  \begin{figure}[htb]
    \begin{center}
    \begin{tikzpicture}[scale=1.2]
      \draw[fill, color=black!10!white] (0,0) circle (2);
      \draw[very thick] (0,0) circle (2);
      \draw[green, pattern=north west lines, pattern color=green] (0,0)
      circle (0.8);
      \draw (0,0) node {$\sigma \equiv 1$};
      \draw (0,1.4) node {$\sigma \equiv 0$};
      \draw (1.5,1.5) node[anchor=south west] {$\Omega$};
      \draw[blue,very thick,->] (-60:2) -- (0:2);
      \draw[blue,very thick,dotted,->] (0:2) -- ++ (120:1);
    \end{tikzpicture}
    \hspace{3cm}
    \begin{tikzpicture}[scale=1.2]
      \draw[fill, color=black!10!white] (0,0) circle (2);
      \draw[very thick] (0,0) circle (2);
      \draw[green, pattern=north west lines, pattern color=green] (0,0)
      circle (0.8);
      \draw (0,0) node {$\sigma \equiv 1$};
      \draw (0,1.4) node {$\sigma \equiv 0$};
      \draw (1.5,1.5) node[anchor=south west] {$\Omega$};
      \draw[blue,very thick,->] (-60:2) -- (0:2);
      \draw[blue,very thick,dotted,->] (0:2) -- ++ (180:1);
      \draw[blue,very thick,dotted,->] (0:2) -- ++ (130:1);
      \draw[blue,very thick,dotted,->] (0:2) -- ++ (230:1);
    \end{tikzpicture}
  \end{center}
  \caption{Illustration of a spherical vessel domain with degenerate
    diffusion weight. On the left we have specular boundary conditions
    (no exponential convergence) and on the right we have diffusive
    boundary conditions (exponential convergence)}
  \label{fig:spherical-vessel}
\end{figure}
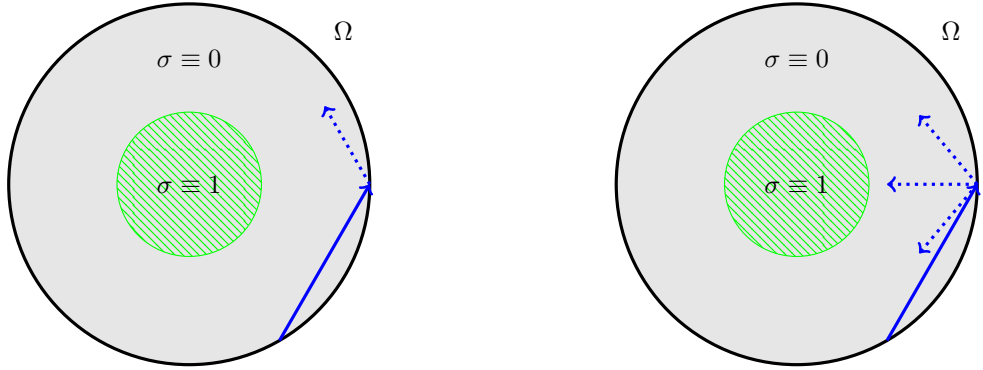
\end{example}

\begin{remark}
  Note that although we only require the Poincaré inequality on the
  potentially smaller set $\Sigma$, and not $\Omega$, the condition
  \eqref{eq:gcc-full} requires that the whole domain $\Omega$ be
  connected in finite time to $\Sigma$.
\end{remark}

\begin{example}[Harmonic potential with thermalisation only around
  zero]\label{ex:weight-gain}
  Consider the harmonic potential \(\phi(x)=x^2/2\) in one dimension
  \(\Omega = \R\) and a thermal degeneracy given by
  \(\sigma = \indicator_{|x| \le 1}\). The transport control
  condition~\eqref{eq:ex:decay-torus:gcc} then fails for $(x_0,v_0)$
  with $r_0 = \sqrt{x_0^2+v_0^2}$ large since the time the
  trajectories spend in $\{|x| \le 1\}$ is proportional to
  $r_0^{-1}$. Since the speed with which such bad trajectories
  cross the good region $\supp \sigma$ is high, one can however
  recover exponential convergence if the local coercivity estimate of
  the collision operator gains a weight $|v|$ at large velocities, see
  \cref{fig:gaining-weight}. For instance~\eqref{eq:gen} converges
  exponentially when $\cL$ is given Fokker-Planck operator over
  \(\sV = \R\).
  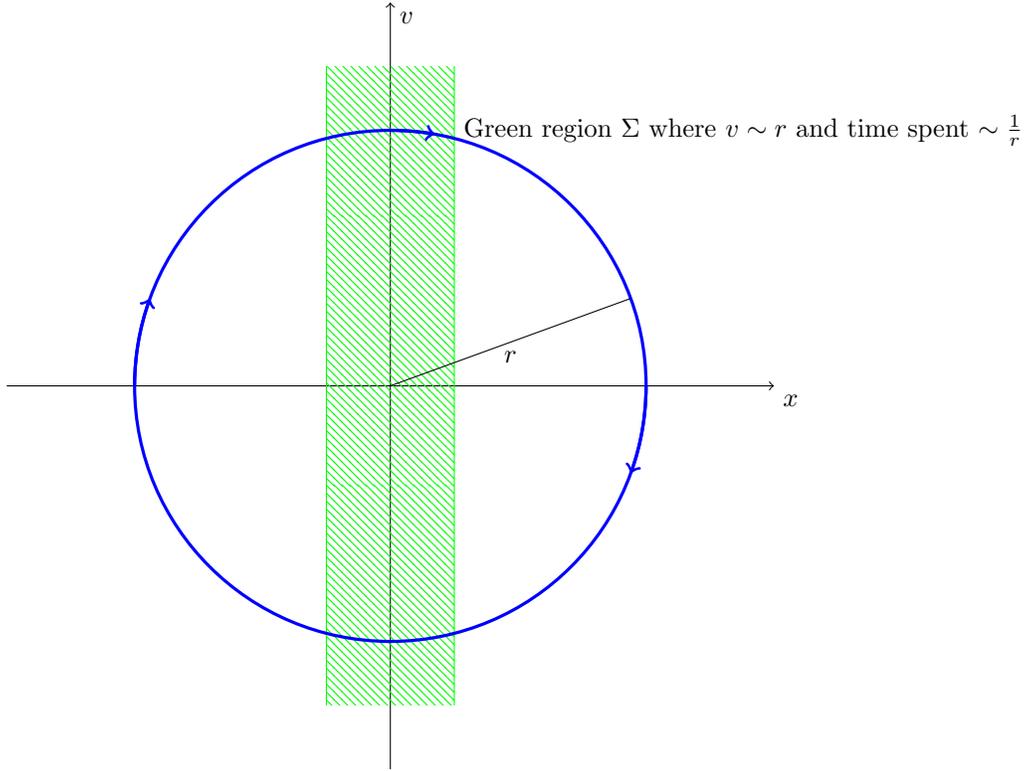
\begin{figure}
    \centering
    \begin{tikzpicture}[scale=1.7]
      \draw[->] (-3,0) -- (3,0) node[anchor=north west] {$x$};
      \draw[->] (0,-3) -- (0,3) node[anchor=north west] {$v$};
      \fill[green, pattern=north west lines, pattern color=green]
      (-0.5,-2.5) -- (0.5,-2.5) -- (0.5,2.5) -- (-0.5,2.5) -- cycle;
      \draw[green] (0.5,-2.5) -- (0.5,2.5);
      \draw[green] (-0.5,-2.5) -- (-0.5,2.5);
      \draw[very thick,blue,->] (0,0) circle (2);
      \draw[very thick,blue,->] (0:2) arc(0:-20:2);
      \draw[very thick,blue,->] (180:2) arc(-180:-200:2);
      \draw[very thick,blue,->] (-270:2) arc(-270:-280:2);
      \draw (0,0) -- node[anchor=north] {$r$} (20:2);
      \draw (0.5,2) node[anchor=west] {Green region $\Sigma$ where $v \sim r$ and time spent $\sim \frac 1 r$};
    \end{tikzpicture}
    \caption{Phase space in a harmonic potential, where the
      trajectories are circles. For $\sigma = \indicator_{|x| \le 1}$,
      the evolution decays exponentially if the collision gains the
      weight \((1+|v|)\)}
    \label{fig:gaining-weight}
  \end{figure}
\end{example}

\begin{example}[Hypoellipticity with control condition failing at a
  point]\label{ex:hypoelliptic-decay}
  In the case of the linear Boltzmann equation with deterministic
  transport flow, the geometric control condition~\eqref{eq:gcc-trans}
  is not only sufficient but \emph{necessary} to the exponential
  relaxation, see~\cite{han-kwan-leautaud-2015-geometric}: this
  follows from considering initial data whose supports concentrate on
  trajectories where the uniform geometric control
  condition~\eqref{eq:gcc-trans} fails. However, for regularising
  collision operators like the Fokker-Planck operator~\eqref{eq:FP},
  it can be relaxed to some extent, as first noticed
  in~\cite{dietert-thesis-2017}. The argument
  from~\cite{han-kwan-leautaud-2015-geometric} then fails because of
  \emph{hypoellipticity}. Let us consider the following prototypical
  example, described in \cref{fig:hypo}.

  \begin{figure}
    \centering
    \begin{tikzpicture}[scale=1.7]
      \draw[->] (-3,0) -- (3,0) node[anchor=north west] {$x$};
      \draw[->] (0,-3) -- (0,3) node[anchor=north west] {$v$};
      \draw[very thick,violet,->] (0,0) circle (1/2);
      \draw[very thick,violet,->] (0:1/2) arc(0:-20:1/2);
      \draw[very thick,violet,->] (180:1/2) arc(-180:-200:1/2);
      \draw[very thick,violet,->] (-270:1/2) arc(-270:-280:1/2);
      \draw[very thick,blue,->] (0,0) circle (1.3);
      \draw[very thick,blue,->] (0:1.3) arc(0:-20:1.3);
      \draw[very thick,blue,->] (180:1.3) arc(-180:-200:1.3);
      \draw[very thick,blue,->] (-270:1.3) arc(-270:-280:1.3);
      \draw[very thick,cyan,->] (0,0) circle (2);
      \draw[very thick,cyan,->] (0:2) arc(0:-20:2);
      \draw[very thick,cyan,->] (180:2) arc(-180:-200:2);
      \draw[very thick,cyan,->] (-270:2) arc(-270:-280:2);
      \draw[red] (0,0) -- node[anchor=north] {$r$} (20:2);
      \fill[domain=-3:3,samples=50,pattern=north west lines, pattern color=green]
      plot ({\x},{0.1*\x*\x-4})
      -- (3,-4) -- (-3,-4) -- cycle;
      \draw[green, very thick,domain=-3:3]
      plot ({\x},{0.1*\x*\x-4});
      \draw (-3,-4) -- (3,-4);
      \draw (3,-4+0.1*9) node[anchor=south] {Thermalisation
        \(\sigma\) depending on \(x\)};
      \draw (0.5,2) node[anchor=west] {Circular trajectories with $2\pi$ period};
      \draw (1.9,0.8) node[anchor=west] {Thermalisation accross a};
      \draw (2.2,0.5) node[anchor=west] {period decreases as $r^2$};
    \end{tikzpicture}
    \caption{The prototypical example of hypoellipticity overcoming
      some level of degeneracy in the control condition}
    \label{fig:hypo}
  \end{figure}
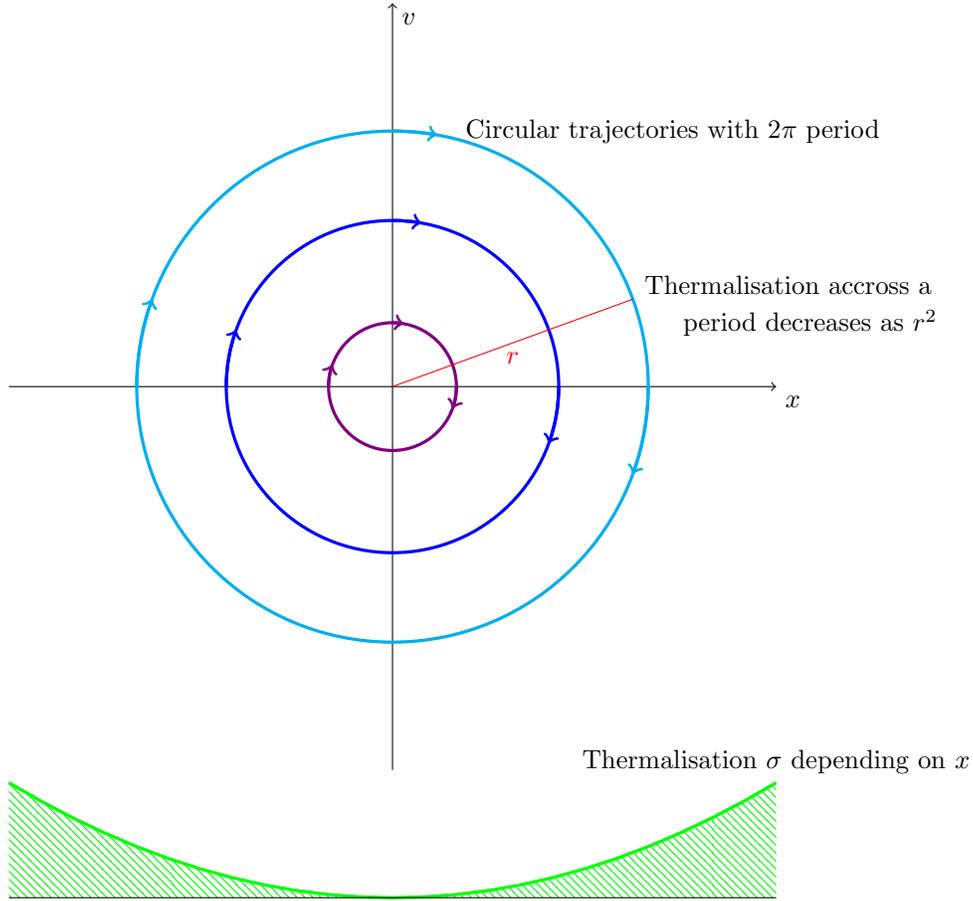

  Consider \(\Omega=\sV=\R\) in dimension $d=1$, with \(\cL\) given by
  the Fokker-Planck operator~\eqref{eq:FP}, with the harmonic
  potential \(\phi(x)=x^2/2\) and with the thermalisation degeneracy
  \(\sigma(x) = \min(x^2,1)\) around \(x=0\). Around $x=0$, the
  spatial degeneracy $\sigma(x)$ behaves as $x^2$ and the
  characteristics describe circle with period $2 \pi$ around
  $(x,v) =(0,0)$, so that the geometric control
  conditions~\eqref{eq:gcc-trans} and~\eqref{eq:gcc-full} both fail
  around $(0,0)$. However, thermalisation is effective apart from the
  latter point, and the stationary probability measure
  $\ee^{-\phi(x)} M(v)$ is still unique. We then prove the existence
  of a spectral gap thanks to \emph{corrector} $\mathfrak C$ in the
  control condition, and~\eqref{eq:gen} relaxes exponentially.
\end{example}

\subsection{Abstract assumptions}
\label{sec:abstract-assumptions}

\begin{hypothesis}[Geometric constraints]
  \label{h:geom}
  The domain $\Omega$ is either a smooth and open subset of $\R^d$, or
  $\Omega = \T^d$, and the velocity space $\sV \subset \R^d$ is even
  ($-\sV=\sV$) and spans $\R^d$. The thermalisation degeneracy
  function $\sigma$ is non-negative and bounded, and the potential
  $\phi \in C^2(\Omega)$ with $\int_\Omega \ee^{-\phi} =1$.
\end{hypothesis}

\begin{remark}
  There are no further assumptions on the velocity space $\sV$. In
  particular, a discrete velocity space $\sV$ is possible, as used in
  lattice Boltzmann methods and in simplified models. The assumption
  that $\sV$ is even is only necessary for handling boundary
  data. Otherwise, it suffices to have two bases of $\R^d$ in $\sV$
  which differ by at least one element, in order to be able to
  construct the functions $\varphi_0,\varphi_1,\dots,\varphi_d$ in
  \cref{subsec:am}. Note also that the boundedness of $\sigma$ on
  $\Omega$ can be relaxed to boundedness merely on the control set
  $\Sigma$ introduced below in \hypref{h:tcc}.
\end{remark}

\begin{hypothesis}[Equilibria and semigroup]
  \label{h:eq}
  There is a positive $M \in L^1(\sV)$ with mass $1$ so that
  $f_\infty(x,v) := \ee^{-\phi(x)} M(v)$ achieves equilibrium:
  $\cL f_\infty=\cT f_\infty=0$ and
  $\cR \gamma_+ f_\infty= \gamma_+ f_\infty$ on $\Gamma_+$.
\end{hypothesis}

There are no further assumptions on the profile $M$, in
particular non-Gaussian equilibria are possible as long as the
equilibrium condition is satisfied separately by $\cT$ and
$\cL$.

\begin{hypothesis}[Microscopic coercivity]
  \label{h:local}
  $\cL$ and $\cL^*$ are real closable operator with dense domains
  including $C_0^\infty(\sV)$ in $L^2_v(M^{-1})$, with
  $\range(\cL) \, \bot \, 1$ (preservation of mass), and so that
  $\cL + \cL^*$ has a spectral gap in the velocity space with a weight
  $w \in L^1_v(M)$ satisfying \(w \ge 1\), i.e.\ there is
  $\lambda_1 >0$ such that for any real-valued $g \in \domain(\cL)$
  \begin{align}
    \label{eq:loc-coerc}
    \int_{\sV} g(v) \cL g(v) \, \frac{\mathrm{d}v}{M(v)}\le - \lambda_1
    \int_{\sV} \left[ g(v) - \left( \int_{\sV} g(v) \dd
    v \right) M(v) \right]^2 \, \frac{w(v)\, \mathrm{d}v}{M(v)}.
  \end{align}
\end{hypothesis}

\begin{remark}
  Whether the previous spectral gap condition holds for the operator
  $\cL$ is a standard question discussed in
  \cref{ss:local-coerc-concrete}. Note that $\cL$ and $w$ could depend
  on $x$, provided conditions and estimates on them are made uniform
  in $x$, i.e.\ degeneracy is entirely captured by $\sigma$. The
  theory could also cover cases with a degenerate spectral gap, i.e.\
  \(w\) goes to zero towards infinity, which is then compensated by a
  suitable growth in \(\sigma\).
\end{remark}

\begin{hypothesis}[Boundary compatibility]
  \label{h:bdd}
  The operator $\cR$ is real and bounded in
  $L^2(\Gamma_+, \mathrm{d}\nu)$ and is contractive.  There is $C_r>0$
  such that for any $g \in L^2(\Gamma_+, \mathrm{d} \nu)$ and
  $\varphi \in L^\infty(\Gamma_+)$ it holds that
  \begin{equation}
    \label{eq:boundary-bound}
    \int_{\Gamma_+} \varphi \left[
      f_\infty \cR\left(f_\infty^{-1} g^2 \right)
       - \left( \cR g \right)^2 \right] \dd \nu \le C_r \,
    \| \varphi \|_{L^\infty(\Gamma_+)}
    \int_{\Gamma_+} \left[
      g^2 - \left( \cR g\right)^2 \right] \dd \nu
  \end{equation}
  where we denote by $\mathrm{d} S$ the surface measure of
  $\partial \Omega$ and
  ${\rm d} \nu := (\vec{n} \cdot v) f_{\infty}^{-1}(x,v) \dd S \dd v$
  on $\Gamma_+$.
\end{hypothesis}

\begin{remark}
  The condition~\eqref{eq:boundary-bound} means that the boundary
  terms coming from estimates of the squared solution with $L^\infty$
  test function can be controlled by the entropy production at the
  boundary. It is satisfied in all standard cases, as discussed in
  \cref{sec:boundary-compatibility}, and in particular \(\cR\)
  correspond to Maxwell boundary
  conditions~\eqref{eq:boundary-concrete} with an arbitrary
  accomodation coefficient \(\alpha : \partial \Omega \to
  [0,1]\). Note also that the left hand side
  of~\eqref{eq:boundary-bound} rewrites
  $\int_{\Gamma_+} [ \cR^T(\varphi) g^2 - \varphi \cR(g)^2 ] \dd \nu$
  with $\cR^T \varphi := f_\infty^{-1} \cR^*(f_\infty \varphi)$ (the
  $L^2$ adjoint), and that with such reformulation, $\cR^T$ could be
  replaced, in this assumption and the next one, by another operator
  $\cR'$ provided it satisfies, for any
  $g \in L^2(\Gamma_+, \mathrm{d} \nu)$ and
  $\varphi \in L^\infty(\Gamma_+)$,
  \begin{equation}
    \label{eq:boundary-bound'}
    \begin{dcases}
      \int_{\Gamma_+} \left[
      \cR'(\varphi)g^2
       - \varphi \left( \cR g \right)^2 \right] \dd \nu \le C_r \,
    \| \varphi \|_{L^\infty(\Gamma_+)}
    \int_{\Gamma_+} \left[
      g^2 - \left( \cR g\right)^2 \right] \dd \nu \\
      \left| \int_{\Gamma_+} \left[ \cR'(\varphi) f_\infty g - \varphi
          f_\infty \cR g \right] \dd \nu \right| \le C_r \, \| \varphi
      \|_{L^\infty(\Gamma_+)} \left(\int_{\Gamma_+} \left[ g^2 -
          \left( \cR g\right)^2 \right] \dd \nu \right)^{\frac12}.
  \end{dcases}
\end{equation}
\end{remark}

\begin{hypothesis}[Transport control condition -- abstract]
  \label{h:tcc}
  Let $\cB$ be either zero or \(\sigma \cL^T\) where
  $\sigma \cL^T \eta :=\sigma f_\infty^{-1} \cL^*(f_\infty \eta)$.
  There is a connected domain $\Sigma \subset \Omega$ so that
  $\inf_{x \in \Sigma} \sigma(x) > 0$, and there is
  $0 \le \chi \in L^\infty(\Omega \times \sV)$ with
  $\supp \chi \subset \Sigma \times \sV$ and an operator $\mathfrak C$
  in $v$ so that the solution \(\varphi\) to the evolution problem
  \begin{align}
    \label{eq:tcc-decay}
    \begin{dcases}
      \partial_t \varphi - \cT \varphi -\cB \varphi = - \mathfrak C \varphi - \chi w \varphi
      &\text{in } \Omega \times \sV, \\
      \gamma_+ \varphi = \cR^T \tilde\gamma_- \varphi
      &\text{on } \Gamma_+ \\
       \varphi_{|t=0} = \varphi_\init =1 & \text{on } \{ t =0 \}
    \end{dcases}
  \end{align}
  is bounded and converges to zero in $L^\infty$, and so that the
  operators $\cB$ and $\mathfrak C$ satisfy (for smooth $f$)
  \begin{equation}
    \label{eq:diff-bound-2}
    \begin{dcases}
      \int_{\Omega \times \sV}
      [\cB^*(f^2) - 2 f \sigma \cL f]\, (1-\varphi_t)
      \dd \mu
      \lesssim \| f \|_{L^2(\Omega\times\sV,\mathrm{d}\mu)} \,
      \left| \int_{\Omega \times \sV} \sigma f \cL f \dd
        \mu\right|^{\frac12}
      +
      \left| \int_{\Omega \times \sV} \sigma f \cL f \dd
        \mu\right|
      ,\\
      \left| \int_{\Omega \times \sV}
        [\cB^*(f f_\infty) - f_{\infty} \sigma \cL f]\, (1-\varphi_t)
        \dd \mu  \right|
      \lesssim \left| \int_{\Omega \times \sV} \sigma f \cL f \dd \mu \right|^{\frac12},
    \end{dcases}
  \end{equation}
  \begin{equation}
    \label{eq:creation-bound}
    \begin{dcases}
      \int_{\Omega \times \sV}
      \mathfrak C^*(f^2) \, \varphi_t
      \dd \mu
      \lesssim \| f \|_{L^2(\Omega\times\sV,\mathrm{d}\mu)}
      \, \left| \int_{\Omega \times \sV} \sigma f \cL f \dd \mu\right|^{\frac12}
      + \left| \int_{\Omega \times \sV} \sigma f \cL f \dd \mu\right|,\\
      \left| \int_{\Omega \times \sV}
      \mathfrak C^*(f f_\infty)\, \varphi_t \dd \mu \right|
      \lesssim \left| \int_{\Omega \times \sV} \sigma f \cL f \dd \mu\right|^{\frac12},
    \end{dcases}
  \end{equation}
  where we denote $\mathrm{d} \mu :=f_\infty(x,v)^{-1} \dd x \dd v$.
\end{hypothesis}

\begin{remark}
  In order to capture many cases, and in particular
  \cref{ex:hypoelliptic-decay} in the hypoelliptic case,
  \hypref{h:tcc} is formulated generally but we provide below easier
  conditions for the main practical cases apart from
  \cref{ex:hypoelliptic-decay}. In practice, note that the operator
  $\mathfrak C$ is only used in order to overcome degeneracy for
  hypoelliptic equations and otherwise one can take $\mathfrak
  C=0$. Regarding the operator $\cB$, one can take \(\cB=0\) if
  regularity is propagated by the transport semigroup or if $\cL$ is
  bounded. Otherwise we can drop the latter assumption by taking
  $\cB=\sigma \cL$, at the price of a $\Gamma$ condition (see
  below). Note that we could allow in \hypref{h:tcc} a more general
  \(\cB\) provided that $\cT+\cB$ combined with the boundary
  conditions associated with $\cR^T$ (or more generally $\cR'$)
  generates a contractive semigroup in $L^\infty$ admitting $1$ as
  equilibrium, and whose adjoint admits $f_\infty^2$ as equilibrium.
\end{remark}

\sethypothesistag{\ref*{h:tcc}'}
\begin{hypothesis}[Transport control condition -- practical]
  \phantomsection
  \label{h:simple-tcc}

  Assume the weight \(w=1\) in \hypref{h:local}. Then \hypref{h:tcc} is
  satisfied in the following cases:

  \begin{description}
  \item[Case 1.]  Consider the transport semigroup \(\strans_t\)
    created by $\cT$ and $\cR$ and assume that the collision operator
    \(\cL\) is bounded. Then assume a connected domain
    $\Sigma \subset \Omega$ so that
    $\inf_{x \in \Sigma} \sigma(x) > 0$ and there is a non-negative
    $\chi \in L^\infty(\Omega \times \sV)$ with
    $\supp \chi \subset \Sigma \times \sV$, and a time $T>0$ and
    constant $c>0$ such that
    \begin{equation*}
      \forall \, (x_0,v_0) \in \Omega \times \sV \qquad
      \int_0^T \int_{\Sigma \times \sV}
      (\strans_t \delta_{x_0,v_0})\, \chi(x) \dd x \dd v \dd t
      \ge c.
    \end{equation*}

  \item[Case 1'.]  Consider the transport semigroup \(\strans_t\)
    created by $\cT$ and $\cR$ and assume \(\strans_t\) propagates
    regularity in \(W^{1,\infty}\) and that the collision operator
    \(\cL\) is the Fokker-Planck operator. Then assume a connected
    domain $\Sigma \subset \Omega$ so that
    $\inf_{x \in \Sigma} \sigma(x) > 0$ and there is a non-negative
    $\chi \in W^{1,\infty}(\Omega \times \sV)$ with
    $\supp \chi \subset \Sigma \times \sV$, and a time $T>0$ and
    constant $c>0$ such that
    \begin{equation*}
      \forall \, (x_0,v_0) \in \Omega \times \sV \qquad
      \int_0^T \int_{\Sigma \times \sV}
      (\strans_t \delta_{x_0,v_0})\, \chi(x) \dd x \dd v \dd t
      \ge c.
    \end{equation*}

  \item[Case 2.] Assume \eqref{eq:gen} generates a bounded semigroup
    \(\sfull_t\) in \(L^1(\Omega\times\sV)\) and $\cL$ satisfies the
    condition
    \begin{equation}
      \label{eq:gamma-2-condition-l}
      f_\infty \cL\left(\frac{f^2}{f_\infty}\right)
      - 2 f \cL f \ge 0.
    \end{equation}
    Then assume a connected domain $\Sigma \subset \Omega$ so that
    $\inf_{x \in \Sigma} \sigma(x) > 0$, and there is a non-negative
    $\chi \in L^\infty(\Omega \times \sV)$ with
    $\supp \chi \subset \Sigma \times \sV$ and a time $T>0$ and constant
    $c>0$ such that
    \begin{equation}
      \label{eq:gcc}
      \forall \, (x_0,v_0) \in \Omega \times \sV \qquad
      \int_0^T \int_{\Sigma \times \sV}
      (\sfull_t \delta_{x_0,v_0})\, \chi(x) \dd x \dd v \dd t
      \ge c.
    \end{equation}
  \end{description}
\end{hypothesis}

The condition \eqref{eq:gamma-2-condition-l} corresponds to the
positivity of the \(\Gamma\) function in the Bakry-Émery calculus,
that follows from reversibility (see
Subsection~\ref{sec:gamma-2-verification}):
\begin{proposition}[$\Gamma$ condition]\label{thm:gamma-2-verification}
  Assume that $\cL$ is either the Fokker-Planck operator or the
  linear Boltzmann operator with a reversible kernel, then $\cL$
  satisfies \eqref{eq:gamma-2-condition-l}.
\end{proposition}

\begin{remark}
  The assumption that $\Sigma$ is connected could be relaxed: our
  method works provided the different connected components of $\Sigma$
  are connected through the transport semigroup. For a simple setting,
  this is discussed in
  \cite{dietert-herau-hutridurga-mouhot-2022-trajectorial}.
\end{remark}

For the functional inequalities capturing the macroscopic coercivity
on \(\Sigma\), we introduce a notion of regularity for
domain-potential pairs, represented in \cref{fig:reg-dom}:
\begin{figure}
  \centering
  \begin{tikzpicture}[xscale=1.4]
    \draw[domain=0:7,samples=150,fill=black!10!white,very thick]
    plot ({7-\x},{2-0.3*sin(50*(7-\x)*(7-\x))})
    -- (0,1)
    -- (0,0)
    -- plot (\x,{0.3*sin(50*\x*\x)});
    \draw (1.2,1) node {Domain \(\sU\)};
    \draw[->] (3,1) -- (2.8,1);
    \draw (3,1) node[anchor=west]{Confining potential \(\Phi\)};
    \draw[->] (1,-1) -- (6,-1);
    \draw (3.5,-1) node[anchor=north]{Curvature does not grow
      faster than \(|\nabla \Phi|\) as \(|x| \to \infty\)};
  \end{tikzpicture}
  \caption{Representation of a regular domain-potential pair
    $(\sU,\Phi)$}
  \label{fig:reg-dom}
\end{figure}
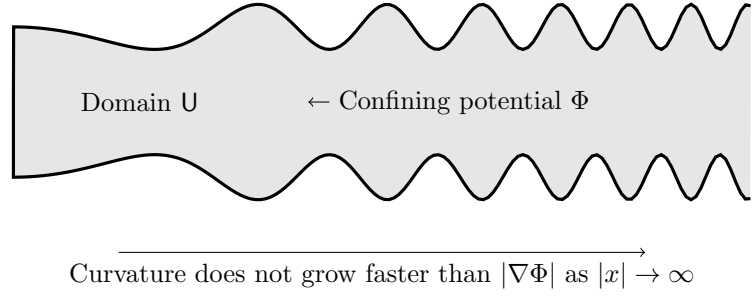
\begin{definition}
  \label{def:weighted-reg}
  Given $\epsilon >0$, a domain $\sU \subset \R^n$, $n \ge 1$, and a
  potential $\Phi \in C^2(\sU)$, we say that this domain-potential
  pair $(\sU,\Phi)$ satisfies the \emph{$\epsilon$-regularity
    condition} if $|\nabla^2 \Phi| \lesssim 1 + |\nabla \Phi|$ and for
  every boundary point $x \in \partial \sU$ there exists an isometry
  $T:\R^n \to \R^n$ with $T(x) = 0$ and $f \in C^1(\R^{n-1})$ with
  $\| \nabla f \|_\infty \le 1/8$ such that, denoting $B(x,r)$ the
  ball of radius $r$ around $x$ and
  $\lfloor \nabla \Phi \rceil := (1+|\nabla \Phi|^2)^{1/2}$,
  \begin{equation*}
    T\left(B(x, \epsilon \lfloor \nabla \Phi \rceil^{-1}) \cap U\right)
    = T\left(B(x, \epsilon \lfloor \nabla \Phi \rceil^{-1})\right)
    \cap
    \left\{ (\xi',\xi_n) \in \R^{n-1} \times \R : \xi_n > f(\xi') \right\}.
  \end{equation*}
\end{definition}

\begin{remark}
  When the boundary $\partial \sU$ is bounded and at least $C^1$ this
  is automatically satisfied.
\end{remark}

\begin{hypothesis}[Macroscopic coercivity]
  \label{h:coe}
  Given the set $\Sigma \subset \Omega$ from \hypref{h:tcc}, the pair
  $(\Sigma,\phi)$ is $\epsilon$-regular for some $\epsilon >0$, and
  the potential $\phi$ satisfies a Poincaré-Wirtinger inequality on
  $\Sigma$, i.e.\ there is $\lambda_2 > 0$ so that
  \begin{align}
    \label{eq:x-coerc}
    \forall \, \rho \in H^1(\Sigma;\ee^{\phi}), \quad
    \int_{\Sigma} \left| \nabla_x \rho + \rho \nabla_x \phi  \right|^2
    \ee^{\phi} \dd x \ge \lambda_2
    \int_{\Sigma} \left| \rho - \left( \int_\Sigma \rho \right)
    \ee^{-\phi} \right|^2\, \lfloor \nabla \phi \rceil^2\,
   \ee^{\phi} \dd x.
  \end{align}
\end{hypothesis}

\begin{remark}
  Such inequality, with or without the additional weight
  $\lfloor \nabla \phi \rceil^2$ on the right hand side
  of~\eqref{eq:x-coerc}, is a classical concentration estimate. We
  discuss sufficient conditions in \cref{sec:poincare}.
\end{remark}

\subsection{Main results}

For considering the evolution problem~\eqref{eq:gen} we need a notion
of solution that admits suitable traces. This is provided in
\cref{sec:traces}: given initial data
$f_{\init} \in L^2(\Omega \times \sV, \dd \mu)$, there is a unique
solution
$f=f(t,x,v) \in C^0([0,+\infty);L^2(\Omega \times \sV, \dd \mu))$
to~\eqref{eq:gen} that admits traces
$\gamma f \in L^2([0,+\infty) \times \partial \Omega; \mathcal H(\sV,
(\vec{n} \cdot v)^2 \dd \mu))$, where $\mathcal H$ is $L^2$ for the
bounded linear Boltzmann and $H^{-1}$ for the Fokker-Planck
operator. Under the abstract condition of the previous subsection, we
establish a spectral gap in the following theorem; however note that
when $\cL$ is given by the Fokker-Planck operator (and is unbounded)
and one uses the full semigroup $\sfull_t$ for the control condition,
then to fully justify the a priori estimates of the proof, it would be
necessary to approximate the operator by bounded operators, so that
the trilinear boundary terms are all well-defined (note that integrals
involved are however perfectly controlled from above by the a priori
estimates).

\begin{theorem}[Quantitative relaxation to equilibrium]
  \label{theo:main}
  Assume
  \hypref{h:geom}-\hypref{h:eq}-\hypref{h:local}-\hypref{h:bdd}-\hypref{h:tcc}-\hypref{h:coe}. There
  are $C > 1$, $\Lambda >0$ such that given any
  $f_{\init} \in L^2(\Omega \times \sV, \dd \mu)$, any
  $f=f(t,x,v) \in C^0([0,+\infty);L^2(\Omega \times \sV, \dd
  \mu))$ admitting traces
  $\gamma f \in L^2([0,+\infty) \times \partial \Omega;
  H^{-1}(\sV, (\vec{n} \cdot v)^2 \dd \mu))$ and solution
  to~\eqref{eq:gen} satisfies
  \begin{align}
    \label{eq:relax}
    \left\| f_t - \left( \int_{\Omega \times \sV}
    f_{\init}\right) f_\infty  \right\|_{L^2(\Omega \times \sV,
    \dd \mu)}
    \le C \ee^{-\Lambda t}
    \left\| f_{\init} - \left( \int_{\Omega
    \times \sV} f_{\init} \right) f_\infty \right\|_{L^2(\Omega
    \times \sV, \dd \mu)}.
  \end{align}
  The constants $C$ and $\Lambda$ can be computed from the
  constants in the assumptions.
\end{theorem}

This covers the following common cases.
\begin{corollary}[Application to concrete equations]
  \label{cor:LB}
  We list the possible settings:

  Regarding the \textbf{geometric constraints}:
  \begin{itemize}
  \item The spatial domain is either $\Omega = \T^d$ or a smooth
    open subset $\Omega \subset \R^n$.\
  \item The velocity space $\sV$ and potential $\phi$ are either
    $(\S^{d-1},0)$ or $(\R^d,\phi)$ with $\phi \in C^2(\Omega)$.
  \item There is a connected open subset $\Sigma \subset \Omega$ so
    that $(\Sigma,\phi)$ is $\epsilon$-regular for some $\epsilon>0$,
    and $\sigma \gtrsim 1$ and
    $\frac{|\nabla \phi|^2}{2} - \Delta \phi \xrightarrow[]{|x| \to
      \infty} \infty$ on $\Sigma$ and
    $\vec{n} \cdot \nabla \phi \ge 0$ on $\partial \Sigma$.
  \end{itemize}

  Regarding the \textbf{collision process}:
  \begin{itemize}
  \item The operator $\cL$ is the linear Boltzmann
    operator~\eqref{eq:LB} or the Fokker-Planck
    operator~\eqref{eq:FP}.
  \item There is unique equilibrium measure $M=M(v)$ for $\cL$.
  \item The operator $\cL$ is symmetric non-negative and satisfies a
    weighted spectral gap inequality~\eqref{eq:loc-coerc} with weight
    $1 \le w \in L^1(M)$.
  \end{itemize}

  Regarding the \textbf{transport control conditions}:
  \begin{itemize}
  \item When the transport flow is \textbf{deterministic}, e.g. when
    there is no boundaries or for specular or bounce-back boundary
    conditions in~\eqref{eq:boundary}, there are characteristics
    $(X_t(x,v),V_t(x,v))$ of $\strans_t$. Assume that there is
    $\chi \in C^\infty_c(\Sigma)$ and $T>0$ such that
    \begin{align}
      \label{eq:ugcc-deterministic}
      \forall \, (x,v) \in \Omega \times \sV, \quad \int_0^T
      \chi\left(X_{t}(x,v)\right) w(V_{t}(x,v)) \dd  t \ge 1.
    \end{align}
    Additionally, assume that \(\cL\) is bounded or that \(\cL\) is
    the Fokker-Planck operator, \(\strans_t\) propagates
    \(W^{1,\infty}\) regularity and the weight \(w\) is such that
    \(|\nabla w| \lesssim w\) and
    \begin{equation}
      \label{eq:ugcc-deterministic-upper}
      \int_0^T \indicator_{\supp \chi}\left(X_{t}(x,v)\right)
      w(V_{t}(x,v)) \dd t \lesssim 1.
    \end{equation}
  \item Assume general Maxwell conditions~\eqref{eq:boundary-concrete}
    with $\alpha : \partial\Omega \to [0,1]$. Assume that \(\cL\) is a
    scattering operator satisfying detailed-balance (including
    Fokker-Planck). Then assume that there is
    $\chi \in C^\infty_c(\Sigma)$, \(c>0\) and $T>0$ such that
    \begin{equation}
      \label{eq:decay-bd-hyp}
      \forall \, (x_0,v_0) \in \Omega \times \sV \qquad
      \int_0^T \int_{\Sigma \times \sV}
      (\sfull_t \delta_{x_0,v_0})\, \chi(x) \dd x \dd v \dd t
      \ge c.
    \end{equation}
  \end{itemize}

  Then \cref{theo:main} applies and implies exponential
  relaxation to equilibrium for solutions in $L^2_{x,v}(\mu)$
  with quantitative estimate on the rate.
\end{corollary}

\begin{remark}
  The condition \eqref{eq:ugcc-deterministic} for a deterministic flow
  covers \cref{ex:weight-gain} with the weight
  \(w(v) = \sqrt{1+|v|^2}\). For this application, we only use the
  necessary gain from the weight so that we also have the upper bound
  \eqref{eq:ugcc-deterministic-upper}.
\end{remark}

The proof of the abstract \cref{theo:main} is given in
\cref{sec:abstract} but makes use of quantitative \emph{divergence
  inequalities} proved in \cref{sec:divergence}. The proof of
\cref{cor:LB} is given in \cref{sec:concrete}.

The trajectorial approach can also be used to show convergence without
a rate for a non-uniform control condition, which is discussed in
\cref{sec:non-unif}.

\subsection{Geometric control conditions and hypoellipticity}
\label{sec:intro-hypo}

We now come back to \cref{ex:hypoelliptic-decay} in which we relax the
control condition when the collision operator \(\cL\) is regularising,
i.e.\ \eqref{eq:gen} is hypoelliptic. To achieve such a relaxation, we
use the additional \emph{corrector} operator \(\mathfrak C\) in
\hypref{h:tcc}.  In the case of the Fokker-Planck operator with the
Hörmander-type structure $\sigma \cL = - \cA^*\cA$, a natural choice
is $\mathfrak C = f_\infty^{-1} \cA^*(f_\infty a \cdot)$ for a
well-chosen bounded function $a \in L^\infty(\Omega\times \sV)$ that
satisfies \eqref{eq:creation-bound} and ensures the decay of the
solution to~\eqref{eq:tcc-decay}.  In the simple but illuminating
\cref{ex:hypoelliptic-decay}, we deduce the existence of a spectral
gap with quantitative estimates, in spite of the fact that the
transport control condition is \emph{not} satisfied (and neither are
more standard uniform geometric control conditions); moreover the
corresponding stochastic process suggests that the quadratic
degeneracy of $\sigma$ is critical for the existence of a spectral
gap. Note that we also give an alternative proof using the commutator
method developed in~\cite{villani-2009-hypocoercivity}, see
\cref{sec:counter}.
\begin{theorem}
  \label{thm:w-example}
  Consider $\Omega=\sV=\R$, the potential $\phi(x) = x^2/2$, a
  thermalisation degeneracy function $\sigma \in L^\infty$ that
  satisfies $\sigma(x) \gtrsim \min(1,x^2)$, and the Fokker-Planck
  collision operator~\eqref{eq:FP}. Then the semigroup is relaxing
  exponentially to equilibrium.
\end{theorem}

\subsection{Functional inequalities}
\label{sec:intro-fct-in}

In the good set \(\Sigma\) of \hypref{h:tcc}, we perform a standard
micro-macro decomposition and note that the dissipation gives some
control on the macroscopic density. This was already implicitly
understood in \cite{hoermander-1967-hypoelliptic} and more recently
formulated in the kinetic setting in
\cite{albritton-armstrong-mourrat-novack-2019-variational-fokker-planck}. This
approach has also been used to show hypocoercivity in much more
restricted cases
\cite{cao-lu-wang-2019-l-langevin,brigati-2021-time-fokker-planck}.

A key ingredient is the use of a suitable divergence inequality (also
known as Bogovskiǐ operator to invert the divergence), that we extend
to cases combining weight and boundaries. This is of independent
interest, and it is proved in \cref{sec:divergence}.
\begin{theorem}
  [Divergence, Poincaré-Lions, Korn and Stokes inequalities]
  \label{theo:ineq}
  Consider $n \ge 1$ and $\sU \subset \R^n$ open and consider a
  $C^2$ potential $\Phi : \sU \to \R$ so that $(\sU,\Phi)$ is
  $\epsilon$-regular for some $\epsilon>0$, and
  $\ee^{-\Phi(z)} \dd z$ satisfies the Poincaré-Wirtinger
  inequality on $\sU$:
  \begin{equation*}
    \forall \, \rho \in H^1(\sU ; \ee^\Phi) \text{ with }
    \int_\sU \rho =0, \quad \int_{\sU} \left| \nabla \rho + \rho
      \nabla \Phi  \right|^2 \ee^{\Phi} \dd z \gtrsim
    \int_{\sU} \rho^2\,
    \lfloor \nabla \Phi \rceil^2\,
    \ee^{\Phi} \dd z.
  \end{equation*}
  Then the following inequalities hold with quantitative
  estimates:
  \begin{enumerate}
  \item {\bf Divergence inequality.}  There is $C_\sD > 0$ and a
    linear map $\sD$ mapping any $g \in L^2(\sU;\ee^\Phi)$ with
    $\int_{\sU} g =0$ to a $\vec{F} : \sU \to \R^n$ in
    $H^1(\sU;\ee^\Phi)$ that satisfies
    \begin{align}
      \label{eq:divergence}
      \left\{
      \begin{aligned}
        &\nabla \cdot \vec{F} = g \text{ in } \sU, \\[2mm]
        &\vec{F} = 0 \text{ on } \partial \sU, \\[2mm]
        &\| \vec{F} \|_{L^2(\sU; \lfloor \nabla \Phi \rceil^2
          \ee^\Phi)} + \| \nabla \vec{F} \|_{L^2(\sU; \ee^\Phi)}
        \le C_{\sD} \| g \|_{L^2(\sU ; \ee^\Phi)}.
      \end{aligned}
          \right.
    \end{align}
  \item {\bf Poincaré-Lions inequality.} There is
    $C_{\text{\tiny \emph{PL}}}>0$ so that for any
    $h \in L^2(\sU; \ee^\Phi)$ one has
    \begin{align}
      \label{eq:poincare-lions}
      \left\| h - \left( \int_\sU h \right)\ee^{-\Phi}
      \right\|_{L^2(\sU ; e^\Phi)} \le C_{\text{\tiny
      \emph{PL}}}
      \left\| \nabla h + h \nabla \Phi
      \right\|_{(H^1_0(\sU;\ee^\Phi))'}
    \end{align}
    where $(H^1_0(\sU;\ee^\Phi))'$ is the standard dual space.
  \item {\bf Stokes inequality.} There is
    $C_{\text{\tiny \emph{S}}}>0$ so that for any
    $\vec{s} : \sU \to \R^n$ in $L^2(\sU;\ee^\Phi)$ with
    $\int_\Sigma \vec{s} = 0$, the unique solution
    $(\vec{u},p) \in H^1(\sU;\ee^\Phi) \times L^2(\sU;\ee^\Phi)$ with
    $\int_\sU p =0$ to
    \begin{equation}
      \label{eq:stokes}
      \begin{dcases}
        - \nabla \cdot \left( \nabla + \nabla \Phi  \right)
        \vec{u} + \left( \nabla + \nabla \Phi \right) p =
        \vec{s} \quad \text{in } \sU, \\[2mm]
        \nabla \cdot \vec{u} = 0 \quad \text{in } \sU, \\[2mm]
        \vec{u} = 0 \quad \text{on } \partial \sU,
      \end{dcases}
    \end{equation}
    satisfies
    \begin{align}
      \label{eq:stokes-estim}
      \left\| \vec{u} - \left( \int_\sU \vec{u} \right)\ee^{-\Phi}
      \right\|_{L^2(\sU;\lfloor \nabla \Phi \rceil^2 \ee^\Phi)} +
      \left\| \left( \nabla + \nabla \Phi \cdot \right) \vec{u}
      \right\|_{L^2(\sU;\ee^\Phi)}
      + \left\| p \right\|_{L^2(\sU;\ee^\Phi)}
      \le C_{\text{\tiny \emph{S}}} \left\| \vec{s}
      \right\|_{L^2(\sU;\ee^\Phi)}.
    \end{align}
  \item {\bf Korn inequalities.} There is
    $C_{\text{\tiny \emph{K}}}>0$ so that for any
    $\vec{u} : \sU \to \R^n$ in $H^1(\sU;\ee^\Phi)$ satisfying
    \begin{align}
      \label{eq:korn-average-conditions}
      \forall \, i,j=1,\dots,n, \quad
      \int_\sU \left( \partial_{z_i} + \partial_{z_i} \Phi
      \right) \vec{u}_j  = \int_\sU \left( \partial_{z_j} +
      \partial_{z_j} \Phi \right) \vec{u}_i
    \end{align}
    it holds
    \begin{align}
      \label{eq:korn-statement}
      \left\| \left( \nabla + \nabla \Phi \cdot \right) \vec{u}
      \right\|_{L^2(\sU ; \ee^\Phi)}  \le C_{\text{\tiny
      \emph{K}}}
      \left\| \left( \nabla + \nabla \Phi \cdot
      \right)^{\sym} \vec{u}
      \right\|_{L^2(\sU;\ee^\Phi)}
    \end{align}
    with
    \begin{align*}
      \left\| \left( \nabla + \nabla \Phi \cdot
      \right)^{\sym} \vec{u}
      \right\|_{L^2(\sU;\ee^\phi)} ^2 := \sum_{i,j =1,\dots,n}
      \left\| \frac{\left( \partial _{z_i} + \partial_{z_i} \Phi
      \right) \vec{u}_j + \left( \partial_{z_j} + \partial_{z_j}
      \Phi \right) \vec{u}_i}{2} \right\|_{L^2(\sU;\ee^\Phi)} ^2.
    \end{align*}

  \end{enumerate}
\end{theorem}

\begin{remark}\label{thm:boundary-korn}
  Without the assumption~\eqref{eq:korn-average-conditions}, the
  estimate~\eqref{eq:korn-statement} holds if the boundary prohibits
  all rotations and drifts. Let us illustrate this with the case of
  non-penetration boundary condition $\vec{n} \cdot \vec{u} = 0$ on
  $\partial \sU$. Let $p = \int_{\sU} z\, \ee^{-\Phi} \in \R^n$ be the
  weighted centre and for $1\le i < j \le n$ let
  $E^{ij} = e^i \otimes e^j - e^j \otimes e^i$, where $e^1,\dots,e^n$
  is an orthonormal basis of $\R^n$. Then suppose that there are
  scalar functions $\chi^i, \chi^{ij} \in C^1_c(\overline{\sU})$,
  $1 \le i \not = j \le n$, such that for $1 \le k \not = l \le n$,
  \begin{equation}
    \label{eq:boundary-basis-chi}
    \left\{
      \begin{lgathered}
        \int_{\partial \sU} \chi^i\, \left( \vec{n} \cdot e^{k}
        \right) = \delta_{i=k}, \quad \int_{\partial \sU}
        \chi^i\, \left[ \vec{n}
          \cdot \left( E^{kl} (z-p) \right) \right] = 0, \\
        \int_{\partial \sU} \chi^{ij}\, \left( \vec{n} \cdot
          e^{k}\right) = 0, \quad \int_{\partial \sU} \chi^{ij}\,
        \left[ \vec{n} \cdot \left( E^{kl} (z-p) \right) \right]
        = \delta_{(i,j)=(k,l)}.
      \end{lgathered}
    \right.
  \end{equation}
  Then \eqref{eq:korn-statement}
  holds. When~\eqref{eq:boundary-basis-chi} is not imposed, the
  symmetric gradient has a non-trivial kernel. This kernel can,
  e.g., be controlled by $\|\vec{u} \cdot \nabla \Phi\|$, as
  in~\cite{carrapatoso2020weighted}, if $\Phi$ has no rotation
  symmetry, which can be quantified by the constant
  \begin{equation*}
    \sup_{J \in \mathfrak S}
    \frac{\displaystyle \left| \int_{\sU} J(z)  \nabla
        \ee^{-\Phi} \right|}{\displaystyle \left\| J \right\|}
  \end{equation*}
  where the supremum is taken over the set of affine functions
  \begin{equation*}
    \mathfrak S := \left\{ J(z) =\sum_{i,j} b_{ij} E^{ij}(z-p) +
      \sum_{i} b_i e^{i}, \quad b_i, b_{ij} \in \R, \quad
    \text{ compatible with the boundary} \right\}.
  \end{equation*}
\end{remark}

\begin{remark}
  The above form of the divergence inequality seems more general than
  the existing literature, due to the addition of the potential force
  and also its combination with a boundary. We refer
  to~\cite{MR553920,MR631691,duran-2012-bogovskii,geissert-heck-hieber-2006-bogoskii,acosta-duran-2017-divergence-operator,danchin-mucha-2013-divergence,duran-garcia-2010-divergence-inequality,bourgain-brezis-2002-divergence-equation,duran-muschietti-2001-right-inverse-divergence,borchers-sohr-1990-rotation-divergence-equation}
  among an important literature. What is called Poincaré-Lions
  inequality above is used in~\cite{MR0521262} and mentioned
  in~\cite{MR1932965} and we introduced the terminology
  in~\cite{carrapatoso2020weighted}. The Stokes equation is a
  classical equation of fluid dynamics, see
  \cite{galdi-2011-navier-stokes,hieber-saal-2018-stokes-equation} for
  an overview and
  \cite{geissert-heck-hieber-sawada-2010-stokes-unbounded} for some
  extension to unbounded domains without potential. The Korn
  inequality was discovered
  in~\cite{korn-1906-abhandlungen-elastizitaetstheorie-ii,
    korn-1908-solution,
    korn-1909-ueber-ungleichungen-theorie-schwingungen-rolle} in the
  case of bounded domain with Dirichlet conditions (see
  also~\cite{MR1368384} for a somehow recent review), and extended to
  non-penetration conditions in~\cite{MR1932965}, and to the whole
  space with confining potential in~\cite{Duan_2011} (non-constructive
  argument) and~\cite{carrapatoso2020weighted} (constructive
  argument). Our statement includes these previous works and extend
  them. We deduce the Korn inequality from the Poincaré-Lions
  inequality arguing as in~\cite{carrapatoso2020weighted}, however the
  Poincaré-Lions inequality is proved in new cases and by a novel
  method.
\end{remark}

\section{Weighted divergence and related inequalities}
\label{sec:divergence}

We prove \cref{theo:ineq} in this section. Let us denote
$\nabla^\Phi := \nabla + \nabla \Phi$.

\subsection{The Poincaré inequality}
\label{sec:poincare}

We first extend the standard Poincaré-Wirtinger inequality. In the
case $\sU=\R^n$, standard arguments, see for instance~\cite[Proof of
Theorem~6.2.21]{deuschel-stroock-1989-large} and \cite[Theorem~A.1 in
A.19]{villani-2009-hypocoercivity}, show that the Poincaré inequality
follows from $\frac{|\nabla \Phi|^2}{2} - \Delta \Phi \to \infty$ as
$|z| \to \infty$. The additional weight $\lfloor \nabla \Phi \rceil^2$
in~\eqref{eq:x-coerc} is classically obtained under the assumption
$|\nabla^2 \Phi| \lesssim 1 + |\nabla \Phi|$, see for
instance~\cite[Lemma~A.24 in
Section~A.23]{villani-2009-hypocoercivity}. In order to deal with
boundaries, we will assume furthermore that
$\vec{n} \cdot \nabla \Phi \ge 0$ on $\partial \sU$, where $\vec{n}$
is the unit outgoing normal on $\sU$. (Note that the latter assumption
could likely be replaced by simply assuming the $\epsilon$-regularity
of $(\sU,\Phi)$).

\begin{lemma}\label{thm:poincare-bounded}
  Consider $n \ge 1$, $\sU \subset \R^n$ open,
  $\Phi : \sU \to \R$ in $C^2$ so that
  $\vec{n} \cdot \nabla \Phi \ge 0$ on $\partial \sU$ and
  $\frac{|\nabla \Phi|^2}{2} - \Delta \Phi \to \infty$ as
  $|z| \to \infty$. Then it satisfies the weighted
  Poincaré-Wirtinger inequality
  \begin{equation}
    \label{eq:pw-weight}
    \forall \, \rho \in H^1(\sU ; \ee^{\Phi}) \text{ with }
    \int_\sU \rho =0, \quad
    \int_{\sU} \left| \nabla^\Phi \rho \right|^2
    \ee^{\Phi} \gtrsim \int_{\sU} \rho^2\,
    \lfloor \nabla \Phi \rceil^2\, \ee^{\Phi}.
  \end{equation}
\end{lemma}

\begin{proof}
  Assume first that $\rho = 0$ on $\partial \sU$. Then a standard
  calculation yields
  \begin{equation}
    \label{eq:strook}
    \begin{split}
      \int_{\sU} \left| \nabla^\Phi \rho \right|^2 \ee^\Phi
      & = \int_{\sU} \left|\nabla \left( \rho \ee^\Phi
        \right) \right|^2 \ee^{-\Phi} \\
      & = \int_{\sU} |\nabla(\rho \ee^{\Phi/2})|^2+ \int_\sU
      \rho^2 \left( \frac{|\nabla \Phi|^2}{4} - \frac{\Delta
          \Phi}{2} \right) \ee^{\Phi} + \frac12 \int_{\partial\sU} \rho^2
      \left( \vec{n} \cdot \nabla \Phi \right) \ee^\Phi.
    \end{split}
  \end{equation}
  Since we assume that
  $\frac{|\nabla \Phi|^2}{2} - \Delta \Phi \to \infty$ at
  $|z| \to \infty$, this controls the $L^2(\sU;\ee^\Phi)$ norm of
  $\rho$ for large $z$. As explained in \cite[Thm~A.1 in
  A.19]{villani-2009-hypocoercivity} this can be combined with a
  standard Poincaré inequality on a ball to deduce
  \begin{equation}
    \label{eq:pw-noweight}
    \int_{\sU} \left| \nabla^\Phi \rho \right|^2
    \ee^{\Phi} \gtrsim \int_{\sU} \rho^2\, \ee^{\Phi}
  \end{equation}
  and if moreover $|\nabla^2 \Phi| \lesssim 1 + |\nabla \Phi|$
  one has
  $\frac{|\nabla \Phi|^2}{2} - \Delta \Phi \gtrsim |\nabla
  \Phi|^2$ and thus the combination of~\eqref{eq:strook}
  and~\eqref{eq:pw-noweight} implies~\eqref{eq:pw-weight}.
\end{proof}

\subsection{The divergence inequality $L^2 \to L^2$}

The assumptions on the potential $\Phi$ imply the following
Poincaré-type inequality on
$\tilde \Phi := \Phi + 2 \ln \lfloor \nabla \Phi \rceil$: for any
$h \in H^1(\sU;\ee^{\tilde \Phi})$ with $\int_\sU h =0$,
\begin{align}
  \label{eq:poincare-higher-weight}
  \int_{\sU} h^2 \lfloor \nabla \Phi \rceil^2\ee^{\tilde
  \Phi} \dd z
  \lesssim \int_{\sU} \left| \nabla h + h \nabla \Phi \right|^2
  \ee^{\tilde \Phi} \dd z
\end{align}
with quantitative estimate on the constant (it follows from applying
the Poincaré inequality on $\ee^{-\Phi}$ to $h$ and
$h \lfloor \nabla \Phi \rceil$ and combining linearly the two
estimates by the assumption on $|\nabla^2 \Phi|$).

Given $g \in L^2(\sU;\ee^\Phi)$ with $\int_{\sU} g =0$ we consider the
problem
\begin{align}
  \label{eq:div1}
  \nabla \cdot \vec{F}_0 = g \text{ in } \sU,
  \quad \vec{F}_0 \cdot \vec{n} = 0 \text{ on } \partial \sU
\end{align}
for a vector field $\vec{F}_0 : \sU \to \R^n$ in
$L^2(\sU;\ee^{\tilde \Phi})$. To solve it, we consider the following
elliptic problem
\begin{equation}
  \label{eq:elliptic-div}
  \nabla \cdot (\nabla \psi + \psi \nabla \tilde \Phi) = 0
  \text{ in } \sU, \qquad
  (\nabla \psi + \psi \nabla \tilde \Phi) \cdot \vec{n} = 0
  \text{ on } \partial\sU,
\end{equation}
and then define $\vec{F}_0:= \nabla \psi + \psi \nabla \tilde
\Phi$. The existence of a unique solution
to~\eqref{eq:elliptic-div} in $L^2(\ee^{\tilde \Phi})$ follows
from~\eqref{eq:poincare-higher-weight} and the Lax-Milgram
theorem. This solution then satisfies
\begin{align*}
  \| \vec{F}_0 \|_{L^2(\sU;\lfloor \nabla \Phi \rceil^2
  \ee^\Phi)} = \| \vec{F}_0 \|_{L^2(\sU ;\ee^{\tilde \Phi})} \lesssim
  \| g \|_{L^2(\sU ; \lfloor \nabla \Phi \rceil^{-2} \ee^{\tilde
  \Phi})} = \| g \|_{L^2(\sU ; \ee^\Phi)}.
\end{align*}

(Note that the same argument also shows
$\| \vec{F}_0 \|_{L^2(\sU;\ee^{\tilde \Phi})} \lesssim \| g
\|_{\big(H^1_0(\sU ; \ee^{\tilde \Phi})\big)'}$).

\subsection{The divergence inequality $L^2 \to H^1$}

The idea is to use the $L^2 \to L^2$ divergence inequality from
the previous subsection to reduce the problem to balls $B_k$ such
that on $B_k$ the weight $\ee^{\Phi}$ is not changing
significantly and that $B_k \cap \sU$ is star-shaped. After a
change of variable to flatten the boundary we can then use the
explicit representation formula of Bogovski\v{\i}. To start, note
that $|\nabla^2 \Phi| \lesssim 1+|\nabla \Phi|$ implies by a
direct Gronwall argument that there exists $\epsilon > 0$ so that
\begin{equation}
  \label{eq:var-grad-phi}
  2 \Big(
  1 + |\nabla \Phi(x)|
  \Big)
  \ge
  1 + |\nabla \Phi(y)|
  \ge
  \frac 12 \Big(
  1 + |\nabla \Phi(x)|
  \Big)
  \qquad
  \forall x,y \in \sU \text{ with }
  |x-y| \le 2 \epsilon,
\end{equation}
which directly implies
\begin{equation}
  \label{eq:w-grad-phi}
  4
  \lfloor\nabla \Phi(x)\rceil
  \ge
  \lfloor\nabla \Phi(y)\rceil
  \ge
  \frac 14
  \lfloor\nabla \Phi(x)\rceil
  \qquad
  \forall x,y \in \sU \text{ with }
  |x-y| \le 2 \epsilon.
\end{equation}
By reducing $\epsilon$ if necessary, we may assume without loss of
generality that~\eqref{eq:var-grad-phi} and \eqref{eq:w-grad-phi} hold
with the same $\epsilon$ as the \(\epsilon\)-regularity of
\((\sU,\Phi)\) in \cref{def:weighted-reg}. We then need the following
covering lemma.
\begin{lemma}
  Consider a domain $\sU \subset \R^n$ with $C^2$ potential
  $\Phi : \sU \to \R$ and $\epsilon > 0$ so that the
  \(\epsilon\)-regularity of \((U,\Phi)\) from \cref{def:weighted-reg}
  and \eqref{eq:w-grad-phi} holds. There is a cover
  $(B_k)_{k \in \mathcal{I}}$ of $\sU$ and a subordinate partition of
  unity $(\theta_k)_{k \in \mathcal{I}}$ such that
  \begin{enumerate}
  \item $B_k=B(z_k,r_k)$ is either a ball with $z_k \in \sU$ and
    $r_k = (\epsilon/40) \lfloor \nabla \Phi \rceil^{-1}$ and
    $B_k \subset \sU$ or a ball around $z_k \in \partial \sU$ of
    radius $r_k = \epsilon \lfloor \nabla \Phi \rceil^{-1}$ in which
    case there exists $B_k' = B(y_k,r_k/4)$ and $|z_k-y_k|=r_k/2$
    so that $B_k$ is star-shaped with respect to $B_k'$,
  \item each point is covered at most $C_d$ times where $C_d \in \N$
    only depends on the dimension $d$,
  \item one has
    $\| \nabla \theta_k \|_{L^\infty(B_k)} \lesssim r_k^{-1} $ and
    $ \ee^{\Phi(z_k)} \lesssim \ee^{\Phi(z)} \lesssim
    \ee^{\Phi(z_k)} $ for $z \in B_k$.
  \end{enumerate}
\end{lemma}
\begin{proof}
  Introduce the radius
  $r_b(z) = \epsilon \lfloor \nabla \Phi \rceil^{-1}(z)$ for
  $z \in \partial \sU$ and
  $r_i(z) = (\epsilon/40) \lfloor \nabla \Phi \rceil^{-1}(z)$ for
  $z \in \sU$.  By~\eqref{eq:w-grad-phi} we can cover
  $\sU$ by the balls $\tilde{B}(z)=B(z,r_b(z)/10)$ for
  $z \in \partial \sU$ and the balls $\tilde{B}(z)=B(z,r_i(z)/10)$
  for $z \in \sU$ satisfying $B(z,r(z)) \subset \sU$.

  By Vitali's covering lemma, there exists a disjoint
  subcollection $(\tilde{B}(z_k))_{k\in\mathcal{I}}$ of balls
  such that $\sU \subset \bigcup_k 5\tilde{B}(z_k)$, and we
  consider the covering $(B_k)_{k \in \mathcal{I}}$. Around any
  $z \in \sU$, the radii used are comparable by
  \eqref{eq:w-grad-phi} so that the fact that the
  $(\tilde{B}(z_k))_{k\in \mathcal{I}}$ are disjoint implies that
  every point is covered at most $C_d$ times with a dimensional
  constant $C_d$. Then~\eqref{eq:w-grad-phi} and the mean-value
  theorem imply
  $\ee^{\Phi(z_k)} \lesssim \ee^{\Phi(z)} \lesssim
  \ee^{\Phi(z_k)} $ on $z \in B_k$.

  Let $\zeta \in C^\infty(\R^n)$ with $\zeta(x) = 1$ if
  $|x| \le 1/2$ and $\zeta(x) = 0$ if $|x| \ge 1$, and let
  \begin{equation*}
    \forall \, z \in \sU, \quad
    w(z) := \sum_k \zeta\left(\frac{z-z_k}{r_k}\right).
  \end{equation*}
  Then $w(z) \ge 1$ on $\sU$ and
  $|\nabla w| \lesssim \lfloor \nabla \Phi \rceil$ so that
  $\theta_k(z) := \frac{1}{w(z)}
  \zeta\left(\frac{x-z_k}{r_k}\right)$ is a partition of unity
  satisfying the claimed properties. Finally, for a ball $B_k$
  centred around a boundary point, the claimed star-shaped ball
  $B_k'$ follows from the Lipschitz bound in
  \cref{def:weighted-reg} of \(\epsilon\)-regularity.
\end{proof}

Using the partition of unity from the previous lemma, define
$g_k := \nabla \cdot ( \theta_k \vec{F}_0)$ for
$k \in \mathcal{I}$, where $\vec{F}_0$ was constructed in the
previous subsection. Each $g_k$ has support included in $B_k$ and
has zero average on $\sU \cap B_k$ due to the divergence
structure and $\vec{F}_0 \cdot \vec{n} =0$ on $\partial
\sU$. Moreover
$g_k = \nabla \theta_k \cdot \vec{F}_0 + \theta_k g$ and
$|\nabla \theta_k| \lesssim \lfloor \nabla \Phi \rceil$ on $B_k$,
therefore (using the bound on the overlaps of the covering)
\begin{align}
  \label{eq:summation-1}
  &\begin{aligned}
    \sum_{k \in \mathcal{I}} \| g_k \|^2_{L^2(\sU; \ee^\Phi)}
    & \lesssim \sum_{k \in \mathcal{I}} \int_{\sU \cap B_k} g_k^2 \ee^\Phi \dd z
    \lesssim  \sum_{k \in \mathcal{I}}
    \int_{\sU \cap B_k} \left( |\vec{F}_0|^2
      \left| \nabla \theta_k \right|^2  +
      |g|^2  \left| \theta_k \right|^2 \right) \ee^\Phi \dd z \\
    & \lesssim  \int_{\sU} \left( |\vec{F}_0|^2 \lfloor \nabla \Phi
      \rceil^2 + |g|^2\right) \ee^\Phi \dd z
    \lesssim \| g \|_{L^2(\sU ; \ee^\Phi)} ^2,
  \end{aligned}\\
  \label{eq:summation-2}
  &\sum_{k \in \mathcal{I}} \| \lfloor \nabla \Phi \rceil\, \theta_k \vec{F}_0 \|^2_{L^2(\sU; \ee^\Phi)}
    \lesssim  \int_{\sU} |\vec{F}_0|^2 \lfloor \nabla \Phi \rceil^2 \ee^{\Phi}
    \dd z
    \lesssim \| g \|_{L^2(\sU ; \ee^\Phi)} ^2.
\end{align}

On $\overline{\sU \cap B_k}$ we claim
(following~\cite{MR631691,bourgain-brezis-2002-divergence-equation})
that there exists a linear map mapping $g_k$ to a vector field
$\vec{F}_k$ on $\sU \cap B_k$ so that
\begin{equation}
  \label{eq:bogo-loc}
  \left\{
    \begin{aligned}
      &\nabla \cdot \vec{F}_k = g_k = \nabla \cdot(\theta_k \vec{F}_0)
      \text{ in } \sU \cap B_k, \\
      &\vec{F}_k = 0 \text{ on } \partial (\sU \cap B_k) \\
      &\| \vec{F}_k \|_{L^2(\sU \cap B_k)}
      \lesssim
      \| \theta_k \vec{F}_0 \|_{L^2(\sU \cap B_k)} \\
      &\| \nabla \vec{F}_k \|_{L^2(\sU \cap B_k)}
      \lesssim
      \| g_k \|_{L^2(\sU \cap B_k)} \\
    \end{aligned}
  \right.
\end{equation}
where the constants are independent of $k \in \mathcal{I}$. The
vector field $\vec{F} := \sum_{k \in \mathcal{I}} \vec{F}_k$ on
$\sU$ then solves
\begin{align}
  \label{eq:bogo}
  \left\{
  \begin{aligned}
    &\nabla \cdot \vec{F} = g \text{ in } \sU, \\
    &\vec{F} = 0 \text{ on } \partial \sU, \\
    &\| \vec{F} \|_{L^2(\sU; \lfloor \nabla \Phi \rceil^2
      \ee^\Phi)} + \| \nabla \vec{F} \|_{L^2(\sU; \ee^\Phi)}
    \lesssim \| g \|_{L^2(\sU ; \ee^\Phi)}
  \end{aligned}
  \right.
\end{align}
by
\eqref{eq:summation-1}--\eqref{eq:summation-2}--\eqref{eq:bogo-loc},
which concludes the proof.

\begin{figure}
  \centering
  \begin{tikzpicture}[scale=2]
    \path[save path=\bdd] (0,-1.5)
    .. controls (0.2,-0.5) and (-0.1,-0.5) .. (0,0)
    .. controls (0.1,0.5) .. (-0.05,1.5);
    \path[save path=\din]
    (0,-1.5)
    .. controls (0.2,-0.5) and (-0.1,-0.5) .. (0,0)
    .. controls (0.1,0.5) .. (-0.05,1.5)
    -- (2,1.5) -- (2,-1.5) -- cycle;
    \path[save path=\b] (0,0) circle (1.0);
    \path[save path=\bd] (0.5,0) circle (0.25);

    \begin{scope}
      \clip [use path=\din];
      \draw[pattern={mydlines[size=6pt,line
        width=0.5pt,angle=-45]}, pattern color=blue] [use
      path=\b]; \draw[pattern={myslines[size=6pt,line
        width=0.5pt,angle=45]}, pattern color=magenta] [use
      path=\bd];
    \end{scope}
    \draw[dotted,thin] (0,-1.5) -- (0,1.5);
    \draw[red] (0,1.5) node[anchor=north west] {$\partial\sU$};
    \draw[red] [use path=\bdd];
    \draw[blue] (60:1) node[anchor=south west] {$B_k$};
    \draw[blue] [use path=\b];
    \draw[magenta] (0.5,0.23) node[anchor=south,fill=white] {$B_k'$};
    \draw[magenta] [use path=\bd];
  \end{tikzpicture}
  \caption{A ball of the cover that intersects the boundary of $\sU$}
  \label{fig:bogo}
\end{figure}
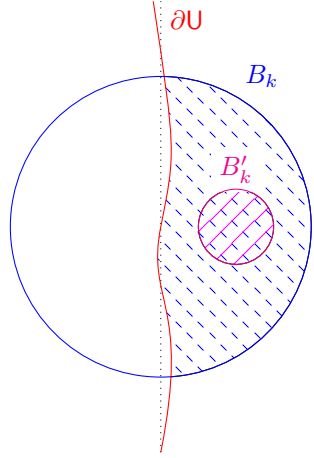

Let us now prove~\eqref{eq:bogo-loc}. We use the explicit
construction of Bogovski\v{\i}~\cite{MR553920,MR631691} (inspired
from an older idea of Sobolev~\cite{MR0165337}). As it is
important that the estimates in \eqref{eq:bogo-loc} are
independent of $k \in \mathcal{I}$ (so in particular independent
of the scale $r_k$), we explicitly scale the problem by the map
$T_k: x \to (x-y_k)/r_k$: the inner ball $B_k'$ becomes a ball
around the origin of radius $1/4$ and $T_k(B_k \cap \sU)$ is a
star-shaped domain around $T_k B_k'$ and contained in the ball of
radius $3/2$ around the origin, see \cref{fig:bogo}. We then
denote $\tilde g_k(x) = g_k(T_k x)$,
$\tilde{\vec{F}}_k(x) = r_k \vec{F}_k(T_k x)$ and
$\tilde{\vec{G}}_k(x) = r_k \vec{G}_k(T_k x) \text{ with }
\vec{G}_k =\theta_k \vec{F}_0$.  Then
$\tilde g_k = \nabla \cdot \vec{\tilde G}_k$ and \eqref{eq:bogo-loc} is
equivalent to
\begin{equation}
  \label{eq:bogo-loc-rescaled}
  \left\{
    \begin{aligned}
      &\nabla \cdot \tilde{\vec{F}}_k = \tilde g_k
      \text{ in } T_k(\sU \cap B_k), \\
      &\tilde{\vec{F}}_k = 0 \text{ on } \partial T_k(\sU \cap B_k) \\
      &\| \tilde{\vec{F}}_k \|_{L^2(T_k(\sU \cap B_k))}
      \lesssim
      \| \tilde{\vec{G}}_k \|_{L^2(T_k(\sU \cap B_k))} \\
      &\| \nabla \tilde{\vec{F}}_k \|_{L^2(T_k(\sU \cap B_k))}
      \lesssim
      \| \tilde g_k \|_{L^2(T_k(\sU \cap B_k))} \\
    \end{aligned}
  \right.
\end{equation}
We then define $\tilde{\vec{F}}_k$ through the explicit integral
representation
\begin{align*}
  \tilde{\vec{F}}_k(x) := \int_{T_k(\sU \cap B_k)} \tilde
  g_k(y)\, \vec{K}(x,y) \dd y \quad
  \text{ with } \quad
  \vec{K}(x,y) := \int_1 ^{+\infty}
  (x-y)\, \psi\left( y+ \tau (x-y)\right) \tau^{n-1} \dd \tau
\end{align*}
where $\psi : \R^n \to \R_+$ is a smooth function with support in
the ball of radius $1/4$ around the origin (which is $T_k(B_k')$)
and $\int \psi = 1$.

The estimates in the third and fourth equations
in~\eqref{eq:bogo-loc-rescaled} follow from Calderón-Zygmund
theory~\cite{MR84633} since all the components of
$\nabla \vec{K}$ behave like $|x-y|^{-n}$ (obtained by changing
variable $\tau = |x-y|^{-1} \tau'$) and satisfy the
Calderón-Zygmund conditions (for details see e.g. \cite[Section
III.3 (p~161)]{galdi-2011-navier-stokes}). Alternatively, the
bounds can be obtained by Fourier analysis
\cite{duran-2012-bogovskii}. To prove the first equation
in~\eqref{eq:bogo-loc-rescaled}, we argue by $L^2$-density and
assume $\tilde g_k \in C^0_c(T_k(\sU \cap B_k))$ and compute
\begin{align*}
  \nabla \cdot \tilde{\vec{F}}_k(x)
  & = \nabla \cdot \left[ \lim_{\ell \to +\infty}
    \int_{T_k(\sU \cap B_k)} \tilde{g}_k(y)
    \int_1 ^\ell (x-y)
    \psi\left( y+ \tau (x-y)\right) \tau^{n-1} \dd \tau  \dd y \right] \\
  & = \lim_{\ell \to +\infty}
    \int_{T_k(\sU \cap B_k)} \tilde{g}_k(y)
    \left[ \psi\left( y+ \tau (x-y)\right) \tau^n \right]_{\tau =
    1}^{\tau = \ell}
    \dd y \\
  & = \lim_{\ell \to +\infty}
    \left[ \ell^n  \int_{T_k(\sU \cap B_k)}
    \tilde g_k(y) \psi\left( y + \ell (x-y)  \right) \dd y
    \right] = \tilde g_k(x) \left( \int_{\R^n} \psi(y) \dd y \right)
    = \tilde g_k(x)
\end{align*}
where we have used $\int \tilde g_k = 0$ and $\int \psi =1$. By
density one deduces $\nabla \cdot \tilde{\vec{F}}_k = \tilde g_k$
in $L^2$. Finally to prove the second equation
in~~\eqref{eq:bogo-loc-rescaled}, we argue by density again and
consider $\tilde g_k \in C^0_c(T_k(\sU \cap B_k))$. Then the
vector field $\tilde{\vec{F}}_k$ is continuous and if
$x \not \in T_k(\sU \cap B_k)$ and $y \in T_k(\sU \cap B_k)$ then the
half-line $y + \tau (x-y)$ for $\tau \ge 1$ never intersects
$T_k(B_k')$, hence $\tilde{\vec{F}}_k$ vanishes on
$\partial T_k(\sU \cap B_k)$. By density, we thus deduce that
$\tilde{\vec{F}}_k \in H^1_0(T_k(\sU \cap B_k))$. This concludes
the proof of~\eqref{eq:bogo-loc} and thus concludes the proof
of~\eqref{eq:divergence} in \cref{theo:ineq}.

\subsection{Consequences and related inequalities}

The \emph{weighted Poincaré-Lions
  inequality}~\eqref{eq:poincare-lions} follows from integrating
$h \in L^2(\sU; \ee^\Phi)$ against $\nabla \cdot \vec{F} = g$ for any
$g \in L^2(\sU;\ee^\Phi)$ with zero-mass and $\vec{F}$
satisfying~\eqref{eq:bogo}. Regarding the \emph{Stokes
  equations}~\eqref{eq:stokes}, the a priori
estimate~\eqref{eq:stokes-estim} follows from integrating the first
equation on $\sU$ against $\vec{u} \ee^\Phi$ and using the Poincaré
inequality, and then integrating the first equation against
$\vec{F} \ee^\Phi$ where $\vec{F}$ solves~\eqref{eq:bogo} with any
$g \in L^2(\sU;\ee^\phi)$ with zero mass. In fact the right hand side
in~\eqref{eq:stokes-estim} could be replaced by
$\| \vec{s} \|_{L^2(\sU;\lfloor \nabla \Phi \rceil^{-2} \ee^\Phi)}$ or
$\| \vec{s} \|_{(H^1_0(\sU;\ee^\Phi))'}$.

The \emph{weighted Korn inequality} is slightly more involved (see
also~\cite{carrapatoso2020weighted}). Consider a vector field
$\vec{u} : \sU \to \R^n$ in $H^1(\sU;\ee^\Phi)$ satisfying the average
conditions~\eqref{eq:korn-average-conditions}.  To
prove~\eqref{eq:korn-statement}, we work on
$\vec{v} = \vec{u} \ee^\Phi \in L^2(\sU;\ee^{-\Phi})$ and write
\begin{align*}
  \left\| \nabla ^\Phi \vec{u}
  \right\|_{L^2(\sU ; \ee^\Phi)} = \left\| \nabla  \vec{v}
  \right\|_{L^2(\sU ;\ee^{-\Phi})} ^2
  & = \left\| \nabla
    ^{\sym} \vec{v}
    \right\|_{L^2(\sU ;\ee^{-\Phi})} ^2 + \left\| \nabla
    ^{\anti}\vec{v} \right\|_{L^2(\sU ;
   \ee^{-\Phi})} ^2 \\
  & \lesssim \left\| \nabla
    ^{\sym} \vec{v}
    \right\|_{L^2(\sU ;\ee^{-\Phi})} ^2 + \left\| \nabla \nabla
    ^{\anti}\vec{v} \right\|_{(H^1_0(\sU ;\ee^{-\Phi}))'} ^2
  \\
  & \lesssim \left\| \nabla
    ^{\sym} \vec{v}
    \right\|_{L^2(\sU ;\ee^{-\Phi})} ^2 + \left\| \nabla \nabla
    ^{\sym}\vec{v} \right\|_{(H^1_0(\sU ;
   \ee^{-\Phi}))'} ^2 \\
  &  \lesssim \left\| \nabla^\sym \vec{v}
    \right\|_{L^2(\sU ;\ee^{-\Phi})} ^2
    = \left\| \nabla^{\Phi,\sym} \vec{u} \right\|_{L^2(\sU;\ee^\Phi)}
\end{align*}
where we have used the previous Poincaré-Lions inequality since
$\int_\sU (\nabla ^{\anti}\vec{v} )\ee^{-\Phi} =0$ and
we have used the fact that all second-order derivatives are
linear combination of those in
$\nabla \nabla ^{\sym}$. This
proves~\eqref{eq:korn-statement} under the average
conditions~\eqref{eq:korn-average-conditions}.

In \cref{thm:boundary-korn} the average
conditions~\eqref{eq:korn-average-conditions} are replaced by boundary
conditions. Here we write
\begin{equation*}
  \vec{v} = \vec{v}_0 + \sum_{i} b_i e^i
  + \sum_{1\le i < j \le n} b_{ij}\,e^{ij} (z-p)
\end{equation*}
such that $\int_{\sU} \vec{v}_0 \dd x = 0$ and
$\int_{\sU} (\nabla^\anti \vec{v}_0) \ee^{-\Phi} = 0$. By the
non-penetration boundary condition, we find for \(\chi^k\) from
\cref{thm:boundary-korn} that
$\int_{\sU} \nabla \cdot (\chi^k \vec{v})=0$ for $k=1,\dots,n$ and
\begin{equation*}
  \int_{\sU} \nabla \cdot
  \left(\chi^k \sum_{i} b_i e^i
    + \chi^k \sum_{1\le i < j \le n} b_{ij}\, e^{ij} (z-p)
  \right)
  = b_k.
\end{equation*}
Regarding \(\vec{v}_0\) we have
\begin{equation*}
  \begin{split}
    \left|
      \int_\sU \nabla \cdot (\chi^k \vec{v}_0)
    \right|
    &\le
    \int_{\sU} |\nabla \chi^k \cdot \vec{v}_0|
    +
    \int_{\sU} |\chi^k  \nabla \cdot \vec{v}_0| \\
    &\le \| \nabla \chi^k \|_{L^2(\sU;\ee^{\Phi})}
    \| \vec{v}_0 \|_{L^2(\sU;\ee^{-\Phi})}
    + \| \chi^k \|_{L^2(\sU;\ee^{\Phi})}
    \| \nabla \cdot \vec{v}_0 \|_{L^2(\sU;\ee^{-\Phi})}\\
    &\lesssim
    \| \nabla^\sym \vec{v}_0 \|_{L^2(\sU;\ee^{-\Phi})},
  \end{split}
\end{equation*}
where we used the Poincaré inequality and the control by the
symmetric gradient for the zero average of the antisymmetric
gradient from the first part. This shows
$|b_k| \lesssim \| \nabla^\sym \vec{v}
\|_{L^2(\sU;\ee^{-\Phi})}$, and we likewise control $b_{ij}$
by testing against $\chi^{ij}$.

\section{The quantitative trajectorial method}
\label{sec:abstract}

In this section we prove \cref{theo:main}.

\subsection{Transport mapping condition}
\label{sec:transport-mapping-condition}

In \hypref{h:tcc}, we only assume a lower bound on the transport
condition. As a first step, we show that \hypref{h:tcc} implies the
following \emph{transport mapping condition} which is then used in the
following trajectorial method.

Consider the evolution~\eqref{eq:tcc-decay}, and introduce the
semigroup \(\cG_t = \ee^{t(\cT+\cB)}\) with boundary condition
\(\cR^T\). Remembering that \(\cB=0\) or $\cB=\sigma \cL^T$ in
\hypref{h:tcc}, observe that \(\cG_t\) is a semigroup in \(L^\infty\)
and has stationary state \(1\) and \(\cG_t^*\) has stationary state
\(f_\infty^2\).

\sethypothesistag{\ref*{h:tcc}M}
\begin{hypothesis}[Transport mapping condition]
  \phantomsection
  \label{h:tmc}

  Let $\cB$ be an operator such that the evolution
  \begin{align}
    \label{eq:tmc-semmigroup}
    \begin{dcases}
      \partial_t \varphi - \cT \varphi -\cB \varphi = 0
      &\text{in } \Omega \times \sV, \\
      \gamma_+ \varphi = \cR^T \tilde\gamma_- \varphi
      &\text{on } \Gamma_+
    \end{dcases}
  \end{align}
  defines a semigroup $\cG_t$ in $L^\infty$ with the constant $1$ as
  stationary state and \(f_\infty^2\) as stationary state of
  \(\cG_t^*\).

  Assume a connected domain $\Sigma \subset \Omega$ so that
  $(\Sigma,\phi)$ is $\epsilon$-regular for some $\epsilon >0$,
  $\inf_{x \in \Sigma} \sigma(x) > 0$ and a non-negative
  $\chi \in L^\infty(\Omega \times \mathcal \sV)$ with
  $\supp \chi \subset \Sigma \times \sV$. Assume that there exists a
  time $T>0$ and functions
  $\psi,\tilde \psi : [0,T] \times \Omega \times \sV \to \R$ such that
  \begin{equation}
    \label{eq:tmc-mapping}
    \left|
      1 - \int_0^T \cG_t (\psi_t + \tilde \psi_t) \dd t
    \right|
    \le \frac 18.
  \end{equation}

  Moreover, it satisfies the following compatibility conditions
  \begin{itemize}
  \item The sources \(\psi_t\) is controlled as
    \begin{equation}
      \label{eq:tmc:psi-bound}
      \forall \, t \in [0,T], \quad |\psi_t| \lesssim \chi\, w
    \end{equation}
    and \(\tilde \psi_t\) satisfies for all \(t \in [0,T]\) the estimates
    \begin{equation}\label{eq:tmc:psi-tilde-bounds}
      \left\{
        \begin{lgathered}
          \int_{\Omega \times \sV} f^2 \tilde \psi_t \dd \mu
          \lesssim  \| f \|_{L^2(\Omega\times\sV,\mathrm{d}\mu)}
          \left| \int_{\Omega \times \sV} \sigma f \cL f \dd \mu \right|^{\frac12}
          + \left| \int_{\Omega \times \sV} \sigma f \cL f \dd \mu \right|,\\
          \left| \int_{\Omega \times \sV} f f_\infty \tilde \psi_t \dd \mu \right|
          \lesssim \left| \int_{\Omega \times \sV} \sigma f \cL f \dd \mu \right|^{\frac12}.
        \end{lgathered}
      \right.
    \end{equation}
  \item For all \(t \in [0,T]\) the function
    \begin{equation}
      \label{eq:tmc:def-g}
      G_t = \int_{\tau=0}^{T-t} \cG_\tau (\psi_{t+\tau}+\tilde
      \psi_{t+\tau}) \dd \tau
    \end{equation}
    is uniformly bounded in \(L^\infty(\Omega \times \sV)\) and
    satisfies
    \begin{equation}
      \label{eq:g-bound-2}
      \int_{\Omega \times \sV}
      [\cB^*(f^2) - 2 f \sigma \cL f]\, G_t
      \dd \mu \lesssim \| f \|_{L^2(\Omega\times\sV,\mathrm{d}\mu)} \,
      \left| \int_{\Omega \times \sV} \sigma f \cL f \dd \mu \right|^{\frac12}
      + \left| \int_{\Omega \times \sV} \sigma f \cL f \dd \mu \right|
    \end{equation}
    and
    \begin{equation}
      \label{eq:g-bound-1}
      \left|
        \int_{\Omega \times \sV}
        [\cB^*(f f_\infty) - f_{\infty} \sigma \cL f]\, G_t
        \dd \mu  \right|
      \lesssim \left| \int_{\Omega \times \sV} \sigma f \cL f \dd \mu \right|^{\frac12}.
    \end{equation}
  \end{itemize}
\end{hypothesis}

The relation to \hypref{h:tcc} is captured by the following
lemma.
\begin{lemma}
  \label{thm:tmc}
  The hypothesis \hypref{h:tcc} implies \hypref{h:tmc}.
\end{lemma}
\begin{proof}
  By the assumed decay of \eqref{eq:tcc-decay} in \hypref{h:tcc},
  there exists a time $T>0$ such that
  $\|\varphi_T\|_{\infty} \le 1/8$. As the constant $1$ is a
  stationary state, we find by Duhamel that
  \begin{equation*}
    1
    - \int_0^T \cG_t (\mathfrak C \varphi_{T-t} + \chi w \varphi_{T-t}) \dd t
    = \varphi_T.
  \end{equation*}
  We therefore find functions \(\psi\) and \(\tilde \psi\) satisfying
  \eqref{eq:tmc-mapping} by setting
  \begin{equation*}
    \psi_t = \chi w \varphi_{T-t}
    \text{ and }
    \tilde \psi_t = \mathfrak C \varphi_{T-t}.
  \end{equation*}

  By \hypref{h:tcc} \(\varphi_t\) is bounded in \(L^\infty\) so that
  the choice of \(\psi_t\) implies \(|\psi_t| \lesssim \chi w\). The
  bounds on \(\tilde \psi\) in \eqref{eq:tmc:psi-tilde-bounds} follows
  directly from the assumption \eqref{eq:creation-bound} in
  \hypref{h:tcc}.

  Then note again by Duhamel that
  \begin{equation*}
    G(t,x,v)
    = \int_{\tau=0}^{T-t} \cG_\tau (\psi_{t+\tau}+\tilde
    \psi_{t+\tau}) \dd \tau
    = 1 - \varphi_{T-t}.
  \end{equation*}
  Hence by the boundedness of \(\varphi_t\), we find that \(G_t\) is
  uniformly bounded in \(L^\infty\). The last assumptions
  \eqref{eq:g-bound-2} and \eqref{eq:g-bound-1} then follow directly
  from \eqref{eq:diff-bound-2}.
\end{proof}

\subsection{Decay criterion}
\label{sec:preliminary}

We now prove the decay claimed in \cref{theo:main}, where we replace
\hypref{h:tcc} by \hypref{h:tmc} using \cref{thm:tmc} from the
previous subsection.  Note that \(\tilde \psi\) is only appearing in
the case when \(\mathfrak C\not = 0\) for capturing hypoelliptic
effects. Hence at first reading, \(\tilde \psi\) could be ignored.

Assume
\hypref{h:geom}-\hypref{h:eq}-\hypref{h:local}-\hypref{h:bdd}-\hypref{h:tmc}-\hypref{h:coe}
and consider $f \in L^\infty (\R_+;L^2(\Omega \times \sV, \dd \mu))$ a
real-valued solution to~\eqref{eq:gen} with zero global average
$\int_{\Omega \times \sV} f_{\init} =0$ (the complex case follows by
linearity). Its associated equilibrium is zero. The dissipation
is given by
\begin{align}
  \label{eq:dissipation}
  \cD(f) := -\frac{\mathrm{d}}{\mathrm{d}t} \| f_t \|_{L^2(\Omega
  \times \sV, \dd \mu)}^2 = - 2 \int_{\Omega \times \sV} \sigma f \cL f \dd \mu
  - \int_{\Gamma_+} \left[
  \left( \cR \gamma_+ f\right)^2 - \left( \gamma_+ f
  \right)^2 \right] \dd \nu
  \ge 0.
\end{align}

The following decay criterion is standard in semigroup theory and is
used at least since~\cite{MR1908664} in kinetic theory. Assume there
is $\eta \in (0,1)$ so that
\begin{align}
  \label{eq:crit}
  \int_0 ^T \cD(f_t) \dd t \ge \eta \left\| f_{\init}
  \right\|_{L^2(\Omega \times \sV, \dd \mu)} ^2
\end{align}
for a solution $f_t \in L^2(\mu)$, then
\begin{align*}
  \left\| f_t \right\|_{L^2(\Omega \times \sV, \dd \mu)} ^2 \le C
  \ee^{-\Lambda t} \left\| f_{\init} \right\|_{L^2(\Omega \times
  \sV, \dd \mu)} ^2
\end{align*}
with $C := (1-\eta)^{-1}>1$ and
$\Lambda := - \frac{\ln(1-\eta)}{T} >0$. We now
prove~\eqref{eq:crit} for some $\eta \in (0,1)$.

\subsection{Following trajectories}
\label{sec:following-trajectories}

The underlying idea for the first step is to use the transport to
relate the initial condition
\(\| f_\init \|^2_{L^2(\Omega\times\sV,\mathrm{d}\mu)}\) to the value
of \(f^2\) over \([0,T]\times\Omega\times\sV\) where it can be
controlled by the dissipation; in particular, to \(f^2\) in
\([0,T]\times\Omega\times\Sigma\). To achieve this idea, note that
\(f^2\) satisfies the transport equation up to error terms and that
the semigroup \(\cG_t\) is solving up to the effect of \(\cB\) the dual
transport equation.

To implement this idea, we use \eqref{eq:tmc-mapping} to relate
\(\| f_\init \|^2_{L^2(\Omega\times\sV,\mathrm{d}\mu)}\) to
\(\int f^2 (\psi+\tilde \psi) \dd \mu\).  By the evolution equations
\eqref{eq:gen} and \eqref{eq:tmc-semmigroup} we find for
$s \in [0,t]$ and $t \in [0,T]$ that
\begin{equation}
  \begin{split}
    &\frac{\mathrm{d}}{\mathrm{d}s} \int_{\Omega \times \sV}
    f_s^2\; \cG_{t-s} (\psi_t{+}\tilde \psi_t)  \dd \mu\\
    &= - \int_{\Omega \times \sV}
    \cT \Big[ f_s^2  \cG_{t-s} (\psi_t{+}\tilde \psi_t)
    \Big] \dd \mu
    + \int_{\Omega \times \sV}
    \Big(2f_s \sigma \cL f_s
    - \cB^*(f_s^2)
    \Big)  \cG_{t-s} (\psi_t{+}\tilde \psi_t) \dd \mu.
  \end{split}
\end{equation}
Using $\cT f_\infty = \cT f_\infty ^{-1} =0$, the transport term can
be integrated and then yields boundary terms. Further integrating $s$
over \([0,t]\) and then $t$ over \([0,T]\) yields
\begin{align*}
  & \int_0 ^T \int_{\Omega \times \sV}
    f_t^2\; (\psi_t{+}\tilde \psi_t) \dd \mu \dd t
    =
    \int_0 ^T \int_{\Omega \times \sV}
    f_{\init}^2\;
    \cG_t (\psi_t{+}\tilde \psi_t) \dd \mu \dd t  \\
  & + \int_0 ^T \int_0 ^t \int_{\Omega \times \sV}
    \Big(2 f_s \sigma \cL f_s - \cB^*(f_s^2)\Big)\,
    \cG_{t-s} (\psi_t {+} \tilde \psi_t) \dd \mu \dd s \dd t \\
  & + \int_0 ^T \int_0 ^t \int_{\Gamma_+}
    \left[ \left(\cR \gamma_+ f_s \right)^2
    \tilde\gamma_-\left( \cG_{t-s}(\psi_t{+}\tilde\psi_t) \right)
    - (\gamma_+ f_s)^2 \cR^T\left(\tilde \gamma_- \cG_{t-s}(\psi_t{+}\tilde
    \psi_t)\right)
    \right]
    \dd \nu \dd s \dd t \\
  & \quad =: I_1 + I_2 + I_3.
\end{align*}

The first term $I_1$ gives a control on \(\|f_{\init}\|\) by
\eqref{eq:tmc-mapping} in \hypref{h:tmc} as
\begin{align*}
  I_1   = \int_0 ^T \int_{\Omega \times \sV}
  f_{\init}^2
  \cG_t (\psi_t{+}\tilde \psi_t)  \dd \mu \dd t
  \ge \frac{7}{8} \int_{\Omega \times \sV}
  f_{\init}(x,v)^2
  \dd \mu \dd t
  = \frac{7}{8} \| f_{\init} \|^2_{L^2(\Omega \times \sV, \dd \mu)}.
\end{align*}
The second term $I_2$ can be written with \(G_t\) of
\eqref{eq:tmc:def-g} and bounded by \eqref{eq:g-bound-2} as
\begin{align*}
  -I_2
  = \int_0 ^T \int_{\Omega \times \sV}
  \Big(\cB^*(f_s^2) - 2 f_s \sigma \cL f_s\Big)
  G_s \dd \mu \dd s
  \lesssim \int_0^T
  \| f_s \|_{L^2(\Omega\times\sV,\dd \mu)}
  \sqrt{\cD(f_s)}\, \dd s
  + \int_0^T \cD(f_s)\, \dd s
\end{align*}
The third term $I_3$ is controlled by \eqref{eq:boundary-bound} in
\hypref{h:bdd} as
\begin{align*}
  -I_3
  = \int_0^T \int_{\Gamma_+}
  \left[
  (\gamma_+ f_s)^2
  \cR^T(\tilde \gamma_- G_s)
  - (\cR \gamma_+ f_s)^2 \tilde \gamma_- G_s
  \right]
  \dd \nu \dd s
  \le C_r \| G_s \|_{L^\infty([0,T]\times\Omega\times \sV)}
  \int_0 ^T \cD(f_s) \dd s.
\end{align*}

For the LHS we keep at this stage the contribution of
\(\int_0^T \int_{\Omega\times\sV} f_t^2 \psi_t \dd \mu \dd t\) and
note for the other term by \eqref{eq:tmc:psi-tilde-bounds} that
\begin{equation}
  \label{eq:traj:psi-tilde-bound}
  \int_0 ^T \int_{\Omega \times \sV}
  f_t^2\; \tilde \psi_t \dd \mu \dd t
  \lesssim \int_0^T \| f_t \|_{L^2(\Omega\times\sV,\dd\mu)}
  \sqrt{\cD(f_t)} \dd t
  + \int_0^T \cD(f_t) \dd t.
\end{equation}

As the evolution of \(f\) is dissipative, we have the trivial bound
\(\| f_t \|_{L^2(\Omega\times\sV, \dd \mu)} \le \| f_{\init}
\|_{L^2(\Omega\times\sV, \dd \mu)}\) for \(t \in [0,T]\) which we can
use to further estimate the bound on \(-I_2\) and
\eqref{eq:traj:psi-tilde-bound}. Combining the different parts, we
therefore arrive at
\begin{align}
  \label{eq:following}
  \frac 34
  \| f_{\init} \|^2_{L^2(\Omega \times \sV, \dd \mu)}
  -
  \int_0 ^T \int_{\Omega \times \sV}
  f_t^2\, \psi_t \dd \mu \dd t
  \lesssim
  \int_0 ^T \cD(f_t) \dd t.
\end{align}

\subsection{Removing the local and global averages}

Using the spatial averages
\begin{equation*}
  \vAvg{f_t}(x) =
  \vAvg{f}(t,x) := \int_{\sV} f(t,x,v) \dd v, \quad
  \langle \psi_t \rangle_M (x) :=
  \int_{\sV} \psi(t,x,v) M(v) \dd v,
\end{equation*}
split the term remaining to control by \eqref{eq:loc-coerc} of
\hypref{h:local} as
\begin{equation}
  \label{eq:split-local}
  \begin{aligned}
    \int_0 ^T \int_{\Omega \times \sV} f_t(x,v)^2
    \psi_t(x,v) \dd \mu \dd t
    &\le 2
    \int_0 ^T \int_{\Omega \times \sV} \langle f_t
    \rangle^2 M \psi_t \dd \mu \dd t
    + 2
    \int_0 ^T \int_{\Omega \times \sV} \big[ f_t -
    \langle f_t \rangle M \big]^2 \psi_t \dd \mu \dd t \\
    & \le 2
    \int_0 ^T \int_{\Omega} \langle f_t \rangle^2
    \langle \psi_t \rangle_M \ee^\phi \dd x \dd t
      + 2 \lambda_1 \int_0^T
      \left\|
      \frac{\psi_t}{\sigma w}
      \right\|_{\infty}
      \cD(f_t)
      \dd t.
  \end{aligned}
\end{equation}
and by assumption \eqref{eq:tmc:psi-bound} we know that
\(\psi_t/(\sigma w)\) is uniformly bounded. Using the global
projection
\begin{equation*}
  \gAvg{f}_\psi :=
  \frac{1}{m_\psi}
  \int_0 ^T \int_{\Omega} \vAvg{f_t}(x) \langle \psi_t \rangle_M (x)
  \dd x \dd t,\qquad
  m_\psi :=
  \int_0 ^T \int_{\Omega} \vAvg{\psi_t}_M(x)\, \ee^{-\phi(x)} \dd x \dd t,
\end{equation*}
we further split the integral over \(\vAvg{f_t}^2\) as
\begin{equation}
  \label{eq:split-global}
  \int_0 ^T \int_{\Omega} \langle f_t \rangle^2
  \langle \psi_t \rangle_M \ee^\phi \dd x \dd t
  - m_\psi \gAvg{f}_\psi^2
  \le \int_0 ^T \int_{\Omega} \left[ \vAvg{f_t} -
    \gAvg{f}_\psi\, \ee^{-\phi} \right]^2 \langle \psi_t
  \rangle_M \ee^\phi \dd x \dd t.
\end{equation}

\subsection{Control of the local average}
\label{subsec:am}

Let us prove
\begin{align}
  \label{eq:am}
  \int_{\Sigma_T}
  \left[\vAvg{f_t} - \gAvg{f}_\psi\ee^{-\phi} \right]^2
  \langle \psi_t \rangle_M \ee^\phi \dd z
  \le \var' \int_{\Sigma_T\times \sV}
  f_t(x,v)^2 \psi_t(x,v) \dd \mu \dd t
  + C(\var') \int_{0} ^T \cD(f_t) \dd t
\end{align}
for any $\var'>0$ and some corresponding constant $C(\var')>0$, and
with the notation $z:=(t,x)$ and $\Sigma_T := (0,T) \times
\Sigma$.

For proving \eqref{eq:am}, consider
\begin{equation}
  \label{eq:def-g}
  g := \langle f \rangle \langle \psi \rangle_M - \gAvg{f}_\psi \langle
  \psi \rangle_M \ee^{-\phi}
\end{equation}
and note that \(g\) has zero mass and \(g \in L^2(\Sigma_T;\ee^\phi)\)
with the bound
\begin{equation}
  \label{eq:control-g}
  \int_{\Sigma_T} |g|^2 \ee^\phi \dd z
  \lesssim \int_{\Sigma_T \times \sV}
  f^2 \psi \dd \mu \dd t
  + \int_{0} ^T \cD(f_t) \dd t
\end{equation}
which follows from
\begin{align*}
  \langle f_t \rangle (x) \langle \psi_t \rangle_M(x)
  & = \int_{v,v_*\in \sV} f_t(x,v) \psi_t(x,v_*) M(v_*) \dd v \dd
    v_* \\
  & = \int_{v_* \in \sV} f_t(x,v_*) \psi_t(x,v_*) \dd v_* +
    \int_{v_* \in \sV} \big[ f_t(x,v_*) - \langle f_t \rangle
    M(v_*) \big] \psi_t(x,v_*) \dd v_*.
\end{align*}
For this \(g\) we can therefore apply the divergence inequality of
\cref{theo:ineq} which provides a solution
$\vec{F} = \vec{F}(z) : \Sigma_T \to \R^{1+d}$
to~\eqref{eq:divergence}. Using the vector field \(\vec{F}\), we then find
\begin{align*}
  \int_{\Sigma_T} \left[ \langle
  f \rangle -
  \gAvg{f}_\psi
  \ee^{-\phi} \right]^2 \langle \psi \rangle_M\, \ee^\phi \dd z
  = \int_{\Sigma_T} \left( \nabla _z \cdot \vec{F}
  \right)  \left[ \langle
  f \rangle \ee^\phi - \gAvg{f}_\psi \right] \dd z
  = - \int_{\Sigma_T} \vec{F}
  \cdot \nabla_z \left( \langle
  f \rangle \ee^\phi  \right) \dd z
\end{align*}
where the integration by part has no boundary term since
$\vec{F} \in H^1_0(\Sigma_T;\ee^\phi)$. Let us denote
$\partial_0 = \partial_t$ and $\partial_i = \partial_{x_i}$ for
$i =1,\dots,d$ and prove
$\partial_i ( \langle f \rangle \ee^\phi ) = K_i + \sum_{j=0} ^d
\partial_j J_{ij}$, $i,j=0,\dots,d$, with
\begin{align}
  \label{eq:claim-am}
  \int_{\Sigma_T} \left( \lfloor \nabla \phi \rceil^{-2} K_i(t,x)^2 + |J_i(t,x)|^2
  \right) \ee^{-\phi} \dd z
  \lesssim
  \int_0^T \cD(f_t) \dd t.
\end{align}
Let us first accept~\eqref{eq:claim-am} and conclude the proof
of~\eqref{eq:am} (using~\eqref{eq:divergence}, \eqref{eq:control-g} and
$|\nabla^2 \phi| \lesssim 1+|\nabla \phi|$):
\begin{align*}
  & \int_{\Sigma_T}  \left[ \langle
    f \rangle -
    \gAvg{f}_\psi
    \ee^{-\phi} \right]^2 \langle \psi \rangle_M \ee^\phi \dd z
    = - \int_{\Sigma_T}  \vec{F}
    \cdot \nabla_z \left( \langle
    f \rangle \ee^\phi  \right) \dd z \\
  & \hspace{1cm} = - \sum_{i=0} ^d \int_{\Sigma_T}
    \vec{F}_i(z) K_i(z) \dd z
    + \sum_{i,j=0} ^d \int_{\Sigma_T}
    \partial_j \vec{F}_i (z) J_{ij}(z) \dd z  \\
  & \hspace{1cm} \lesssim \left(
    \int_{\Sigma_T} \left| \vec{F} \right|^2 \lfloor
    \nabla_x \phi \rceil^2 \ee^\phi \dd z + \int_{\Sigma}
    \left| \nabla_z \vec{F} \right|^2 \ee^\phi \dd z
    \right)^{\frac12}
    \left( \int_{\Sigma_T} \left( \lfloor \nabla_x
    \phi \rceil^{-2} K_i(z)^2 +
    |J_i(z)|^2 \right)\ee^{-\phi} \dd z \right)^{\frac12}  \\
  & \hspace{1cm} \lesssim \left( \int_{\Sigma_T}
    f^2 \psi \dd \mu \dd t \right)^{\frac12}
    \left[ \int_0^T \cD(f_t) \dd t \right]^{\frac 12}
    + \int_0^T \cD(f_t) \dd t.
\end{align*}
which proves~\eqref{eq:am} by splitting the product into squares
adequately.

Let us now prove~\eqref{eq:claim-am}. Define
$\varphi_i \in C^\infty_c(\sV)$, $i=0,\dots,d$ so that, denoting
$v_0=1$, $\int_{\sV} \varphi_i(v) v_j M(v) \dd v = \delta_{ij}$. Then
$\cT f_\infty = 0$ from \hypref{h:eq} implies
\begin{equation*}
  \int_{\sV} \Big\{ \left( \partial_t + v \cdot \nabla_x -
    \nabla_x \phi \cdot \nabla_v \right) \big[ \langle f
  \rangle M \big] \Big\} \varphi_i\ee^{\phi} \dd v =
  \partial_i \left( \langle f
    \rangle \ee^\phi \right)
\end{equation*}
so that the evolution~\eqref{eq:gen} implies
\begin{equation*}
  \partial_i \left( \langle f \rangle \ee^\phi \right) =
  \int_{\sV} \Big\{ \left( \partial_t + v \cdot \nabla_x -
    \nabla_x \phi \cdot \nabla_v \right) \big[ \langle f
  \rangle M - f \big] \Big\} \varphi_i\ee^{\phi} \dd v +
  \int_{\sV} (\sigma \cL f)\varphi_i\ee^{\phi} \dd v.
\end{equation*}
The terms on the RHS can be collected as
\begin{align*}
  & \int_{\sV} \Big\{ \left( \partial_t
    + v \cdot \nabla_x - \nabla_x \phi \cdot \nabla_v \right)
    \big[ \langle f \rangle M - f \big] \Big\} \varphi_i
   \ee^{\phi} \dd v  +
    \int_{\sV} (\sigma \cL f)\varphi_i\ee^{\phi} \dd v
    = K_i +
    \sum_{j=0} ^d \partial_j J_{ij}
\end{align*}
where
\begin{align*}
  \begin{dcases}
    K_i(t,x) := \int_{\sV}
    \left[ \langle f \rangle M - f \right] \nabla_x \phi \cdot
    \nabla_v \varphi_i
   \ee^{\phi} \dd v +
    \int_{\sV} (\sigma \cL f)\varphi_i\ee^{\phi} \dd v, \\
    J_{i0}(t,x) := \int_{\sV}
    \left[ \langle f \rangle M - f \right]  \varphi_i
   \ee^{\phi} \dd v,  \\
    J_{ij}(t,x) := \int_{\sV}
    \left[ \langle f \rangle M - f \right] v_j  \varphi_i
   \ee^{\phi} \dd v.
  \end{dcases}
\end{align*}

On the one hand, \hypref{h:local} implies as
\(\varphi_i \in \domain{\cL^*}\) that
\begin{align*}
  \int_{\Sigma_T}  \left| \int_{\sV} (\sigma \cL
  f_t)(x,v) \varphi_i(v) \ee^\phi \dd v \right|^2\ee^{-\phi} \dd z
  \lesssim
  \int_0^T \cD(f_t) \dd t
\end{align*}
and on the other hand $\vAvg{f} M - f$ is controlled by
assumption~\eqref{eq:loc-coerc} as $\sigma \gtrsim 1$ on $\Sigma$ so
that
\begin{align*}
  \forall \, i,j=0,\dots,d, \quad
  & \int_{\Sigma_T} \lfloor \nabla_x
    \phi \rceil^{-2} K_i^2\ee^{-\phi} \dd z +
    \int_{\Sigma_T} J_{ij}^2\ee^{-\phi} \dd z \\
  & \lesssim
    \int_0 ^T \int_{\Omega} \sigma \left| \int_{\sV}
    \left| \langle f \rangle M - f \right| \dd v \right|^2
    \ee^\phi \dd z
    \lesssim \int_0^T \cD(f_t) \dd t.
\end{align*}

The combination of~\eqref{eq:split-local},~\eqref{eq:split-global} and~\eqref{eq:am} yields by
choosing $\var'$ small enough:
\begin{align}
  \label{eq:split2}
  \int_0 ^T \int_{\Omega \times \sV} f_t(x,v)^2
  \psi_t(x,v) \dd \mu \dd t
  - 4 \gAvg{f}_\psi ^2 m_\psi
  \lesssim
  \int_0^T \cD(f_t) \dd t.
\end{align}

Together with~\eqref{eq:following} it implies
\begin{align}
  \label{eq:reduc}
  \frac 34
  \| f_{\init}\|^2_{L^2_{x,v}(\mu)}
  - 4 m_\psi \gAvg{f}^2
  \lesssim
  \int_0 ^T \cD(f_t) \dd t.
\end{align}

\subsection{Control of the global average}
\label{sec:control-global-average}

We decompose the global average as
\begin{align*}
  \gAvg{f}_\psi
  & = \frac{1}{m_\psi}
    \int_0 ^T \int_{\Omega \times \sV}
    \vAvg{f_t}(x)\, \psi_t(x,v) M(v) \dd x \dd v \dd t \\
  & = \frac{1}{m_\psi}
    \int_0 ^T \int_{\Omega \times \sV}
    \big[ \vAvg{f_t} (x) M(v) - f_t(x,v) \big]
    \psi_t(x,v) \dd x \dd v \dd t \\
  & \qquad + \frac{1}{m_\psi}
    \int_0 ^T \int_{\Omega \times \sV}
    f_t(x,v)\, \psi_t(x,v) \dd x \dd v \dd t
    =: A_1 + A_2.
\end{align*}
The first term $A_1$ is controlled using~\eqref{eq:loc-coerc} by
\begin{align*}
  |A_1| \lesssim \left( \int_0 ^T \int_{\Omega \times \sV} \sigma
  \left| \langle f \rangle M - f \right|^2
  w \dd \mu \dd t \right)^{\frac12}
  \lesssim
  \left[ \int_{0} ^T \cD(f_t) \dd t \right]^{\frac 12}.
\end{align*}
The second term \(A_2\) is controlled similarly to the idea in
\cref{sec:following-trajectories}, except that we can work directly
with \(f\). Here we find for \(s \in [0,t]\) and \(t \in [0,T]\) that
\begin{equation}
  \begin{split}
    \frac{\mathrm{d}}{\mathrm{d}s} \int_{\Omega \times \sV}
    f_s\; f_\infty \cG_{t-s} (\psi_t{+}\tilde \psi_t)  \dd \mu
    &= - \int_{\Omega \times \sV}
    \cT \Big[ f_s f_\infty \cG_{t-s} (\psi_t{+}\tilde \psi_t)
    \Big] \dd \mu \\
    &\quad+ \int_{\Omega \times \sV}
    \Big(\sigma \cL f_s
    - \frac{1}{f_\infty} \cB^*(f_s f_\infty)
    \Big) f_\infty \cG_{t-s} (\psi_t{+}\tilde \psi_t) \dd \mu.
  \end{split}
\end{equation}
Following the same calculation as in \cref{sec:following-trajectories}
we arrive at
\begin{equation}
  \label{eq:duhamel-1}
  \begin{split}
    \int_0^T \int_{\Omega\times \sV} f_t (\psi_t{+}\tilde \psi_t) \dd x \dd v
    &= \int_{\Omega\times \sV} f_{\init}
    \left(\int_0^T \cG_t(\psi_t{+}\tilde \psi_t)\dd t\right)
    \dd x \dd v\\
    &+ \int_0^T \int_{\Omega\times \sV}
    \Big(\sigma \cL f_s
    - \frac{1}{f_\infty} \cB^*(f_s f_\infty)
    \Big) f_\infty G_s \dd \mu \dd s\\
    &+ \int_0^T \int_{\Gamma_+}
    \left[
      \cR(\gamma_+f_s) f_{\infty} \tilde \gamma_- G_s
      - \gamma_+ f_s\, f_\infty \cR^T(\tilde \gamma_- G_s)
    \right] \dd \nu \dd s.
  \end{split}
\end{equation}
The last term vanishes by the definition of \(R^T\) as the
appropriately weighted adjoint.  As \(f_\init\) has mass zero, we find
by \eqref{eq:tmc-mapping} that
\begin{equation*}
  \left|
    \int_{\Omega\times \sV} f_{\init}
    \left(\int_0^T \cG_t(\psi_t{+}\tilde \psi_t)\dd t\right)
    \dd x \dd v
  \right|
  \le \frac 18 \| f_{\init} \|_{L^2(\mu)}.
\end{equation*}
Hence \eqref{eq:duhamel-1} shows by \eqref{eq:tmc:psi-tilde-bounds},
\eqref{eq:g-bound-1} for \(A_2\) that
\begin{equation*}
  \left|
    \int_0^T \int_{\Omega\times \sV} f_t \psi_t \dd x \dd v
  \right|
  - \frac 18 \| f_{\init} \|_{L^2(\mu)}
  \lesssim \int_0^T \sqrt{\cD(f_t)} \dd t.
\end{equation*}
We therefore arrive at the bound
\begin{equation*}
  4 m_\psi \gAvg{f}_\psi^2
  - \frac{1}{8m_\psi} \| f_\init \|_{L^2(\mu)}^2
  \lesssim \int_0^T \cD(f_t) \dd t.
\end{equation*}

For the final conclusion, we note that \(m_\psi\) is close to \(1\)
due to \eqref{eq:tmc-mapping}. As
\(\int_{\Omega\times \sV} \tilde \psi f_\infty \dd x \dd v = 0\) due
to \eqref{eq:tmc:psi-tilde-bounds} and \(f_\infty^2\) is a stationary
state of \(\cG^*_t\) we find that
\begin{equation*}
  m_\psi
  =
  \int_0 ^T \int_{\Omega \times \sV}
  \psi_t\, f_\infty \dd x \dd v \dd t
  =
  \int_0 ^T \int_{\Omega \times \sV}
  (\psi_t+\tilde \psi_t) f_\infty \dd x \dd v \dd t
  =
  \int_0 ^T \int_{\Omega \times \sV}
  \cG_t(\psi_t + \tilde \psi_t) f_\infty \dd x \dd v \dd t
\end{equation*}
so that \(7/8 \le m_\psi \le 9/8\). We therefore arrive at the bound
\begin{equation}
  \label{eq:average-psi}
  4m_\psi \gAvg{f}_\psi ^2
  - \frac{1}{7} \| f_{\init} \|^2_{L^2(\mu)}
  \lesssim
  \int_0^T \cD(f_t) \dd t.
\end{equation}

\subsection{Conclusion}

We combine~\eqref{eq:reduc} and~\eqref{eq:average-psi} to get
\begin{align*}
  \frac 12
  \| f_{\init} \|^2_{L^2(\mu)} \lesssim \int_0 ^T \cD(f_t) \dd t
\end{align*}
which implies the exponential convergence as discussed in
\cref{sec:preliminary}. This concludes the proof of \cref{theo:main}.

\section{Application to concrete equations}
\label{sec:concrete}

In this section we prove the results for applying \cref{theo:main} to
the concrete examples.

\subsection{Proof of local spectral gap \hypref{h:local}}
\label{ss:local-coerc-concrete}

In the described geometric settings, we can directly verify
$\cL M = 0$ and $\cR f_\infty= 0$. For the Fokker-Planck operator
\eqref{eq:FP}, we have the equilibrium measure
$M(v)=(2\pi)^{-d/2} \ee^{-v^2/2}$ for $\sV = \R^d$ or the uniform
probability measure on $\S^{d-1}$ for $\sV = \S^{d-1}$. By the
weight we find directly that
\begin{equation*}
  \int_{\sV} g(v) \cL g(v) \frac{\dd v}{M(v)}
  =
  \begin{dcases}
    -\int_{\sV} |\nabla g|^2 \frac{\dd v}{M(v)}
    &\text{if } \sV = \S^{d-1},\\
    -\int_{\sV}
    |\nabla_v g + v g|^2
    \frac{\dd v}{M(v)}
    &\text{if } \sV =\R^d,
  \end{dcases}
\end{equation*}
so that the spectral gap follows from the Poincaré inequality of the
Gaussian measure \cite{villani-2009-hypocoercivity,dolbeault-volzone-2012-improved-poincare}.

For the linear Boltzmann operator~\eqref{eq:LB} and a given
equilibrium measure $M=M(v)$, the symmetry condition
$ k(v,v_*) M(v_*) = k(v_*,v) M(v) $ corresponds to the detailed
balance and implies that $\cL$ is symmetric. For the dissipation, we
find
\begin{equation*}
  \int_{\sV} g(v) \cL g(v) \frac{\dd v}{M(v)}
  = -\frac{1}{2} \int_{v,v_* \in \sV}
  k(v,v_*) M(v_*)
  \left[
    \frac{g(v)}{M(v)} -
    \frac{g(v_*)}{M(v_*)}
  \right]^2
  \dd v \dd v_*.
\end{equation*}
For the specific kernels of the linear Boltzmann operator, several
conditions for the spectral condition are known
\cite{lods-mouhot-toscani-2008-relaxation-ficks-boltzmann,bisi-canizo-lods-2015-entropy-boltzmann,lods-mokhtar-kharroubi-2017-convergence-boltzmann,canizo-einav-lods-2018-boltzmann}.

Another viewpoint for the spectral gap can be obtained from the
Cheeger's inequality which is another popular tool to establish
spectral gaps
\cite{cheeger-1970-laplacian,chung-2010-four-cheeger,lee-gharan-trevisan-2014-multiway-cheeger,spielman-2015-conductance,de-mondino-2021-sharp-cheeger-k}.

Adapting to our case, denote the size of a set \(A \subset \sV\)
measured by \(M\) as \(|A|_M = \int_{A} M(v) \dd v\). Then the
condition
\begin{equation*}
  \Phi := \inf_{A \subset \sV}
  \frac{\int_{v\in A}\int_{v_*\in A^c} \sqrt{q(v,v_*)M(v)M(v_*)} \dd
    v_* \dd v}{\min(|A|_{M},|A^c|_{M})} > 0
  \text{ with }
  q(v,v_*) = k(v,v_*) M(v_*)
\end{equation*}
implies a spectral gap as a weighted Cheeger's inequality, which
intuitively says that we cannot split the velocity space $\sV$ into
two separate parts between which the mass is equiliberating slowly.

\begin{lemma}[Cheeger's inequality]
  For the linear Boltzmann operator, introduce the Rayleigh
  coefficients
  \begin{equation*}
    R(g) = \frac{-\ip{g}{\cL g}_{L^2(M^{-1})}}{\|g\|_{L^2(M^{-1})}^2}.
  \end{equation*}
  Then the spectral gap
  \begin{equation*}
    \lambda_1 = \inf_{g \perp M} R(g)
  \end{equation*}
  is bounded by
  \begin{equation*}
    2\lambda_1 \ge \Phi^2.
  \end{equation*}
\end{lemma}
\begin{proof}
  Given $g \perp M$, assume wlog that $\|g\|_{L^2(M^{-1})} = 1$. We
  first claim that there are two vectors $g_-$ and $g_+$ with
  disjoint support such that $\max(R(g_-),R(g_+)) \le 2
  R(g)$. Indeed, for $s \in \R$, define the vector $g^s$ by
  $g^s(v) = g(v) + s\, M(v)$ and set
  \begin{equation*}
    g_{-}^s = \min(g^s,0) \text{ and }
    g_{+}^s = \max(g^s,0).
  \end{equation*}
  As $\|g_{-}^s\|_{L^2(M^{-1})}$ varies continuously for changing $s$
  and $\|g_{-}^s\|_{L^2(M^{-1})} \to \infty$ as $s\to-\infty$ and
  $\|g_{-}^s\|_{L^2(M^{-1})} \to 0$ as $s\to+\infty$, we can find some
  $\bar s$ such that $\|g_{-}^{\bar{s}}\|_{L^2(M^{-1})}^2=1/2$. Moreover,
  \begin{equation*}
    \|g_{-}^{\bar{s}}\|_{L^2(M^{-1})}^2
    + \|g_{+}^{\bar s}\|_{L^2(M^{-1})}^2
    = \|g^{\bar{s}}\|_{L^2(M^{-1})}^2
    = 1 + \bar{s}^2 \ge 1,
  \end{equation*}
  where we used that $g \perp M$. Hence we also have that
  $\|g_{+}^{\bar{s}}\|_{L^2(M^{-1})}^2\ge1/2$. From the expression of the
  dissipation
  \begin{equation*}
    \ip{g_{\pm}^{\bar s}}{-\cL g_{\pm}^{\bar s}}
    \le \ip{g^{\bar s}}{-\cL g^{\bar s}}
    = \ip{g}{-\cL g},
  \end{equation*}
  where we used in the last step that $\cL M = 0$. Hence we indeed
  find that $g_{\pm}^{\bar s}$ have disjoint support and
  \begin{equation*}
    \max(R(g_-),R(g_+)) \le 2 R(g).
  \end{equation*}

  By taking the one of $g_\pm^{\bar s}$ with smaller support, we can
  find $\bar g$ such that $R(\bar{g}) \le 2 R(g)$ and
  $|\supp \bar g|_{M} \le 1/2$. By rescaling $\bar{g}$, we may also
  assume that $\sup_{v \in\sV} \bar{g}(v) M^{-1}(v) = 1$.

  Take $t$ as a uniform random variable on the interval $[0,1]$
  and define the random set
  \begin{equation*}
    S_t := \{ v \in \supp \bar g : |\bar{g}(v)|^2 \ge t M^2(v) \}.
  \end{equation*}
  Then we find directly that
  \begin{equation*}
    \E[ |S_t|_{M} ] = \| \bar g \|_{L^2(M^{-1})}^2.
  \end{equation*}
  Also
  \begin{equation*}
    \begin{aligned}
      &\E\left[
        \int_{v \in S_t} \int_{v_*\in S_t^c} \sqrt{q(v,v_*)M(v)M(v_*)} \dd v_* \dd v
      \right]\\
      &= \frac 12 \int_{v,v_* \in \sV}
      \left|
        \left(
          \frac{\bar g(v)}{M(v)}
        \right)^2
        -
        \left(
          \frac{\bar g(v_*)}{M(v_*)}
        \right)^2
      \right|
      \sqrt{q(v,v_*)M(v)M(v_*)} \dd v_* \dd v \\
      &= \frac 12 \int_{v,v_* \in \sV}
      \left|
          \frac{\bar g(v)}{M(v)}
        -
          \frac{\bar g(v_*)}{M(v_*)}
      \right|
      \left|
          \frac{\bar g(v)}{M(v)}
        +
          \frac{\bar g(v_*)}{M(v_*)}
      \right|
      \sqrt{q(v,v_*)M(v)M(v_*)} \dd v_* \dd v \\
      &\le \frac 12
      \left(\int_{v,v_* \in \sV}
      \left(
          \frac{\bar g(v)}{M(v)}
        -
          \frac{\bar g(v_*)}{M(v_*)}
      \right)^2
      q(v,v_*) \dd v_* \dd v\right)^{1/2} \times \\
    & \hspace{3cm}
      \times \left(\int_{v,v_* \in \sV}
        \left( \frac{\bar g(v)}{M(v)} +
          \frac{\bar g(v_*)}{M(v_*)}
      \right)^2
      M(v)M(v_*) \dd v_* \dd v\right)^{1/2} \\
      &\le \sqrt{R(\bar g)}\; \| \bar g \|_{L^2(M^{-1})}^2.
    \end{aligned}
  \end{equation*}
  Hence there exists a $t \in (0,1)$ such that
  \begin{equation*}
    \int_{v \in S_t} \int_{v_*\in S_t^c} \sqrt{q(v,v_*)M(v)M(v_*)} \dd
    v_* \dd v
    \le \sqrt{R(\bar{g})}\, |S_t|_{M}.
  \end{equation*}
  As $|S_t|_{M} \le 1/2$ by construction, this set $S_t$ implies
  \begin{equation*}
    \Phi \le \sqrt{R(\bar g)} \le \sqrt{2 R(g)},
  \end{equation*}
  which yields the claimed bound.
\end{proof}

\subsection{Boundary compatibility}\label{sec:boundary-compatibility}

As already noted in \cite[Eq.~(5)]{MR432101}, it is natural to
derive the boundary interaction like the Boltzmann equation from
some reversible dynamics. Assuming that the boundary has the
final temperature so that the system converges to equilibrium,
this yields the detailed-balance condition
\begin{equation}
  \label{eq:detailed-balance-bdd}
  r(x,v,v_*) f_\infty(x,v_*) = r(x,v_*,v) f_\infty(x,v).
\end{equation}
By the mass conservation \eqref{eq:mass-conservation-bdd}, the
operator $\cR$ can be understood as step for a time-discrete Markov
chain on $\Gamma_+$ and then the detailed-balance
equation~\eqref{eq:detailed-balance-bdd} means that it is a reversible
Markov chain with stationary state $f_\infty$.

For a given $\mathfrak f \in L^2(\Gamma_+,(\vec{n}\cdot v)\dd \nu)$,
the corresponding limiting state of the Markov chain is denoted by
$\Pi \mathfrak f$ and takes the form
\begin{equation}
  \label{eq:equi-bdd}
  (\Pi \mathfrak f)(x,v)
  = \lambda(x,v) f_{\infty}(x,v)
\end{equation}
where \(\lambda(x,v) = \lambda(x,v_*)\) if \(r(x,v,v_*) > 0\).  For
Maxwell boundary conditions \eqref{eq:boundary-concrete} it is
explicitly given by
\begin{equation}
  \label{eq:proj-maxwell}
  (\Pi \mathfrak f)(x,v)
  =
  \begin{cases}
    \mathfrak f(x,v) &\text{if } \alpha(x) = 0, \\
    \sqrt{2\pi} M(v) \int_{(\vec{n}\cdot v_*)}
    \mathfrak f(x,v_*)\, (\vec{n}\cdot v_*) \dd v_* &\text{if } \alpha(x) \in (0,1].
  \end{cases}
\end{equation}
\begin{remark}
  In terms of Markov chains, \eqref{eq:equi-bdd} states that
  \(\Pi \mathfrak f\) is proportional to \(f_{\infty}\) where the
  proportionality factor can be different in different communicating
  classes. For many boundary conditions with a diffusive component,
  e.g.\ Maxwell boundary conditions with \(\alpha>0\), all velocities
  are related so that
  \((\Pi \mathfrak f)(x,v) = \lambda(x) f_{\infty}(x,v)\).  A case
  with different communicating classes would be boundary conditions
  which only thermalise the normal velocity component but keep the
  tangential velocity component unchanged.
\end{remark}

Indeed we can verify \hypref{h:bdd} in this framework, if there exists
either a uniform spectral gap or we have the special algebra
\eqref{eq:proj-maxwell} (in which case $\alpha$ can be
arbitrary). Note that it handles further boundary conditions discussed
in \cite{MR432101}.

\begin{proposition}[Boundary compatibiliy]\label{thm:boundary-compatibility}
  Let $r$ be a boundary condition kernel with a projection $\Pi$
  satisfying \eqref{eq:equi-bdd}. Assume  either
  \begin{itemize}
  \item the uniform bound
    \begin{equation}
      \label{eq:uniform-gap}
      \int_{\Gamma_+}
      \left[(\mathfrak f - \Pi \mathfrak f)^2
        + (R \mathfrak f - \Pi \mathfrak f)^2
      \right]
      \dd \nu
      \lesssim
      \int_{\Gamma_+}
      \Big[ \mathfrak f^2 - \left(\cR \mathfrak f \right)^2 \Big]
       \dd \nu,
    \end{equation}
  \item or that $r$ has the form of Maxwell boundary
    conditions~\eqref{eq:boundary-concrete} for any
    $\alpha : \partial \Omega \to [0,1]$.
  \end{itemize}
  Then the boundary compatibility condition \hypref{h:bdd} is satisfied.
\end{proposition}
\begin{proof}
  By \eqref{eq:equi-bdd} it holds that
  \begin{align*}
    &\int_{\Gamma_+} \left[
      \frac{(\cR^*(f_\infty \varphi))}{f_\infty}
      \mathfrak f^2
      - \varphi \left( \cR \mathfrak f \right)^2
      \right]  \dd \nu \\
    &\le
      \int_{\Gamma_+} \left[
      \frac{(\cR^*(f_\infty \varphi)-\Pi(f_\infty \varphi))}{f_\infty}
      \mathfrak f^2
      - \frac{(f_\infty\varphi - \Pi(f_\infty \varphi))}{f_\infty}
      \left( \cR \mathfrak f \right)^2
      \right]  \dd \nu \\
    &\qquad + \| \phi \|_\infty
      \int_{\Gamma_+}
      \Big[ \mathfrak f^2 - \left(\cR \mathfrak f \right)^2 \Big]
       \dd \nu.
  \end{align*}
  Furthermore, note that
  \begin{equation} \label{eq:bdd-comp-inter}
    \begin{aligned}
      &\int_{\Gamma_+} \left[
        \frac{(\cR^*(f_\infty \varphi)-\Pi(f_\infty \varphi))}{f_\infty}
        \mathfrak f^2
        - \frac{(f_\infty\varphi - \Pi(f_\infty \varphi))}{f_\infty}
        \left( \cR \mathfrak f \right)^2
      \right]  \dd \nu \\
      &=
      \int_{\Gamma_+} \left[
        \frac{(\cR^*(f_\infty \varphi)-\Pi(f_\infty \varphi))}{f_\infty}
        (\mathfrak f- \Pi \mathfrak f)^2
        - \frac{(f_\infty\varphi - \Pi(f_\infty \varphi))}{f_\infty}
        \left( \cR \mathfrak f - \Pi \mathfrak f\right)^2
      \right]  \dd \nu
    \end{aligned}
  \end{equation}
  because over regions that $R$ is relating
  $(\Pi \mathfrak f)/f_\infty$ is constant by \eqref{eq:equi-bdd} so
  that by the definition of the adjoint
  \begin{equation*}
    \int_{\Gamma_+}
    \left[
      \frac{(\cR^*(f_\infty \varphi)-\Pi(f_\infty \varphi))}{f_\infty}
      \mathfrak f\; \Pi \mathfrak f
      -
      \frac{(f_\infty\varphi - \Pi(f_\infty \varphi))}{f_\infty}
      R \mathfrak f\; \Pi \mathfrak f
    \right] \dd \nu
    = 0
  \end{equation*}
  and by the conservation of mass~\eqref{eq:mass-conservation-bdd}
  \begin{equation*}
    \int_{\Gamma_+}
    \left[
      \frac{(\cR^*(f_\infty \varphi)-\Pi(f_\infty \varphi))}{f_\infty}
      (\Pi \mathfrak f)^2
      -
      \frac{(f_\infty\varphi - \Pi(f_\infty \varphi))}{f_\infty}
      (\Pi \mathfrak f)^2
    \right] \dd \nu
    = 0.
  \end{equation*}

  In the case of the uniform bound~\eqref{eq:uniform-gap}, the resulting
  form in~\eqref{eq:bdd-comp-inter} can be directly estimated by the
  boundary dissipation as required.

  In the case that $R$ is a Maxwell boundary condition, we have that
  $R \mathfrak f = (1-\alpha)\mathfrak f + \alpha \Pi \mathfrak f$ so that the
  final expression from~\eqref{eq:bdd-comp-inter} is
  \begin{equation*}
    \begin{aligned}
      &\int_{\Gamma_+} \left[
        \frac{(\cR^*(f_\infty \varphi)-\Pi(f_\infty \varphi))}{f_\infty}
        (\mathfrak f- \Pi \mathfrak f)^2
        - \frac{(f_\infty\varphi - \Pi(f_\infty \varphi))}{f_\infty}
        \left( \cR \mathfrak f - \Pi \mathfrak f\right)^2
      \right]  \dd \nu \\
      &=
      \int_{\Gamma_+}
      \left[(1-\alpha) - (1-\alpha)^2\right]
      \left[
        \frac{(f_\infty\varphi - \Pi(f_\infty \varphi))}{f_\infty}
        (\mathfrak f- \Pi \mathfrak f)^2
      \right]
       \dd \nu \\
      &\le
      \| \phi \|_\infty
      \int_{\Gamma_+}
      (1-\alpha) \alpha\,
      (\mathfrak f- \Pi \mathfrak f)^2
       \dd \nu.
    \end{aligned}
  \end{equation*}
  In this case the boundary dissipation is
  \begin{equation*}
    \int_{\Gamma_+}
    \Big[ \mathfrak f^2 - \left(\cR \mathfrak f \right)^2 \Big]
     \dd \nu
    =
    \int_{\Gamma_+}
    \alpha\, (2-\alpha)
    \Big[ \mathfrak f - \Pi \mathfrak f \Big]^2
     \dd \nu,
  \end{equation*}
  which yields the uniform bound for any
  $\alpha : \partial \Omega \to [0,1]$.
\end{proof}

\subsection{Proof of the control condition for deterministic
  transport}

In this subsection, we cover the case of deterministic transport
\eqref{eq:ugcc-deterministic} in \cref{cor:LB}. This conclusion also
proves the decay in \cref{ex:weight-gain}. Instead of using the
transport control condition \hypref{h:tcc}, we can directly verify the
transport mapping condition \hypref{h:tmc} with vanishing
\(\tilde \psi = 0\).

In this case, we take \(\cB = 0\) so that \(\cG_t\) is the dual
transport semigroup following the trajectories backward. We then set
\(\tilde \psi \equiv 0\) and
\begin{align}
  \label{eq:deterministic-psi}
  \psi(t,x,v) :=
  \frac{\chi(x) w(v)}
  {\int_0 ^T \chi(X_{t-s}(x,v)) w(V_{t-s}(x,v)) \dd s},\quad
  \forall \, (t,x,v) \in [0,T] \times \Omega \times \sV.
\end{align}
It is well-defined as the denominator is uniformly bounded from
below. Moreover it is
$W^{1,\infty}([0,T] \times \Omega \times \sV)$, non-negative,
$\supp \psi(t,\cdot,v) = \supp \chi \subset \sigma$ and
importantly
\begin{align*}
  \forall \, (x,v) \in \Omega \times \sV, \qquad
  \int_0 ^T \left( \cG_t \psi_t \right)(x,v) \dd t = 1.
\end{align*}

For the required bounds in \hypref{h:tmc} note that \(G\) from
\eqref{eq:tmc:def-g} becomes in this choice
\begin{equation*}
  G(t,x,v) =
  \frac{\int_{s=t}^T \chi(X_{t-s}(x,v)) w(V_{t-s}(x,v)) \dd s}
  {\int_{s=0}^T \chi(X_{t-s}(x,v)) w(V_{t-s}(x,v)) \dd s}.
\end{equation*}
Then by construction \(|G| \le 1\). In the case of a bounded collision
operator, this implies the required bounds \eqref{eq:g-bound-2} and
\eqref{eq:g-bound-1} by the spectral gap of \(\cL\).

For the case of the Fokker-Planck operator, the assumed propagation of
regularity along the transport shows with \(|\nabla w| \lesssim w\)
and \eqref{eq:ugcc-deterministic-upper} that \(G \in W^{1,\infty}\).
In the case $\sV=\R^d$, we find that
\begin{equation*}
  \begin{split}
    &- \int_{0}^T \int_{\Omega \times
      \sV}  G(t,x,v) (\sigma f \cL f)(t,x,v) \dd \mu
    \dd t \\
    &= \int_{0}^T \int_{\Omega \times \sV}
    \sigma G\, |(\nabla_v + v)f|^2 \dd \mu \dd t
    + \int_{0}^T \int_{\Omega \times \sV}
    \sigma f \nabla_v G \cdot (\nabla_v + v)f \dd \mu \dd t \\
    &\le \int_{0}^T \int_{\Omega \times \sV}
    |G|\, \sigma\, |(\nabla_v + v)f|^2 \dd \mu \dd t \\
    &\qquad +
    \int_0^T \int_{\Omega}
    \| \nabla_v G(t,x,\cdot) \|_\infty \sqrt{\sigma}\,
    \| f(t,x,\cdot) \|_{L^2(M^{-1})}
    \| \sqrt{\sigma} (\nabla_v +v) f(t,x,\cdot) \|_{L^2(M^{-1})}
    \ee^{\phi} \dd x \dd t
  \end{split}
\end{equation*}
and
\begin{equation*}
  \begin{split}
    \left| \int_{\sV}
      G(t,x,v) (\sigma \cL f)(t,x,v) \dd v \right|^2
    &=
    \left| \int_{\sV}
      \sigma\, \nabla_v G(t,x,v) \cdot (\nabla_v+v)f(t,x,v) \dd v
    \right|^2 \\
    &\le
    \left(
      \int_{\sV} \sigma |\nabla_vG|^2 M(v)\dd v
    \right)
    \left(
      \int_{\sV} \sigma |(\nabla_v+v)f|^2 \frac{\dd v}{M(v)}
    \right).
  \end{split}
\end{equation*}
and the desired bounds follow from the fact that $G_t$ and
$\sqrt{\sigma} \nabla_v G_t$ are uniformly bounded by construction.
The case of the Laplace-Baltrami operator is similar.

\subsection{Proof that \hypref{h:simple-tcc} implies \hypref{h:tcc}
  (with vanishing corrector $\mathfrak C =0$)}

We first cover the case 2 of \hypref{h:simple-tcc}.  We claim
\eqref{eq:gcc} implies that
\begin{equation}\label{eq:decay-l1-dual}
  \begin{dcases}
    \partial_t f + \cT f - \sigma \cL f = - \chi f
    &\text{in } \Omega \times \sV, \\
    \gamma_- f(v) = \left( \cR \gamma_+ f \right)
    \big(v-2(\vec{n}\cdot v)\vec{n}\big) &\text{in } \partial
    \Omega \times \sV \text{ with } \vec{n} \cdot v \le 0,
  \end{dcases}
\end{equation}
decays in $L^1$. To prove the claimed decay, it suffices by linearity
to take the positive part and note the decay as
\begin{equation*}
  \| f_T \|_{L^1} = \| f_{\init} \|_{L^1}
  - \int_{s=0}^T \int |f_s \chi| \dd x \dd v \dd s.
\end{equation*}
By the assumption \eqref{eq:gcc} it holds that
\begin{equation*}
  \int_{s=0}^T \int \sfull_s f \chi \dd x \dd v \dd s \ge c
\end{equation*}
and by Duhamel
\begin{equation*}
  \int_{s=0}^T \int |f_s - \sfull_s f| \dd x \dd v \dd s
  \lesssim \int_{s=0}^T \int |f_s \chi| \dd x \dd v \dd s.
\end{equation*}
As $\chi$ is bounded we thus find the claimed decay in \(L^1\).

By duality, the dual evolution to \eqref{eq:decay-l1-dual} decays in
\(L^\infty\) which is the required decay of \eqref{eq:tcc-decay} in
\hypref{h:tcc} with \(\cB = \sigma \cL^T\). To verify the bounds of
\eqref{eq:diff-bound-2} note by the choice of \(\cB\) that
\(\cB^*(f f_\infty) - f_{\infty} \sigma \cL f =0\) which proves the
second part. For the first part note that
\begin{equation*}
  \int_{\Omega \times \sV}
  [\cB^*(f^2) - 2 f \sigma \cL f]\, (1-\varphi_t)
  \dd \mu
  \lesssim \| 1 - \varphi_t \|_{L^\infty}
  \int_{\Omega \times \sV}
  \left|\cB^*(f^2) - 2 f \sigma \cL f\right|
  \dd \mu
\end{equation*}
As \(\varphi_t\) is bounded in \(L^\infty\), the factor \(\| 1 -
\varphi_t \|_{L^\infty}\) is bounded. By the assumed sign in
\eqref{eq:gamma-2-condition-l}, we can drop the absolute value and
find
\begin{equation*}
  \int_{\Omega \times \sV}
  \left|\cB^*(f^2) - 2 f \sigma \cL f\right|
  \dd \mu
  =
  \int_{\Omega \times \sV}
  \cB^*(f^2) - 2 f \sigma \cL f
  \dd \mu
  = -2
  \int_{\Omega \times \sV}
  f \sigma \cL f
  \dd \mu
\end{equation*}
where we used that
\(\int_{\Omega \times \sV} \cB^*(f^2) \dd \mu = \int \sigma
\cL(f^2/f_\infty) f_\infty \dd \mu = 0\) by mass conservation. This
shows the required first bound in \eqref{eq:diff-bound-2}.

For the case 1 and 1' of \hypref{h:simple-tcc}, the given transport
control assumptions imply by the same decay of \eqref{eq:tcc-decay}
with \(\cB=0\). Then \eqref{eq:diff-bound-2} can be verified as in the
previous subsection for the deterministic case.

\subsection{Proof of the $\Gamma$ condition}\label{sec:gamma-2-verification}

Here we prove \cref{thm:gamma-2-verification}. It is classical for the
Fokker-Planck operator so that we focus on the linear Boltzmann
operator of the form \eqref{eq:LB} with a reversible kernel, i.e.\ for
all \(v,v_*\in \sV\) it holds that $ k(v,v_*) M(v_*) = k(v_*,v) M(v) $.

\begin{proof}[Proof of \cref{thm:gamma-2-verification}]
  For the linear Boltzmann operator we find
  \begin{equation*}
    \begin{aligned}
      \left[M \cL \left(\frac{f^2}{M}\right)
      - 2 f \cL f\right](v)
      &= \int k(v_*,v) \dd v_*
        \left[
        f(v) - \frac{\int k(v,v_*) f(v_*) \dd v_*}{\int k(v_*,v) \dd v_*}
        \right]^2 \\
      &\quad
        - \frac{\left[\int k(v,v_*) f(v_*) \dd v_*\right]^2}{\int k(v_*,v)
        \dd v_*}
        + M(v) \int k(v,v_*) \frac{f(v_*)^2}{M(v_*)}\dd v_*
    \end{aligned}
  \end{equation*}
  Under the reversibility condition
  $ k(v,v_*) M(v_*) = k(v_*,v) M(v) $ we find
  \begin{equation*}
    \begin{split}
      &M(v) \int k(v,v_*) \frac{f(v_*)^2}{M(v_*)} \dd v_*
      -
      \frac{\left[\int k(v,v_*) f_* \dd v_*\right]^2}{\int k(v_*,v)
        \dd v_*} \\
      &= M(v) \int
      \left(
        \frac{f(v_*)}{M(v_*)}
        -
        \frac{\int k(v,w) f(w) \dd w}{\int k(v,w) M(w) \dd w}
      \right)^2
      k(v,v_*) M(v_*) \dd v_* \ge 0
    \end{split}
  \end{equation*}
  which implies the claimed sign.
\end{proof}

\subsection{Proof of \hypref{h:tcc} in the hypoelliptic case}

We now complete the proof of \cref{thm:w-example} by constructing the
corrector operator $\mathfrak C$ in \hypref{h:tcc} for the setting of
\cref{ex:hypoelliptic-decay}. Take $\cB = 0$ in \hypref{h:tcc}. The
characteristics are given by
\begin{equation*}
  \begin{pmatrix}
    X_t \\ V_t
  \end{pmatrix}
  = r
  \begin{pmatrix}
    \sin(\theta + t)\\
    \cos(\theta + t)
  \end{pmatrix}
\end{equation*}
for parameters $r$ and $\theta$ determined by $X_0,V_0$. Introduce a
cutoff function $\gamma : \R \to [0,1]$ with
$\supp \gamma \in [-2,2]$ and $\gamma(x)=1$ for $|x| \le 1$ and
introduce
\begin{equation*}
  a(x,v) := - f_\infty^{-1}(x,v) \gamma(x) \gamma(v)\, x
  \max\left(-1,\min\left(\frac{v}{x}, 1\right)\right).
\end{equation*}
Using the Hörmander notation for the Fokker-Planck operator
$\cL = - \cA^* \cA$ with $\cA = (\nabla_v + v)$ and
$\cA^* = - \nabla_v$, we then take
\begin{equation*}
  \mathfrak C \varphi = \epsilon f_\infty^{-1} \cA^*(f_\infty a \varphi)
\end{equation*}
for a small constant $\epsilon >0$. Then \eqref{eq:tcc-decay} takes
the form
\begin{equation*}
  \partial_t \varphi - \cT \varphi - \epsilon a \nabla_v \varphi
  = -\epsilon f_\infty^{-1} \nabla_v\left(\gamma(x) \gamma(v)\, x
    \max\left(-1,\min\left(\frac{v}{x}, 1\right)\right) \right)
  \varphi
  - \chi \varphi
\end{equation*}
where we take $\chi$ smooth and zero around $x=0$ and with
$\chi \gtrsim 1$ for $|x| \ge 1/2$. The LHS then defines a transport
semigroup which is still essentially a circle for small enough
$\epsilon$.  For $\|(x,v)\| \le 3/4$ the new term on the right hand
side will create some decay. Outside it creates some growth which
can be absorbed by $\chi$ when choosing $\epsilon$ small enough. As
for the compatibility condition~\eqref{eq:diff-bound-2} note that
$|\nabla \varphi_t| \lesssim r^{-1}$ so that we can absorb it with the
factor $\sqrt{\sigma}$.

\appendix

\section{The commutator method for hypoelliptic control}
\label{sec:counter}

In this appendix, we provide another viewpoint on why the uniform
transport control condition can be partially relaxed in the case of
hypoelliptic operators, based on Villani's commutator conditions for
hypocoercivity~\cite{villani-2009-hypocoercivity}, itself inspired by
Hörmander~\cite{hoermander-1967-hypoelliptic}. It is an interesting
example of the commutator method that requires \emph{three}
commutators. We consider a slight generalisation of the setting of
\cref{ex:hypoelliptic-decay} as in \cref{thm:w-example}, with
$\Omega=\sV=\R$ and $\phi(x) = x^2/2$ (in fact the method can cover
small variations of the harmonic potential), and $\sigma = \kappa^2$
with
\begin{itemize}
\item $\kappa$ and $\kappa'$ are bounded and $\kappa' \ge -1/4$,
  and $|\kappa|$ is strictly bounded away from zero in the
  complement of every neighbourhood of $x=0$,
\item $\kappa'(0) = 1$ and $\kappa'(x) |x|$ is uniformly bounded
  for $x \in \R$,
\item $\kappa$ is smooth enough (as needed in the proof).
\end{itemize}

Following the setup of \cite{villani-2009-hypocoercivity} we
equivalently prove the decay of $h := f/f_\infty$ in
$L^2(f_\infty)$ where the equation write
$\partial_t h + \cB h = -\cA^* \cA h$ with
\begin{equation}
  \label{eq:definition-a-and-b}
  \cA = \kappa \partial_v, \qquad
  \cB = v \partial_x - \phi' \partial_v,
\end{equation}
where $\cB^* = -\cB$. Then \cite[Thm~24 and
Remark~26]{villani-2009-hypocoercivity} shows exponential decay
if there exist operators $\cC_0,\dots,\cC_{N_C+1}$,
$\cR_1,\dots,\cR_{N_C+1}$ and $\cZ_1,\dots,\cZ_{N_C+1}$ such that
\begin{align*}
  &\cC_0 = \cA,\\
  &[\cC_j,\cB] = \cZ_{j+1} \cC_{j+1} + \cR_{j+1}, \qquad
    \text{for $0\le j \le  N_c$},\\
  &\cC_{N_C+1} = 0
\end{align*}
and for $k = 0,\dots N_c$
\begin{equation}
  \label{eq:ass-villani}
  \left\{
    \begin{lgathered}
      \text{$[\cA,\cC_k]$ is bounded relative to
        $(\cC_j)_{0\le j \le k}$ and $(\cC_j \cA)_{0\le j \le k-1}$} \\
      \text{$[\cC_k,\cA^*]$ is bounded relative to
        $\mathrm{id}$ and $(\cC_j)_{0\le j \le k}$}\\
      \text{$\cR_k$ is bounded relative to
        $(\cC_j)_{0\le j \le k-1}$ and $(\cC_j \cA)_{0\le j \le k-1}$} \\
      \text{there are positive constants $\lambda_k$ and
        $\Lambda_k$ such that
        $\lambda_j \mathrm{id} \le \cZ_j \le \Lambda_j
        \mathrm{id}$}.
    \end{lgathered}
  \right.
\end{equation}
and
\begin{equation}
  \label{eq:coercive-villani}
  \sum_{j=0}^{N_c} \cC_j^* \cC_j \text{ is coercive.}
\end{equation}

In order to fix the operators, take a cutoff
$\gamma \in C^\infty(\R)$ with $\gamma : \R \to [0,1]$ and
$\gamma(x)=1$ for $|x| \le 1$ and $\supp \gamma \subset
[-2,2]$. As the rescaled version define
$\gamma_\delta(x) = \gamma(x/\delta)$. Then consider
\begin{equation}
  \label{eq:commutator:def-kappa-tilde}
  w := \gamma_\delta \kappa' + (1-\gamma_\delta)
  \quad\text{and}\quad
  \tilde{\kappa} := \frac{\kappa}{w}
\end{equation}
for a sufficiently small $\delta$. By choosing $\delta$
sufficiently small we can ensure that
\begin{equation*}
  \frac 12 \le w \le \frac 32,\quad
  \| \tilde{\kappa} \|_\infty + \| (1+|x|) \tilde{\kappa}' \| \le
  \infty\quad
  \text{and}\quad
  \tilde{\kappa}' \ge - 1/2.
\end{equation*}

We then build a sequence of commutators $(\cC_n)_n$ where we
take away known parts using $(\cZ_n)_n$ and $(\cR_n)_n$ in
order to avoid problematic additional commutator terms. Starting
with $\cC_0=\cA$ find
\begin{equation*}
  [\cC_0,\cB] = \kappa \partial_x - v \kappa' \partial_v
  = w (\tilde{\kappa} \partial_x - v \partial_v)
  + (1-\gamma_\delta) (1-\kappa') v \partial_v
\end{equation*}
so that we take
\begin{align*}
  \cZ_1 &= w \\
  \cC_1 &= \tilde{\kappa} \partial_x - v \partial_v \\
  \cR_1 &= (1-\gamma_\delta) (1-\kappa') v \partial_v.
\end{align*}

In the next iteration we find
\begin{equation*}
  [\cC_1,\cB] = - \tilde{\kappa} \phi'' \partial_v - v \partial_x
  - v \tilde{\kappa}' \partial_x - \phi' \partial_v
  = (1+\tilde{\kappa}') (-v\partial_x - \phi' \partial_v)
  + (\tilde{\kappa}' \phi' - \tilde{\kappa} \phi'') \partial_v
\end{equation*}
so that we take
\begin{align*}
  \cZ_2 &= 1 + \tilde{\kappa}' \\
  \cC_2 &= -v \partial_x - \phi' \partial_v \\
  \cR_2 &= (\tilde{\kappa}' \phi' - \tilde{\kappa} \phi'')
          \partial_v.
\end{align*}

In the next iteration we find
\begin{equation*}
  [\cC_2,\cB] = 2 \phi'' (\tilde{\kappa} \partial_x - \cC_1) -
  2\phi' \partial_x = 2(\phi'' \tilde{\kappa} - \phi') \partial_x
  - 2 \phi'' \cC_1
\end{equation*}
so that we take
\begin{align*}
  \cZ_3 &= 2 \\
  \cC_3 &= (\phi'' \tilde{\kappa} - \phi') \partial_x \\
  \cR_3 &= 2 \phi'' \cC_1.
\end{align*}

In the last step we find
\begin{equation*}
  \cR_4 := [\cC_3,\cB] = (-\phi'' \tilde{\kappa} - \phi') \phi''
  \partial_v - v (\phi'' \tilde{\kappa} - \phi')' \partial_x
\end{equation*}
where we have taken $\cZ_4 = 1$ and $\cC_4 = 0$.

The operators $\cC_1,\cC_2,\cC_3$ give a good weighted control.
\begin{lemma}
  \label{thm:controlled-derivatives}
  There exists a constant $c$ such that
  \begin{equation*}
    \| (x^2+v^2) |\nabla h|^2 \|_{L^2(f_{\infty})}^2
    \le c  \left(
      \| \cC_1 h \|_{L^2(f_{\infty})}^2 + \| \cC_2 h
      \|_{L^2(f_{\infty})}^2 + \| \cC_3 h \|_{L^2(f_{\infty})}^2
    \right).
  \end{equation*}
\end{lemma}
\begin{proof}
  Split into the cases $|x| \le \delta$ and $|x| \ge \delta$. For
  $|x| \le \delta$ find by elementary algebra
  \begin{align*}
    &\int_{|x| \le \delta}\int_{v \in \R}
      (x^2 + v^2) |\nabla h|^2\, f_{\infty} \dd x \dd v \\
    &=
      \int_{|x| \le \delta}\int_{v \in \R}
      \left [
      |\cC_2 h|^2 + \frac{x}{\tilde\kappa} |\cC_1 h|^2
      + (x^2-x\tilde\kappa) |\partial_x h|^2 + v^2
      \left(1-\frac{x}{\tilde\kappa}\right) |\partial_v h|^2
      \right] \, f_{\infty} \dd x \dd v.
  \end{align*}
  The last two summands in the integral of the RHS can be
  absorbed into the LHS giving the claimed inequality.

  For $|x| \ge \delta$, first note that there exists a constant
  $c$ with
  \begin{equation*}
    \int_{|x| \ge \delta}\int_{v \in \R} |x^2|\; |\partial_x h|^2
    f_{\infty} \dd x \dd v
    \le c \int_{|x| \ge \delta}\int_{v \in \R} |\cC_3 h|^2
    f_{\infty} \dd x \dd v.
  \end{equation*}
  Hence we can find a constant $c$ such that
  \begin{equation*}
    \int_{|x| \ge \delta}\int_{v \in \R} |v^2| |\partial_v h|^2
    f_{\infty} \dd x \dd v \le c \int_{|x| \ge \delta}\int_{v \in
      \R} \left[|C_1 h|^2
      + |C_3 h|^2\right] f_{\infty} \dd x \dd v.
  \end{equation*}
  Finally with $C_2$ we can control the remaining terms
  $v^2 |\partial_x h|^2$ and $x^2 |\partial_v h|^2$ giving the
  claimed result.
\end{proof}

For the bounds we use a simple adaptation of the Poincaré inequality:
\begin{lemma}
  \label{thm:iterated-a-bound}
  There exists a constant $c$ such that
  \begin{equation*}
    \| \kappa^2 v \partial_v h \|_{L^2(f_\infty)}^2
    \le c \left[ \| \cA h \|_{L^2(f_\infty)}^2 + \| \cA \cA h \|_{L^2(f_\infty)}^2
    \right].
  \end{equation*}
\end{lemma}
\begin{proof}
  Use the refined Poincaré inequality
  \begin{equation*}
    \int_{\R} v^2 |f|^2 \ee^{-v^2/2} \dd v
    \le c \left[
      \int_{\R} |\partial_v f|^2 \ee^{-v^2/2} \dd v
      + \left( \int_{\R} f(v)\, \ee^{-v^2/2} \dd v \right)^2
    \right]
  \end{equation*}
  for a constant $c$, which immediately implies
  \begin{equation*}
    \int_{\R} v^2 |f|^2 \ee^{-v^2/2} \dd v
    \le c \left[
      \int_{\R} |\partial_v f|^2 \ee^{-v^2/2} \dd v
      + \int_{\R} |f|^2 \ee^{-v^2/2} \dd v
    \right].
  \end{equation*}
  Applying this to $\partial_v h$ yields the claimed result.
\end{proof}

We now prove the relevant error bounds.
\begin{lemma}
  The operators $\cC_0,\dots,\cC_{N_C+1}$,
  $\cR_1,\dots,\cR_{N_C+1}$, $\cZ_1,\dots,\cZ_{N_C+1}$ satisfy
  the bounds \eqref{eq:ass-villani}.
\end{lemma}
\begin{proof}
  For the commutator $[\cA,\cC_k]$ and $[\cC_k,\cA^*]$, note that
  $ \cA^* = \kappa v - \cA $ and we find
  \begin{align*}
    [\cA,\cC_0] &= 0 & [\cC_0,\cA^*] &= \kappa^2 \\
    [\cA,\cC_1] &= -\kappa (1+\kappa') \partial_v
                   &  [\cC_1,\cA^*] &= [\cA,\cC_1] + \kappa (\kappa' -1) v \\
    [\cA,\cC_2] &= -\kappa \partial_x - v\kappa' \partial_v
                   &  [\cC_2,\cA^*] &= [\cA,\cC_2] - v^2 \kappa - \kappa x \\
    [\cA,\cC_3] &= (\kappa-x) \kappa' \partial_v
                   &  [\cC_3,\cA^*] &= [\cA,\cC_3] - (\kappa-x) \kappa' v
  \end{align*}
  Here are $[\cA,\cC_1]$ and $[\cA,\cC_3]$ relatively bounded to
  $\cC_0=\cA$. For $\cC_2$ note that
  \begin{equation*}
    \| \kappa \partial_x h\|_{L^2(f_\infty)}^2 + \| v \partial_v h \|_{L^2(f_\infty)}^2
    \le \|\cC_1 h\|_{L^2(f_\infty)}^2 + \| \cC_2 h \|_{L^2(f_\infty)}^2.
  \end{equation*}
  Therefore, $[\cA,\cC_2]$ is bounded relative to $\cC_1$ and $\cC_2$.

  Now consider the additional terms in the commutator with $\cA^*$. In
  $[\cC_0,\cA^*]$ the additional term is bounded relative to the
  identity.  As in \cref{thm:iterated-a-bound}, we find from the
  refined Poincaré inequality that the additional term in
  $[\cC_1,\cA^*]$ and $[\cC_3,\cA^*]$ is bounded relative to
  $\mathrm{id}$ and $\cC_0=\cA$. For the additional term in
  $[\cC_2,\cA^*]$ use the refined Poincaré inequality with
  $|v \partial_v h|^2$ which as before is controlled. Hence the
  additional term is bounded relative to $\mathrm{id}$, $\cC_1$ and
  $\cC_2$.

  This proves the assumptions (1) and (2).

  For the bound on the error terms note that $\cR_1$ is bounded
  relative to $\cC_0$ and $\cC_0\cA$. The remainder $\cR_2$ is bounded
  relative to $\cC_0$ and $\cR_3$ is bounded relative to $\cC_1$. For
  $\cR_4$ use \cref{thm:controlled-derivatives} to show that it is
  bounded relative to $\cC_1$, $\cC_2$ and $\cC_3$.

  Finally, the explicit form of the factors $\cZ_1,\dots \cZ_4$ shows
  the required control (4).
\end{proof}

The last remaining part is a weighted Poincaré ineqality.
\begin{lemma}
  There exists a constant $c$ such that
  \begin{equation*}
    \int_{\R} \int_{\R} |h|^2 \, f_\infty \dd x \dd v
    \le c \int_{\R} \int_{\R} (x^2+v^2) |\nabla h|^2 \, f_\infty \dd x \dd v
  \end{equation*}
  for all $h$ with
  \begin{equation*}
    \int_{\R} \int_{\R} h(x)\, f_\infty \dd x \dd v = 0.
  \end{equation*}
\end{lemma}
The proof of the remaining Poincaré inequality can be done
explicitly using the expansion into Hermite polynomials, see
\cite{dietert-thesis-2017}. However, it can also be shown by
adapting Theorem A.1 of \cite{villani-2009-hypocoercivity}:
\begin{proof}
  Define the measure $\nu$ by
  \begin{equation*}
    \dd \nu(x,v) = \ee^{-V(x,v)} \dd x \dd v
  \end{equation*}
  with
  \begin{equation*}
    V(x,v) = \frac{x^2}{2} + \frac{v^2}{2} - \log(x^2+v^2).
  \end{equation*}
  By expanding $0 \le \int \nabla (h \ee^{-V}) \dd x\dd v$, we
  arrive at
  \begin{equation*}
    \int_{\R\times \R} |h|^2
    (|\nabla V|^2 - 2 \Delta V) \dd x \dd v
    \le 4 \int_{\R\times\R} |\nabla h|^2 \dd x \dd v.
  \end{equation*}
  In our case
  \begin{equation}
    \label{eq:counter:poin:1}
    \int_{\R \times \R}
    |h|^2 (4 - 6(x^2+v^2) + (x^2+v^2)^2)
    \ee^{-x^2/2-v^2/2} \dd x \dd v
    \le 4 \int_{\R \times \R} (x^2+v^2) |\nabla h|^2
    \ee^{-x^2/2-v^2/2} \dd x \dd v,
  \end{equation}
  which shows the required inequality for $(x,v) \to 0$ and
  $(x,v) \to \infty$.

  Following the proof in \cite{villani-2009-hypocoercivity}, this
  shows the claimed inequality together with the standard
  Poincaré inequality on bounded domains. This inequality ensures
  that for every $r>0$, there exists a constant $c(r)$ such that
  \begin{align*}
    \int_{r^{-1} \le \| (x,v) \| \le r}
    |h|^2 \mathcal{Z}^{-1}(r) \ee^{-x^2/2-v^2/2} \dd x \dd v
    &\le c(r) \int_{r^{-1} \le \| (x,v) \| \le r}
      |\nabla h|^2 \mathcal{Z}^{-1}(r) \ee^{-x^2/2-v^2/2} \dd x \dd v\\
    &+ \left[
      \int_{r^{-1} \le \| (x,v) \| \le r}
      h\, \mathcal{Z}^{-1}(r) \ee^{-x^2/2-v^2/2} \dd x \dd v
      \right]^2
  \end{align*}
  where
  \begin{equation*}
    \mathcal{Z}(r) = \int_{r^{-1} \le \| (x,v) \| \le r}
    \ee^{-x^2/2-v^2/2} \dd x \dd v.
  \end{equation*}
  As $\mu(h) = 0$ we find
  \begin{align*}
    \left[
    \int_{r^{-1} \le \| (x,v) \| \le r}
    h\, \mathcal{Z}^{-1}(r) \ee^{-x^2/2-v^2/2} \dd x \dd v
    \right]^2
    &= \mathcal{Z}^{-2}(r)
      \left[
      \int_{\| (x,v) \| \not\in [r^{-1},r]}
      h\,  \ee^{-x^2/2-v^2/2} \dd x \dd v
      \right]^2\\
    &\le \epsilon(r)
      \int_{r^{-1} \le \| (x,v) \| \le r}
      |h|^2 \ee^{-x^2/2-v^2/2} \dd x \dd v
  \end{align*}
  where
  \begin{equation*}
    \epsilon(r) = \mathcal{Z}^{-2}(r)
    \int_{r^{-1} \le \| (x,v) \| \le r}
    \ee^{-x^2/2-v^2/2} \dd x \dd v.
  \end{equation*}
  As $r \to \infty$, we have that $\epsilon(r) \to 0$, so that
  for large enough $r$ we can combine this inequality with
  \eqref{eq:counter:poin:1} to find the result.
\end{proof}

Hence with this commutator setup we find exponential convergence
to equilibrium.

\section{Traces for the considered solutions}
\label{sec:traces}

We adapt the techniques in~\cite{mischler-2000-trace} to conclude the
following lemma (the result in~\cite{mischler-2000-trace} only applies
to bounded collision operators in $L^1$):
\begin{proposition}[Existence of the trace]
  Consider $f \in L^\infty([0,+\infty);L^2(\Omega \times \sV, \dd\mu)$
  so that
  $\partial_t f + \cT f \in L^\infty([0,+\infty);L^2(\Omega ; \mathcal
  H(\sV, \langle \cdot \rangle^{-k} \dd\mu))$ for some $k \in \N$,
  then the traces in time and at $\partial \Omega$ exist in
  $L^2(L^2(\partial \Omega; \mathcal H(\sV, (n \cdot v)^2 \langle
  \cdot \rangle^{-k} \dd\mu))$, where $\mathcal H=L^2$ for the bounded
  linear Boltzmann operator and $\mathcal H=H^{-1}$ for the
  Fokker-Planck operator, satisfy the Green formula, and can be
  renormalised by any $\beta \in W^{1,\infty}(\R_+)$.
\end{proposition}

For the transport semigroup in $L^1$ we can construct a semigroup
which contracts the mass (in general the mass can escape to infinity),
see
\cite{voigt-1981-boundary-value,lods-2005-semigroup,arlotti-lods-2005-substochastic,mokhtar-kharroubi-2008-transport-semigroups,arlotti-banasiak-lods-2011-transport}. These
works also include trace theorems with Green's formula. The oldest
reference for the Green's formula seems to be
\cite{cessenat-1984-theoremes,cessenat-1985-theoremes} but they can
also be found in
\cite{mischler-2000-trace,mischler-2010-kinetic-maxwell}. For
completeness
\cite{arlotti-lods-2019-lp-transport,arlotti-lods-2019-lp-transport-2}
constructs the transport semigroup in $L^p$.

The key difficulty with traces in the proof of
\cref{sec:following-trajectories} is dealing with the term
\begin{equation*}
  \int_{\Omega \times \sV}
  \cT \Big[ f_s^2  \cG_{t-s} (\psi_t{+}\tilde \psi_t)
  \Big] \dd \mu.
\end{equation*}
Assume that $\cL$ is bounded. Then first note that then
$f \in C(0,\infty,L^2)$ and $(\partial_t + \cT)f \in L^2$ and its
trace exists in $L^2$; moreover if $g \in C(0,\infty,L^\infty)$ and
$(\partial_t + \cT)g \in L^\infty$ its trace exists in
$L^\infty$. Therefore we have traces for $f$, $f^2$ and
$f^2 \cG_{t-s} (\psi_t{+}\tilde \psi_t)$ by the product rule for the
weak derivative in the interior, since
$g:=\cG_{t-s} (\psi_t{+}\tilde \psi_t)$ satisfies the $L^\infty$ a
priori bounds. By weak-strong convergence of
approximations we also see that the trace commutes, i.e.
$\gamma f^2 = (\gamma f)^2$ and likewise with $\psi$. During the trace
we loose a weight (related to the return time and typically like
$|v|^{-1}$) and we need enough integrability at infinity to make sure
that $f$ preserves mass and the boundary operator makes sense.

\section{Non-uniform control conditions}
\label{sec:non-unif}

We can treat \emph{non-uniform} geometric control
conditions, with less quantitative estimates and additional
non-concentration and tightness assumptions on the solution. For
simplicity we formulate as Case 2 of \hypref{h:simple-tcc} which we
replace by the following condition.

\sethypothesistag{\ref*{h:tcc}''}
\begin{hypothesis}[Non-uniform Geometry Control Condition]
  \phantomsection
  \label{h:nugcc}
  Assume \(\cL\) satisfies \eqref{eq:gamma-2-condition-l} and that
  there is a connected domain $\Sigma \subset \Omega$ with the
  following properties: (i)~$(\Sigma,\phi)$ is $\epsilon$-regular for
  some $\epsilon >0$, (ii) $\inf_{x \in \Sigma} \sigma(x) > 0$, and
  (iii) there is a non-negative $\chi \in L^\infty(\Omega)$ with
  $\supp \chi \subset \Sigma$ so that (with
  $\cG_t \varphi := f_\infty^{-1} {\sfull_t}^*( \varphi f_\infty)$)
  \begin{align}
    \label{eq:nugcc}
     \forall \, (x,v) \in \Omega \times \sV, \quad \exists
     \, T=T(x,v) >0 \ \text{ such that } \
    \int_0 ^{T(x,v)} \cG_t \chi (x,v) \dd t >0.
  \end{align}
  Moreover we assume there is a
  Banach space $\mathcal B$ dense in $L^2(\Omega \times \sV, \dd \mu)$
  so that $\| f_t \|_{\mathcal B} \lesssim \| f_\init \|_{\mathcal B}$
  uniformly in time, and the norm of $\mathcal B$ provides
  non-concentration and tightness on the solution: that there is
  $a:\R_+^* \to \R_+^*$ going to zero at zero and infinity so that for
  any $B \subset \Omega \times \sV$ Borel set and $M>0$:
  \begin{equation}
    \label{eq:higher}
    \int_{\Omega \times \sV} \indicator_B f^2 \dd
    \mu \le a(|B|) \| f \|_{\mathcal H}^2 \quad
    \text{ and } \quad
    \int_{\Omega \times \sV} \left( \indicator_{|x|
        \ge M} + \indicator_{|v| \ge M} \right) f^2 \dd
    \mu  \le a(M) \| f \|_{\mathcal H}^2.
  \end{equation}
\end{hypothesis}

\begin{theorem}[Non-uniform control condition]
  \label{theo:numain}
  Assume \hypref{h:geom}-\hypref{h:eq}-\hypref{h:local}-\hypref{h:bdd}-\hypref{h:coe}-\hypref{h:nugcc}. Then given any
  $f_{\init} \in L^2(\Omega \times \sV, \dd \mu)$, any
  $f=f(t,x,v) \in C^0([0,+\infty);L^2(\Omega \times \sV, \dd \mu))$
  admitting traces
  $\gamma f \in L^2([0,+\infty) \times \Omega; H^{-1}(\sV, (n \cdot
  v)^2 \dd \mu))$ and solution to~\eqref{eq:gen} satisfies
  \begin{align}
    \label{eq:nurelax}
    \left\| f_t - \left( \int_{\Omega \times \sV}
    f_{\init}\right) f_\infty  \right\|_{L^2(\Omega \times \sV,
    \dd \mu)} \xrightarrow[t \to \infty]{} 0.
  \end{align}
\end{theorem}
This implies the convergence to equilibrium without rate for the
concrete equations treated in \cref{cor:LB} when~\hypref{h:tcc} is
replaced by~\hypref{h:nugcc}.

A similar result was already obtained in
\cite{han-kwan-leautaud-2015-geometric} for the linear Boltzmann
equation and we adapt here their idea in our setting. It is
enough to prove the relaxation to equilibrium for initial data
$f_{\init} \in L^2(\mu) \cap \mathcal B$ by density of
$\mathcal B$ in $L^2(\mu)$. Assume by contradiction that there is
$f_{\init} \in L^2(\mu) \cap \mathcal B$ with zero global average
so that the solution does \emph{not} converge to zero in
$L^2(\mu)$. There is therefore $t_k \to \infty$ with
$t_{k+1}-t_k \to \infty$ as $k \to \infty$ so that
$\| f(t_k) \|_{L^2(\mu)}$ is uniformly bounded below in
$k \ge 1$. The sequence
$g_k(t) := f(t_k+t) \| f(t_k) \|_{L^2(\mu)}^{-1}$ then satisfies
$g_k \in L^2(\mu) \cap \mathcal B$ with
\begin{align*}
  \forall \, k \ge 1, \quad \forall \, t \ge 0, \qquad
  \| g_k(0) \|_{L^2(\mu)}=1,
  \quad
  \| g_k(t) \|_{\mathcal B} \lesssim 1,
\end{align*}
and for any fixed $T>0$,
\begin{align*}
  \left| \int_0 ^T \int_{\Omega \times \sV} \sigma g_k
  \cL g_k \dd \mu \dd t \right| \xrightarrow[k \to \infty]{} 0.
\end{align*}
By weak compactness of the unit ball in $L^2_{t,loc}L^2_{x,v}(\mu)$ we
can also\ assume that $(g_k)$ has a weak limit in this space. Such a
weak limit $\tilde g$ satisfies the transport equation without
collision operator and with $\tilde g = \langle \tilde g \rangle M$ on
$\text{supp } \psi$. Since~\eqref{eq:nugcc} implies that all points
can be connected to $\text{supp } \psi$ by a trajectory, $\tilde g$ is
at local equilibrium everywhere, and is thus zero since it solves the
transport equation. We then argue as
in~\eqref{eq:following}-\eqref{eq:split-local}-\eqref{eq:split-global}-\eqref{eq:am}
on $g_k$ with the weight $\chi$ to get
\begin{align}
  \label{eq:estim-cpct}
  \int_0 ^T \int_{\Omega \times \sV}
  g_k(0,x,v)^2
  \cG_t \chi(x,v) \dd \mu \dd t
  \lesssim  \gAvg{g_k}_\chi
  + \left| \int_0 ^T \int_{\Omega \times \sV} \sigma g_k
  \cL g_k \dd \mu \dd t \right|.
\end{align}
We then use~\eqref{eq:nugcc}-\eqref{eq:higher} to find
$B \subset \Omega \times \sV$ small enough and $T$ large
enough so that
\begin{align*}
  \forall \, (x,v) \in \left( \Omega \times \sV\right) \setminus B, \quad
  \left( \int_0 ^T\hat \cS_t ^* \psi_t (x,v) \dd
  t \right) \ge 1 \quad \text{ and } \quad
  \int_B g_k(0,x,v)^2 \dd \mu \le \frac12.
\end{align*}
Then, given this choice of $T$, the RHS of~\eqref{eq:estim-cpct}
goes to zero as $k \to \infty$ since $g_k$ weakly converges to
zero and the dissipation vanishes asymptotically on any time
interval, and thus for $k$ large enough
\begin{align*}
  \int_{\Omega \times \sV} g_k(0,x,v)^2 \dd \mu <1
\end{align*}
which contradicts the assumptions.

\section*{Acknowledgements}

All authors acknowledge partial support from the ERC grant
\textit{MATKIT} grant. HD \& CM acknowledge partial support from the
ERC grant \textit{MAFRAN} and would like to thank the Isaac Newton
Institute for Mathematical Sciences, Cambridge, for support and
hospitality during the programme ``Frontiers in kinetic theory''.
This work was supported by EPSRC grant no EP/R014604/1 and a grant
from the Simons Foundations. HD also acknowledges support from the UK
CDT EPSRC grant \textit{EP/H023348/1} \textit{Cambridge Centre for
  Analysis}, Universit\'e Sorbonne Paris Cit\'e in the framework of
the \textit{``Investissements d'Avenir'' convention ANR-11-IDEX-0005},
and the \textit{People Programme (Marie Curie Actions)} of the
European Union's Seventh Framework Programme (FP7/2007-2013) under REA
grant agreement n.\ PCOFUND-GA-2013-609102, through the PRESTIGE
programme coordinated by Campus France. HH also acknowledges the
support of the EPSRC programme grant \textit{Mathematical fundamentals
  of Metamaterials for multiscale Physics and Mechanics}
(EP/L024926/1).

\bibliographystyle{siam}
\bibliography{lit}

\begin{thebibliography}{10}

\bibitem{acosta-duran-2017-divergence-operator}
{\sc G.~{Acosta} and R.~G. {Dur\'an}}, {\em {Divergence operator and related
  inequalities}}, New York, NY: Springer, 2017.

\bibitem{albritton-armstrong-mourrat-novack-2019-variational-fokker-planck}
{\sc D.~Albritton, S.~Armstrong, J.~C. Mourrat, and M.~Novack}, {\em
  {Variational methods for the kinetic Fokker-Planck equation}}, 2021.
\newblock Preprint arXiv:1902.04037v2.

\bibitem{MR2765738}
{\sc K.~Aoki and F.~Golse}, {\em On the speed of approach to equilibrium for a
  collisionless gas}, Kinet. Relat. Models, 4 (2011), pp.~87--107.

\bibitem{arlotti-banasiak-lods-2011-transport}
{\sc L.~Arlotti, J.~Banasiak, and n.~B. Lods}, {\em On general transport
  equations with abstract boundary conditions. {The} case of divergence free
  force field}, Mediterr. J. Math., 8 (2011), pp.~1--35.

\bibitem{arlotti-lods-2005-substochastic}
{\sc L.~Arlotti and B.~Lods}, {\em Substochastic semigroups for transport
  equations with conservative boundary conditions}, J. Evol. Equ., 5 (2005),
  pp.~485--508.

\bibitem{arlotti-lods-2019-lp-transport-2}
{\sc L.~Arlotti and B.~Lods}, {\em An {{\(L^p\)}}-approach to the
  well-posedness of transport equations associated to a regular field. {II}},
  Mediterr. J. Math., 16 (2019), p.~30.
\newblock Id/No 145.

\bibitem{arlotti-lods-2019-lp-transport}
\leavevmode\vrule height 2pt depth -1.6pt width 23pt, {\em An
  {{\(L^p\)}}-approach to the well-posedness of transport equations associated
  with a regular field. {I}}, Mediterr. J. Math., 16 (2019), p.~25.
\newblock Id/No 152.

\bibitem{MR1178650}
{\sc C.~Bardos, G.~Lebeau, and J.~Rauch}, {\em Sharp sufficient conditions for
  the observation, control, and stabilization of waves from the boundary}, SIAM
  J. Control Optim., 30 (1992), pp.~1024--1065.

\bibitem{MR3048598}
{\sc E.~Bernard and F.~Salvarani}, {\em On the convergence to equilibrium for
  degenerate transport problems}, Arch. Ration. Mech. Anal., 208 (2013),
  pp.~977--984.

\bibitem{bernard-salvarani-2013-degenerate-linear-boltzmann}
{\sc {\'E}.~Bernard and F.~Salvarani}, {\em On the exponential decay to
  equilibrium of the degenerate linear {B}oltzmann equation}, Journal of
  Functional Analysis, 265 (2013), p.~1934–1954.

\bibitem{bernou-carrapatoso-mischler-tristani-2021-preprint-hypocoercivity-bounded-domain}
{\sc A.~Bernou, K.~Carrapatoso, S.~Mischler, and I.~Tristani}, {\em
  {Hypocoercivity for kinetic linear equations in bounded domains with general
  Maxwell boundary condition}}, 2021.
\newblock Preprint arXiv:2102.07709v1.

\bibitem{bisi-canizo-lods-2015-entropy-boltzmann}
{\sc M.~Bisi, J.~A. Ca{\~n}izo, and B.~Lods}, {\em Entropy dissipation
  estimates for the linear {Boltzmann} operator}, J. Funct. Anal., 269 (2015),
  pp.~1028--1069.

\bibitem{MR553920}
{\sc M.~E. Bogovski\u{\i}}, {\em Solution of the first boundary value problem
  for an equation of continuity of an incompressible medium}, Dokl. Akad. Nauk
  SSSR, 248 (1979), pp.~1037--1040.

\bibitem{MR631691}
\leavevmode\vrule height 2pt depth -1.6pt width 23pt, {\em Solutions of some
  problems of vector analysis, associated with the operators {${\rm div}$} and
  {${\rm grad}$}}, in Theory of cubature formulas and the application of
  functional analysis to problems of mathematical physics, vol.~1980 of Trudy
  Sem. S. L. Soboleva, No. 1, Akad. Nauk SSSR Sibirsk. Otdel., Inst. Mat.,
  Novosibirsk, 1980, pp.~5--40, 149.

\bibitem{borchers-sohr-1990-rotation-divergence-equation}
{\sc W.~Borchers and H.~Sohr}, {\em On the equations rot v=g and div u=f with
  zero boundary conditions}, Hokkaido Mathematical Journal, 19 (1990),
  p.~67–87.

\bibitem{bourgain-brezis-2002-divergence-equation}
{\sc J.~Bourgain and H.~Brezis}, {\em On the equation $\operatorname{div}y=f$
  and application to control of phases}, Journal of the American Mathematical
  Society, 16 (2002), p.~393–427.

\bibitem{MR3562318}
{\sc M.~Briant and Y.~Guo}, {\em Asymptotic stability of the {B}oltzmann
  equation with {M}axwell boundary conditions}, J. Differential Equations, 261
  (2016), pp.~7000--7079.

\bibitem{brigati-2021-time-fokker-planck}
{\sc G.~Brigati}, {\em {Time averages for kinetic Fokker-Planck equations}},
  2021.
\newblock Preprint arXiv:2106.12801v2.

\bibitem{MR84633}
{\sc A.~P. Calder\'{o}n and A.~Zygmund}, {\em On singular integrals}, Amer. J.
  Math., 78 (1956), pp.~289--309.

\bibitem{canizo-einav-lods-2018-boltzmann}
{\sc J.~A. Ca{\~n}izo, A.~Einav, and B.~Lods}, {\em On the rate of convergence
  to equilibrium for the linear {Boltzmann} equation with soft potentials}, J.
  Math. Anal. Appl., 462 (2018), pp.~801--839.

\bibitem{cao-lu-wang-2019-l-langevin}
{\sc Y.~Cao, J.~Lu, and L.~Wang}, {\em On explicit $l^2$-convergence rate
  estimate for underdamped langevin dynamics}, 2019.
\newblock To appear in {\em Annals of Applied Probability}, preprint
  arXiv:1908.04746.

\bibitem{carrapatoso2020weighted}
{\sc K.~Carrapatoso, J.~Dolbeault, F.~Hérau, S.~Mischler, and C.~Mouhot}, {\em
  Weighted {K}orn and {P}oincar\'e-{K}orn inequalities in the {E}uclidean space
  and associated operators}, 2020.
\newblock Preprint arXiv:2012.06347.

\bibitem{MR432101}
{\sc C.~Cercignani and M.~Lampis}, {\em Kinetic models for gas-surface
  interactions}, Transport Theory Statist. Phys., 1 (1971), pp.~101--114.

\bibitem{cessenat-1984-theoremes}
{\sc M.~Cessenat}, {\em Th{\'e}or{\`e}mes de trace {{\(L^ p\)}} pour des
  espaces de fonctions de la neutronique. {{\((L^ p\)}} trace theorems for
  neutronic functions spaces)}, C. R. Acad. Sci., Paris, S{\'e}r. I, 299
  (1984), pp.~831--834.

\bibitem{cessenat-1985-theoremes}
\leavevmode\vrule height 2pt depth -1.6pt width 23pt, {\em Th{\'e}or{\`e}mes de
  trace pour des espaces de fonctions de la neutronique. ({Trace} theorems for
  neutronic function spaces)}, C. R. Acad. Sci., Paris, S{\'e}r. I, 300 (1985),
  pp.~89--92.

\bibitem{cheeger-1970-laplacian}
{\sc J.~Cheeger}, {\em A lower bound for the smallest eigenvalue of the
  {Laplacian}}.
\newblock Probl. {Analysis}, {Sympos}. in {Honor} of {Salomon} {Bochner},
  {Princeton} {Univ}. 1969, 195-199 (1970)., 1970.

\bibitem{chung-2010-four-cheeger}
{\sc F.~Chung}, {\em Four proofs for the {Cheeger} inequality and graph
  partition algorithms}, in Fourth international congress of Chinese
  mathematicians. Proceedings of the ICCM '07, Hangzhou, China, December
  17--22, 2007, Providence, RI: American Mathematical Society (AMS);
  Somerville, MA: International Press, 2010, pp.~331--349.

\bibitem{danchin-mucha-2013-divergence}
{\sc R.~Danchin and P.~B.~Mucha}, {\em Divergence}, Discrete \& Continuous
  Dynamical Systems - S, 6 (2013), p.~1163–1172.

\bibitem{de-mondino-2021-sharp-cheeger-k}
{\sc N.~De~Ponti and A.~Mondino}, {\em Sharp {Cheeger}-buser type inequalities
  in {{\(\mathsf{RCD}(K,\infty)\)}} spaces}, J. Geom. Anal., 31 (2021),
  pp.~2416--2438.

\bibitem{MR1932965}
{\sc L.~Desvillettes and C.~Villani}, {\em On a variant of {K}orn's inequality
  arising in statistical mechanics}, ESAIM Control Optim. Calc. Var., 8 (2002),
  pp.~603--619 (electronic).
\newblock A tribute to J. L. Lions.

\bibitem{deuschel-stroock-1989-large}
{\sc J.-D. Deuschel and D.~W. Stroock}, {\em Large deviations}, Boston, MA
  etc.: Academic Press, Inc., 1989.

\bibitem{dietert-thesis-2017}
{\sc H.~Dietert}, {\em Contributions to mixing and hypocoercivity in kinetic
  models}, PhD thesis, University of Cambridge, 2017.

\bibitem{dietert-herau-hutridurga-mouhot-2022-trajectorial}
{\sc H.~Dietert, F.~H\'erau, H.~Hutridurga, and C.~Mouhot}, {\em Trajectorial
  hypocoercivity and application to control theory}, S\'eminaire Laurent
  Schwartz {\textemdash} EDP et applications,  (2021-2022).
\newblock talk:8.

\bibitem{dolbeault-mouhot-schmeiser-2015-hypocoercivity}
{\sc J.~{Dolbeault}, C.~{Mouhot}, and C.~{Schmeiser}}, {\em {Hypocoercivity for
  linear kinetic equations conserving mass}}, {Trans. Am. Math. Soc.}, 367
  (2015), pp.~3807--3828.

\bibitem{dolbeault-volzone-2012-improved-poincare}
{\sc J.~{Dolbeault} and B.~{Volzone}}, {\em {Improved Poincar\'e
  inequalities}}, {Nonlinear Anal., Theory Methods Appl., Ser. A, Theory
  Methods}, 75 (2012), pp.~5985--6001.

\bibitem{Duan_2011}
{\sc R.~Duan}, {\em Hypocoercivity of linear degenerately dissipative kinetic
  equations}, Nonlinearity, 24 (2011), pp.~2165--2189.

\bibitem{duran-2012-bogovskii}
{\sc R.~G. {Dur\'an}}, {\em {An elementary proof of the continuity from
  \(L^{2}_{0}(\Omega )\) to \(H^{1}_{0}(\Omega )^{n}\) of Bogovskii's right
  inverse of the divergence}}, {Rev. Uni\'on Mat. Argent.}, 53 (2012),
  pp.~59--78.

\bibitem{duran-garcia-2010-divergence-inequality}
{\sc R.~G. Dur{\'a}n and F.~L{\'o}pez~Garc{\'{\i}}a}, {\em Solutions of the
  divergence and {Korn} inequalities on domains with an external cusp}, Ann.
  Acad. Sci. Fenn., Math., 35 (2010), pp.~421--438.

\bibitem{duran-muschietti-2001-right-inverse-divergence}
{\sc R.~G. Durán and M.~A. Muschietti}, {\em An explicit right inverse of the
  divergence operator which is continuous in weighted norms}, Studia
  Mathematica, 148 (2001), p.~207–219.

\bibitem{MR0521262}
{\sc G.~Duvaut and J.-L. Lions}, {\em Inequalities in mechanics and physics},
  Springer-Verlag, Berlin-New York, 1976.
\newblock Translated from the French by C. W. John, Grundlehren der
  Mathematischen Wissenschaften, 219.

\bibitem{evans-moyano-2019-preprint-quantitative-rates-convergence-equilibrium-degenerate}
{\sc J.~Evans and I.~Moyano}, {\em Quantitative rates of convergence to
  equilibrium for the degenerate linear boltzmann equation on the torus},
  arxiv.1907.12836, 2019.

\bibitem{galdi-2011-navier-stokes}
{\sc G.~P. {Galdi}}, {\em {An introduction to the mathematical theory of the
  Navier-Stokes equations. Steady-state problems}}, New York, NY: Springer,
  2011.

\bibitem{geissert-heck-hieber-2006-bogoskii}
{\sc M.~{Gei{\ss}ert}, H.~{Heck}, and M.~{Hieber}}, {\em {On the equation
  \(\text{div}\,u=g\) and Bogovski\u{\i}'s operator in Sobolev spaces of
  negative order}}, in Partial differential equations and functional analysis.
  The Philippe Clément Festschrift. Based on the workshop, Delft, Netherlands,
  November 29--December 1, 2004 dedicated to Philippe Clément on the occasion
  of his retirement in December 2004., Basel: Birkh\"auser, 2006, pp.~113--121.

\bibitem{geissert-heck-hieber-sawada-2010-stokes-unbounded}
{\sc M.~{Geissert}, H.~{Heck}, M.~{Hieber}, and O.~{Sawada}}, {\em {Remarks on
  the \(L^p\)-approach to the Stokes equation on unbounded domains}}, {Discrete
  Contin. Dyn. Syst., Ser. S}, 3 (2010), pp.~291--297.

\bibitem{MR1908664}
{\sc Y.~Guo}, {\em The {V}lasov-{P}oisson-{B}oltzmann system near
  {M}axwellians}, Comm. Pure Appl. Math., 55 (2002), pp.~1104--1135.

\bibitem{MR2679358}
\leavevmode\vrule height 2pt depth -1.6pt width 23pt, {\em Decay and continuity
  of the {B}oltzmann equation in bounded domains}, Arch. Ration. Mech. Anal.,
  197 (2010), pp.~713--809.

\bibitem{han-kwan-leautaud-2015-geometric}
{\sc D.~Han-Kwan and M.~Léautaud}, {\em Geometric analysis of the linear
  {B}oltzmann equation {I}. {T}rend to equilibrium}, Annals of PDE, 1 (2015).

\bibitem{hieber-saal-2018-stokes-equation}
{\sc M.~Hieber and J.~Saal}, {\em The {S}tokes equation in the ${L}^p$-setting:
  Well-posedness and regularity properties}, Handbook of Mathematical Analysis
  in Mechanics of Viscous Fluids,  (2018), p.~117–206.

\bibitem{MR1368384}
{\sc C.~O. Horgan}, {\em Korn's inequalities and their applications in
  continuum mechanics}, SIAM Rev., 37 (1995), pp.~491--511.

\bibitem{hoermander-1967-hypoelliptic}
{\sc L.~{H\"ormander}}, {\em {Hypoelliptic second order differential
  equations}}, {Acta Math.}, 119 (1967), pp.~147--171.

\bibitem{korn-1906-abhandlungen-elastizitaetstheorie-ii}
{\sc A.~Korn}, {\em Abhandlungen zur {Elastizit{\"a}tstheorie} {II}. {Die}
  {Eigenschwingungen} eines elastischen {K{\"o}rpers} mit ruhender
  {Oberfl{\"a}che}.}
\newblock M{\"u}nch. {Ber}. 36, 351-401 (1906)., 1906.

\bibitem{korn-1908-solution}
\leavevmode\vrule height 2pt depth -1.6pt width 23pt, {\em Solution
  g{\'e}n{\'e}rale du probl{\`e}me d'{\'e}quilibre dans la th{\'e}orie de
  l'{\'e}lasticit{\'e} dans le cas o{\`u} les efforts donn{\'e}s {\`a} la
  surface.}, Toulouse Ann. (2), 10 (1908), pp.~165--269.

\bibitem{korn-1909-ueber-ungleichungen-theorie-schwingungen-rolle}
\leavevmode\vrule height 2pt depth -1.6pt width 23pt, {\em {\"U}ber einige
  {Ungleichungen}, welche in der {Theorie} der elastischen und elektrischen
  {Schwingungen} eine {Rolle} spielen.}
\newblock Krak. {Anz}., 705-724 (1909)., 1909.

\bibitem{lee-gharan-trevisan-2014-multiway-cheeger}
{\sc J.~R. Lee, S.~O. Gharan, and L.~Trevisan}, {\em Multiway spectral
  partitioning and higher-order {Cheeger} inequalities}, J. ACM, 61 (2014),
  p.~30.
\newblock Id/No 37.

\bibitem{lods-2005-semigroup}
{\sc B.~Lods}, {\em Semigroup generation properties of streaming operators with
  noncontractive boundary conditions}, Math. Comput. Modelling, 42 (2005),
  pp.~1441--1462.

\bibitem{lods-mokhtar-kharroubi-2017-convergence-boltzmann}
{\sc B.~Lods and M.~Mokhtar-Kharroubi}, {\em Convergence to equilibrium for
  linear spatially homogeneous {Boltzmann} equation with hard and soft
  potentials: a semigroup approach in {{\(L^1\)}}-spaces}, Math. Methods Appl.
  Sci., 40 (2017), pp.~6527--6555.

\bibitem{lods-mouhot-toscani-2008-relaxation-ficks-boltzmann}
{\sc B.~Lods, C.~Mouhot, and G.~Toscani}, {\em Relaxation rate, diffusion
  approximation and {Fick}'s law for inelastic scattering {Boltzmann} models},
  Kinet. Relat. Models, 1 (2008), pp.~223--248.

\bibitem{mischler-2000-trace}
{\sc S.~Mischler}, {\em On the trace problem for solutions of the {Vlasov}
  equation}, Commun. Partial Differ. Equations, 25 (2000), pp.~1415--1443.

\bibitem{mischler-2010-kinetic-maxwell}
\leavevmode\vrule height 2pt depth -1.6pt width 23pt, {\em Kinetic equations
  with {Maxwell} boundary conditions}, Ann. Sci. {\'E}c. Norm. Sup{\'e}r. (4),
  43 (2010), pp.~719--760.

\bibitem{mokhtar-kharroubi-2008-transport-semigroups}
{\sc M.~Mokhtar-Kharroubi}, {\em On collisionless transport semigroups with
  boundary operators of norm one}, J. Evol. Equ., 8 (2008), pp.~327--352.

\bibitem{MR0165337}
{\sc S.~L. Sobolev}, {\em Applications of functional analysis in mathematical
  physics}, Translated from the Russian by F. E. Browder. Translations of
  Mathematical Monographs, Vol. 7, American Mathematical Society, Providence,
  R.I., 1963.

\bibitem{spielman-2015-conductance}
{\sc D.~A. Spielman}, {\em Conductance, the normalized laplacian, and
  {Cheeger}'s inequality}.
\newblock Available at
  \url{https://www.cs.yale.edu/homes/spielman/561/lect06-15.pdf} (last access
  31 Aug 2022), 2015.

\bibitem{villani-2009-hypocoercivity}
{\sc C.~Villani}, {\em Hypocoercivity}, vol.~950 of Mem. Am. Math. Soc.,
  Providence, RI: American Mathematical Society (AMS), 2009.

\bibitem{voigt-1981-boundary-value}
{\sc J.~Voigt}, {\em Functional analytic treatment of the initial boundary
  value problem for collisionless gases}.
\newblock Habilitationsschrift, Ludwig-Maximilians-Universität München, 1981.

\end{thebibliography}

\end{document}